\documentclass[a4paper]{amsart}

\usepackage{amsfonts,amssymb,amscd,amsmath,latexsym,amsbsy,enumitem,stmaryrd,a4wide,verbatim,color,subfiles,multirow}
\usepackage{hyperref}
\usepackage{tikz}
\usetikzlibrary{calc,arrows,backgrounds,positioning,fit,shapes.misc,shapes.geometric}

\usetikzlibrary{decorations.markings}
\usepackage[foot]{amsaddr}

\theoremstyle{plain}
\newtheorem{theorem}{Theorem}[section]
\newtheorem{corollary}[theorem]{Corollary}

\newtheorem{lemma}[theorem]{Lemma}
\newtheorem{proposition}[theorem]{Proposition}
\newtheorem{Definition}[theorem]{Definition}
\theoremstyle{remark}
\newtheorem{remark}[theorem]{Remark}
\numberwithin{equation}{section}

\newcommand{\R}{\mathbb R}
\newcommand{\N}{\mathbb N}
\newcommand{\C}{\mathbb C}
\newcommand{\Z}{\mathbb Z}

\newcommand{\ASEP}[1]{\textnormal{ASEP}(#1)}
\newcommand{\RASEP}[1]{\textnormal{ASEP}_{\mathsf{R}}(#1)}
\newcommand{\LASEP}[1]{\textnormal{ASEP}_{\mathsf{L}}(#1)}
\newcommand{\SEP}[1]{\textnormal{SSEP}(#1)}
\newcommand{\RSEP}[1]{\textnormal{SSEP}_{\mathsf{R}}(#1)}
\newcommand{\LSEP}[1]{\textnormal{SSEP}_{\mathsf{L}}(#1)}
\newcommand{\al}{\alpha}
\newcommand{\be}{\beta}
\newcommand{\ga}{\gamma}

\newcommand{\genASEP}{L_{q,\vec{N}}}
\newcommand{\genRASEP}{L_{q,\vec{N},\rho}^{\mathsf R}}
\newcommand{\genLASEP}{L_{q,\vec{N},\lambda}^{\mathsf L}}
\newcommand{\de}{\delta}
\newcommand{\De}{\Delta}

\newcommand{\half}{\frac{1}{2}}

\newcommand{\la}{\lambda}

\newcommand{\om}{\omega}
\newcommand{\Om}{\Omega}
\newcommand{\pitensor}{\pi_{k,k+1}}
\newcommand{\pitensortwo}{\pi_{1,2}}
\newcommand{\tensor}{\otimes}
\newcommand{\rphis}[5]{\,_{#1}\varphi_{#2} \left( \genfrac{.}{.}{0pt}{}{#3}{#4}
\ ;#5 \right)}
\newcommand{\rphisempty}[2]{\,_{#1}\varphi_{#2}}
\newcommand{\rFs}[5]{\,_{#1}F_{#2} \left( \genfrac{.}{.}{0pt}{}{#3}{#4}	\ ;#5 \right)}

\newcommand{\su}{\mathfrak{su}}

\newcommand{\U}{\mathcal U}

\newcommand{\qbinom}[3]{\genfrac{[}{]}{0pt}{0}{#2}{#3}_{#1}}

\newcommand{\rev}{\textnormal{rev}}

\begin{document}
\title[A Generalized Dynamic Asymmetric Exclusion Process]{A Generalized Dynamic Asymmetric Exclusion Process: Orthogonal Dualities and Degenerations}
\author{Wolter Groenevelt} 
\author{Carel Wagenaar} 
\address{Delft Institute of Applied Mathematics\\
	Delft University of Technology, PO Box 5031, 2600 GA Delft, The Netherlands}
\email{W.G.M.Groenevelt@tudelft.nl, C.C.M.L.Wagenaar@tudelft.nl (corresponding author)}
\begin{abstract}
	In this paper, a generalized version of dynamic ASEP is introduced, and it is shown that the process has a Markov duality property with the same process on the reversed lattice. The duality functions are multivariate $q$-Racah polynomials, and the corresponding orthogonality measure is the reversible measure of the process. By taking limits in the generator of dynamic ASEP, its reversible measure, and the duality functions, we obtain orthogonal and triangular dualities for several other interacting particle systems. In this sense,  the duality of dynamic ASEP sits on top of a hierarchy of many dualities.\\
	\indent For the construction of the process, we rely on representation theory of the quantum algebra $\U_q(\mathfrak{sl}_2)$. In the standard representation, the generator of generalized ASEP can be constructed from the coproduct of the Casimir. After a suitable change of representation, we obtain the generator of dynamic ASEP. The corresponding intertwiner is constructed from $q$-Krawtchouk polynomials, which arise as eigenfunctions of twisted primitive elements. This gives a duality between dynamic ASEP and generalized ASEP with $q$-Krawtchouk polynomials as duality functions. Using this duality, we show the (almost) self-duality of dynamic ASEP. 
\end{abstract}
\maketitle

\section{Introduction}
In this paper, we study a continuous-time interacting particle process that we call generalized dynamic ASEP (asymmetric simple exclusion process). This process can be considered as a higher spin version of dynamic ASEP introduced by Borodin \cite{Bo}, as well as a dynamic version of generalized ASEP introduced by Carinci, Giardin\`a, Redig, and Sasomoto \cite{CGRS}. Here the term `dynamic' essentially means that the jump rates of the process depend on a height function corresponding to the particle process. We show that generalized dynamic ASEP is dual to the same process on the reversed lattice with orthogonal duality functions that can be expressed as multivariate $q$-Racah polynomials, i.e.~a multivariate Askey-Wilson polynomial on a discrete set, where the orthogonality is with respect to the reversible measure of the process. Several interesting interacting particle processes appear as limit cases of generalized dynamic ASEP. Taking limits to these processes leads to (orthogonal) dualities for several other interacting particle processes.

In the analysis of interacting particle processes Markov duality is a very useful property. Duality allows us to study a complicated system in terms of a simpler one. Furthermore, as a consequence of duality expectations of certain observables evolve according to specific systems of differential equations. For standard ASEP, in which one particle per site is allowed, self-duality was first obtained by Sch\"utz \cite{Sch}. For its generalized version from \cite{CGRS}, in which multiple particles are allowed on each site, self-duality is also obtained. For both processes duality played a role in showing they belong to the KPZ universality class \cite{BoCoSa}, \cite{CoShTs}. Also, dynamic ASEP has a duality property: in \cite{BoCo} it is shown to be in duality with standard ASEP. Furthermore, there are other generalizations of standard ASEP with duality. E.g.~in \cite{Kuan2018}, \cite{KuanMultiSpecies}, Kuan obtained self-duality functions for a multi-species generalization of ASEP.

In recent years several symmetric interacting particle systems were shown to have products of hypergeometric orthogonal polynomials as duality functions. For example, Franceschini and Giardin\`a \cite{FrGi2019} showed that Krawtchouk polynomials arise as self-duality functions for the generalized symmetric exclusion process. Other orthogonal polynomials, such as Meixner polynomials, Laguerre polynomials, and Hermite polynomials also appear as (self-)duality functions \cite{CarFraGiaGroRed}, \cite{FraGiaGro}, \cite{Gr2019}, \cite{ReSau2018}, \cite{Zh}. The advantage of such orthogonal dualities is that they form an orthogonal basis for the underlying Hilbert space, which greatly simplifies the expansion of observables in terms of the duality function. This was used in \cite{AyaCarRed1}, \cite{AyaCarRed2} to study Bolzmann-Gibbs principles and higher order fluctuation fields, and in \cite{FlReSau} in the study of $n$-point correlation functions in non-equilibrium steady states. 

Very recently orthogonal dualities were also obtained for asymmetric particle processes. In \cite{CFG} Carinci, Franceschini, and the first author show that certain $q$-Krawtchouck and $q$-Meixner polynomials, which are $q$-hypergeometric orthogonal polynomials. appear as duality functions for generalized ASEP and ASIP (asymmetric simple inclusion process). Because of the asymmetry, the products of the polynomials have a nested structure, which links them to the multivariate orthogonal polynomials of Tratnik-type from \cite{GasRahMulti}. The results from \cite{FraGiaGro} were extended to multi-species versions of generalized ASEP in \cite{BlBuKuLiUsZh},\cite{FraKuaZho}, where also nested products of $q$-Krawtchouk polynomials appear as duality functions. 

In the present paper, we show that generalized dynamic ASEP has nested products of $q$-Racah polynomials (also called Askey-Wilson polynomials) as duality functions. These polynomials are generalizations of $q$-Krawtchouck polynomials, thus we obtain the duality results from \cite{FraGiaGro} for generalized ASEP, as well as several other dualities, as special cases. The $q$-Racah polynomials form the top level of the discrete part of the $q$-Askey-scheme of $q$-hypergeometric orthogonal polynomials \cite{KLS}, which is a large scheme that relates families of orthogonal polynomials through specializations and limit relations. There are no families of orthogonal polynomials above the $q$-Racah polynomials in the $q$-Askey-scheme, therefore we consider it unlikely that more general orthogonal polynomial dualities will be found for one-species exclusion processes. As another interpretation for Askey-Wilson polynomials in relation to ASEP, we mention that they appear as correlation functions for standard ASEP with open boundaries \cite{UchSasWad}. 

For the construction of duality functions, we use representation theory of the quantum algebra $\U_q(\mathfrak{sl}_2)$. In \cite{CGRS} the generator of generalized ASEP is constructed from the Casimir element of $\U_q(\mathfrak{sl}_2)$. Through a change of representations (of a subalgebra) we obtain the generator for generalized dynamic ASEP from the Casimir, thus also obtaining a duality function between the two processes from the corresponding intertwiner. This duality result can then be extended to (almost) self-duality of generalized dynamic ASEP. Although we make extensive use of the algebra $\U_q(\mathfrak{sl}_2)$, we will first define the process and state the Markov duality (and corresponding duality functions) without reference to the algebraic construction, so that the statement of the main results of the paper requires no background knowledge on quantum algebras and representation theory.\\

Let us now describe the main character of the paper, generalized dynamic ASEP, in somewhat more detail. The process is a Markov jump process that lives on a 1-dimensional finite lattice. The sites are numbered from 1, the leftmost site, to $M$, the rightmost site. Particles can only jump to neighboring sites. For each site $k$, we let $N_k\in\N$ denote the maximum number of particles allowed on that site and let $\vec{N}$ be the vector containing these $N_k$. Then, for a scaling parameter $q>0$, we propose two closely related versions of generalized dynamic ASEP. Namely a `right' version $\RASEP{q,\vec{N},\rho}$, and a `left' version $\LASEP{q,\vec{N},\lambda}$. The first has jump rates consisting of the rates of generalized ASEP, denoted by $\ASEP{q,\vec{N}}$, times a factor which depends on the particles and free spaces on the right of a site and a right boundary value $\rho\in\R$ via a height function `$h^+$'. The jump rates of the second are a product of the $\ASEP{q,\vec{N}}$ rates and a factor depending on a height function `$h^-$', which depends on the particles and free spaces on the left of a site and a left boundary value $\lambda\in\R$. The two different versions of generalized dynamic ASEP can be obtained from each other by reversing the order of the sites. The rates of generalized dynamic ASEP are invariant under the transformation $q\to q^{-1}$, so $q$ cannot simply be seen as an asymmetry parameter. Without loss of generality, we can assume $q\in (0,1)$. Then, if the height function for a site is very negative, the rates of $\RASEP{q,\vec{N},\rho}$ will be close to the ones of $\ASEP{q,\vec{N}}$, while a very positive height function causes the rates to be close to the ones of $\ASEP{q^{-1},\vec{N}}$. So in some sense, the parameter $q$ influences the local asymmetry. Moreover, we show that if there is only one particle in $\ASEP{q,\vec{N},\rho}$, it gets pulled towards the site(s) where the height function is close to zero. Numerical simulations suggest a similar behaviour in the case of many particles, i.e. particles in generalized dynamic ASEP are distributed around the region where the height function is close to zero, which implies that these sites are filled by half of their capacity.  \\

Both versions of generalized dynamic ASEP are dual to $\ASEP{q,\vec{N}}$, which we can use to express the expectation of the particle current in generalized dynamic ASEP in terms of the expectation of one dual particle in $\ASEP{q,\vec{N}}$. Since both versions of generalized dynamic ASEP are dual to $\ASEP{q,\vec{N}}$, they are dual to each other. Also, $\RASEP{q,\vec{N},\rho}$ generalizes both $\ASEP{q,\vec{N}}$ and $\ASEP{q^{-1},\vec{N}}$; the same is true for the left version $\LASEP{q,\vec{N},\lambda}$. Therefore, the Markov duality between the processes $\RASEP{q,\vec{N},\rho}$ and $\LASEP{q,\vec{N},\lambda}$ sits on top of a hierarchy of several other Markov dualities, see Figure \ref{fig:HierarchyMarkovdualities} for some of these cases. 

\begin{figure}[h] 
	\begin{tikzpicture}[scale=0.75, every node/.style={transform shape}]
		\tikzset{
			box/.style={rounded rectangle, minimum width=58mm, very thick,draw=black!50, top color=white,bottom color=black!20},
			myarrow/.style={color= black!40!white,ultra thick,->},
			myarrowlr/.style={color= black!40!white,ultra thick,<->},
			mylabel/.style={color=black, fill=white}
		}
		\node (dynASEP) [box] {	\textbf{1.} $\LASEP{q,\vec{N},\lambda} \leftrightarrow \RASEP{q,\vec{N},\rho}$};
		\node (dummy) [below=1.5cm of dynASEP] {};
		\node (ASEP-1-RASEP) [box,left=of dummy] {	\textbf{2a.} $\ASEP{q^{-1},\vec{N}} \leftrightarrow \RASEP{q,\vec{N},\rho}$};
		\node (ASEP-RASEP) [box,right=of dummy] {\textbf{2b.} $\ASEP{q,\vec{N}} \leftrightarrow \RASEP{q,\vec{N},\rho}$};
		\node (ASEP-ASEP-1) [box,below=1.5 of dummy] {	\textbf{3b.} $\ASEP{q^{-1},\vec{N}} \leftrightarrow \ASEP{q,\vec{N}}$};
		\node (ASEP-ASEP) [box,right=of ASEP-ASEP-1] {	\textbf{3c.} $\ASEP{q,\vec{N}}\leftrightarrow \ASEP{q,\vec{N}}$};
		\node (ASEP-1-ASEP-1) [box,left=of ASEP-ASEP-1] {	\textbf{3a.} $\ASEP{q^{-1},\vec{N}} \leftrightarrow \ASEP{q^{-1},\vec{N}}$};

		\draw [myarrow](dynASEP.200) -- (ASEP-1-RASEP)
		node[mylabel,midway]{$\lambda \to -\infty$};
		\draw [myarrow](dynASEP.340) -- (ASEP-RASEP)
		node[mylabel,midway]{$\lambda \to \infty$};
		\draw [myarrow](ASEP-1-RASEP) -- (ASEP-ASEP-1)
		node[mylabel,midway]{$\rho \to -\infty$\,};
		\draw [myarrow](ASEP-RASEP) -- (ASEP-ASEP)
		node[mylabel,midway]{$\rho \to -\infty$};
		\draw [myarrow](ASEP-RASEP.200) -- (ASEP-ASEP-1)
		node[mylabel,midway]{$\rho \to \infty$};
		\draw [myarrow](ASEP-1-RASEP) -- (ASEP-1-ASEP-1)
		node[mylabel,midway]{$\rho \to \infty$};
	\end{tikzpicture}\\
	\caption{Hierarchy of Markov dualities, where $q\in(0,1)$.}\label{fig:HierarchyMarkovdualities}
\end{figure}
The duality functions for duality 1 on top of Figure \ref{fig:HierarchyMarkovdualities} are given by a (nested) product of $q$-Racah (or Askey-Wilson) polynomials. As mentioned above, these are on top of the $q$-Askey-scheme \cite{KLS}, which means that many other orthogonal polynomials, such $q$-Krawtchouk polynomials, are special cases of these $q$-Racah polynomials. Taking limits in the parameters of the particle processes often preserves the duality. These limits correspond to certain limits in the $q$-Askey-scheme, therefore duality functions lower in the hierarchy in Figure \ref{fig:HierarchyMarkovdualities} are still orthogonal polynomials. On the other hand, limits in the $q$-Askey-scheme do not always correspond to useful limits of Markov generators. For each of the dualities 1, 2a, 3b, and 3c in Figure \ref{fig:HierarchyMarkovdualities} we list the corresponding type of duality function in Table \ref{tab:dualityfunctions}, where we make a distinction between duality functions with and without a free parameter. The latter is a parameter, independent of both Markov processes, that appears in a non-trivial way in the duality function. The dualities 2b and 3a can be obtained from other dualities by sending $q\to q^{-1}$. All duality functions given in the table are orthogonal polynomials with respect to the reversible measures of the processes, except the triangular ones. The term `triangular' means here that the duality functions when written in matrix form with respect to a certain basis, are lower triangular.\\
\begin{table}[h]
	\begin{tabular}{|l|ll|}
		\hline
		\multirow{2}{*}{Duality} & \multicolumn{2}{l|}{Type of duality function}                         \\  \cline{2-3} 
		& \multicolumn{1}{l|}{Free parameter} &No free parameter      \\ \hline
		\textbf{1}                     & \multicolumn{1}{l|}{$q$-Racah}            &    Special case $q$-Racah     \\ \hline
		\textbf{2a}                     & \multicolumn{1}{l|}{$q$-Hahn}            & $q$-Krawtchouk         \\ \hline
		\textbf{3b}                       & \multicolumn{1}{l|}{Affine $q$-Krawtchouk}      & Triangular  \\ \hline
		\textbf{3c}                       & \multicolumn{1}{l|}{Quantum $q$-Krawtchouk}      & Triangular \\ \hline
	\end{tabular}\vspace{0.1cm}\\
	\caption{Duality functions corresponding to Figure \ref{fig:HierarchyMarkovdualities}}\label{tab:dualityfunctions}.\vspace{-0.8cm}
\end{table}

The $q$-Askey-scheme has a $q\to 1$ counterpart consisting of hypergeometric orthogonal polynomials. Considering corresponding limits in the particle processes will lead to orthogonal dualities for symmetric (dynamic) exclusion processes.\\

\medskip

\subsection{Outlook}
In this paper, we mainly study dualities of the new generalized dynamic ASEP and degenerate cases. We intend to investigate the new processes in more detail in a future paper. It would be interesting to consider stochastic PDE limits of generalized dynamic ASEP, as e.g. is done in \cite{CoGhMa}. Furthermore, it will be intriguing to investigate the relation between generalized dynamic ASEP, introduced in this paper, and the recently introduced \cite{KuaZho} higher spin versions of dynamic stochastic vertex models. Moreover, the standard ASEP with open boundaries, i.e. where particles can enter and leave the system at sites $1$ and $M$, has been studied recently \cite{BaCo,Oh,Sch2}. One could also consider generalized dynamic ASEP with open boundaries, where particles leaving and entering the system could have a global effect on the height function. Another extension might be to consider generalized dynamic ASEP on a ring, i.e. where sites $1$ and $M$ are connected. Given that particles in generalized dynamic ASEP can have a preference to move towards regions with a higher particle density, it would be interesting to investigate whether uphill diffusion might appear when considering generalized dynamic ASEP with open boundaries or on a ring. Uphill diffusion is a phenomenon where there is a particle current flowing from a lower particle density towards a higher one, see e.g. \cite{CirCol,ColGibVer}. Finally, we plan to report on a corresponding inclusion process, dynamic ASIP, and dualities in the very near future. This corresponds to representation theory of the non-compact quantum algebra $\U_q(\su(1,1))$, which was already shown to have a connection with ASIP \cite{CFG}, \cite{CarGiaRedSas}. By taking a suitable limit from dynamic ASIP, one can also obtain a dynamic version of the Asymmetric Brownian Energy Process (ABEP).\\

\subsection{Outline of the paper}
The organization of this paper is as follows. In Section \ref{sec:dynamic ASEP} we briefly recall and discuss dynamic ASEP defined in \cite{Bo}. Then in Section \ref{sec:GeneralizedDynASEP} we introduce two closely related higher spin versions of dynamic ASEP, show their reversibility, and state Markov dualities between these processes and generalized ASEP. The dualities are 1, as well as 2a and 2b (the last two without a free parameter) from Figure \ref{fig:HierarchyMarkovdualities}. The proof of these results is postponed until Sections \ref{sec:ConstructionDynASEP} and \ref{sec:LeftDynASEP}, but stating the results does not require those techniques. In Section \ref{sec:Degenerations} we will investigate degenerations of this duality by taking appropriate limits of the duality 1, showing all the dualities in Figure \ref{fig:HierarchyMarkovdualities} as well as dualities the for totally asymmetric zero range process (TAZRP) and (dynamic) symmetric exclusion processes. Up to this point, no knowledge of quantum algebras is required. In section \ref{sec:QuantumAlgebra-Krawtchouk} we introduce the quantum algebra $\U_q(\mathfrak{sl}_2)$ as well as the $q$-Krawtchouk polynomials, which are eigenfunctions of a realization of Koornwinder's twisted primitive elements. Then in Section \ref{sec:ConstructionDynASEP}, we construct the generator of generalized dynamic ASEP from generalized ASEP by a change of representation of the coproduct of the Casimir. The important observation here is that the generator of generalized ASEP acting on the degree of the $q$-Krawtchouk polynomials can be transferred to an action on its variable, giving the generator of generalized dynamic ASEP. This method of construction automatically gives Markov duality between the two processes (with $q$-Krawtchouk polynomials as orthogonal duality functions) as well as reversibility of generalized dynamic ASEP. In Section \ref{sec:LeftDynASEP} we then show that generalized dynamic ASEP is (almost) self-dual with duality functions given by a (doubly) nested product of $q$-Racah polynomials. In Section \ref{sec:calculations of limits} we carry out the explicit limit calculations to prove the results on degenerations of the $q$-Racah dualities which are stated in Sections \ref{sec:Degenerations} and \ref{sec:symmetric degenerations}. In the appendix, we give an overview of all duality functions appearing in this paper and their explicit description, and we state and prove several useful identities we make use of elsewhere in the paper.

\subsection{Preliminaries and notations}
Let us start with the definition of Markov duality. Let $\{X(t)\}_{t\geq0}$ and $\{\widehat{X}(t)\}_{t\geq0}$ be Markov processes with state spaces $\Omega$ and $\widehat{\Omega}$ and generators $L$ and $\widehat{L}$. We say that $X(t)$ and $\widehat{X}(t)$ are dual to each other with respect to a continuous duality function $D:\Omega\times\widehat{\Omega}\to \C$ if
\[		
	[L D(\cdot,\xi)](\eta)=[\widehat{L} D(\eta,\cdot)](\xi)
\]  
for all $\eta\in \Omega$ and $\xi\in\widehat{\Omega}$. If $\{\widehat{X}(t)\}_{t\geq0}$ is a copy of $\{X(t)\}_{t\geq0}$, we say that the process $\{X(t)\}_{t\geq0}$ is self-dual with respect to the duality function $D$.\\

For future reference, we also make the following remark.
\begin{remark}\label{rem:InvTotPart}
	Let $L$ and $\widehat{L}$ be generators of interacting particle systems where the total number of particles is conserved (as all processes in this paper will be). Then we have the following two basic results. 
	\begin{itemize}
		\item Let $D(\eta,\xi)$ be a duality function between the two processes. 
		If $f$ is a function only depending on parameters of the processes and the total number of (dual) particles $|\eta|$ and $|\xi|$, then $f(|\xi|,|\eta|)D(\eta,\xi)$ is again a duality function since $f$ is invariant under the action of both generators.
		\item Let $\mu$ be a reversible measure for the process generated by $L$, i.e.~ detailed balanced is satisfied:
		\[
		\mu(\eta)L(\eta,\eta') = \mu(\eta')L(\eta',\eta),
		\] 
		where $L(\eta,\eta')$ is the jump rate from the state $\eta$ to $\eta'$. Note that both sides of the above equation become zero if $|\eta|\neq |\eta'|$ since in that case $L(\eta,\eta')=0$. If $g$ is a function only depending on parameters of the process and the total number of particles $|\eta|$, then $g(|\eta|)\mu(\eta)$ is again a reversible measure since we can just multiply above detailed balance condition by $g(|\eta|)=g(|\eta'|)$.
	\end{itemize} 
\end{remark}
Let us introduce some notations and conventions we use throughout the paper. We fix a scaling parameter $q>0$, where we will sometimes require $q\in(0,1)$. By $\N$ we denote all positive integers,
\[
\Z_{\geq 0}=\N \cup \{0\} \qquad \text{and}  \qquad \R^\times = \R\backslash\{0\}.
\]
For $a\in\R$, let 
\[
[a]_q=\begin{cases}
	\begin{split}&\frac{q^a-q^{-a}}{q-q^{-1}} &\text{ for } q\neq 1,\\
	&a &\text{ for } q= 1,\end{split}
\end{cases}
\]
which is justified by
\begin{align*}
	\lim\limits_{q\to1}[a]_q=a . \label{eq:limshq}
\end{align*}
We use standard notation for $q$-shifted factorials and $q$-hypergeometric functions as in \cite{GR}. In particular, $q$-shifted factorials are given by
\[
(a;q)_n = (1-a)(1-aq) \cdots (1-aq^{n-1}), \qquad n \in \N 
\]
and we use the convention $(a;q)_0=1$. For $q \in (0,1)$, $(a;q)_\infty = \lim_{n \to \infty} (a;q)_n$. Moreover, for $q\neq1$, $n \in \Z_{\geq 0}$ and $k=0,\ldots,n$ the $q$-binomial coefficient is given by
\[
\qbinom{q}{n}{k} = \frac{	(q;q)_n}{	(q;q)_k	(q;q)_{n-k}} = \frac{(q^{-n};q)_k}{(q;q)_k} (-q^n)^k q^{-\frac12k(k-1)},
\]
The $q$-hypergeometric series $_{r+1}\varphi_r$ is given by
\[
\rphis{r+1}{r}{a_1,\ldots,a_{r+1}}{b_1,\ldots,b_r}{q,z} = \sum_{n=0}^\infty \frac{(a_1;q)_n \cdots (a_{r+1};q)_n}{(b_1;q)_n \cdots (b_r;q)_n} \frac{z^n}{(q;q)_n}.
\]
If for some $k$ we have $a_k=q^{-N}$ with $N \in \Z_{\geq 0}$, the series terminates after $N+1$ terms, since $(q^{-N};q)_n=0$ for $n>N$.

The shifted factorials are given by
\[
(a)_0=1, \qquad (a)_n = a(a+1)\cdots (a+n-1), \qquad n \in \N,  
\]
and the hypergeometric series $_{r+1}F_r$ is defined by
\[
\rFs{r+1}{r}{a_1,\ldots,a_{r+1}}{b_1,\ldots,b_r}{z} = \sum_{n=0}^\infty \frac{(a_1)_n \cdots (a_{r+1})_n}{(b_1)_n \cdots (b_r)_n}\frac{z^n}{n!},
\]
where the series terminates if $a_k \in -\N$ for some $k$. The following limit relations hold: the $q$-shifted factorials become shifted factorials,
\[
\lim_{q \to 1} \frac{ (a;q)_n }{(1-q)^n} = (a)_n, \qquad a \in \R, \ n \in \Z_{\geq 0},
\]
the $q$-binomial coefficient becomes the ordinary binomial coefficient,
\[
\qbinom{1}{n}{k}=\lim\limits_{q\to1} \qbinom{q}{n}{k} =\binom{n}{k}.
\]
and, for $a_1,\ldots,a_{r}, b_1,\ldots,b_r \in \R$ and $n \in \Z_{\geq 0}$,
\[
\lim_{q \to 1} \rphis{r+1}{r}{q^{-n},q^{a_1}, \ldots, q^{a_{r}}}{q^{b_1},\ldots,q^{b_r}}{q,z} = \rFs{r+1}{r}{-n,a_1,\ldots,a_{r}}{b_1,\ldots,b_r}{z}.
\]

For an ordered $M$-tuple $x=(x_1,\ldots,x_M)$ we denote by $x^\rev$ the reversed $M$-tuple,
\[
x^\rev = (x_M,\ldots,x_1).
\]
and by $|x|$ the sum over its elements,
\[
|x| = x_1+\ldots+x_M.
\]

All interacting particle processes and functions in this paper will depend on certain parameters. To simplify notation we suppress the dependence on the parameters in notations, but occasionally add one or more parameters in the notation to stress dependency on the included parameters. \\

\section{Dynamic ASEP} \label{sec:dynamic ASEP}
The dynamic asymmetric exclusion process is introduced in \cite{Bo} as a limit case of a stochastic Interaction-Round-a-Face model and has been further studied in \cite{BoCo,CoGhMa}. The dynamic ASEP is a continuous-time Markov process on the state space $\mathcal S = \{(h_k)_{k \in \Z} \mid h_k \in \Z, \ h_{k+1}-h_k=\pm 1 \}$. An element $(h_k)_{k\in\Z}$ in $\mathcal S$ can be considered as a height function on the real line that takes integer values at integers and has slope $\pm 1$ in between. The jumps of the process are independent and have exponential waiting times with rates depending on two parameters $q>0$ and $\alpha>0$:
\[
\begin{split}
	h_k &\mapsto h_k + 2 \quad  \text{at rate} \quad q^{-1} \frac{ 1+ \alpha q^{-2h_k} }{1+\alpha q^{-2h_k-2}}, \\
	h_k &\mapsto h_k -2 \quad \text{at rate} \quad q \frac{ 1+ \alpha q^{-2h_k} }{1+\alpha q^{-2h_k+2}}.
\end{split}
\]
Jumps that take a height function out of the state space are of course not allowed and therefore have rate equal to zero. Let us remark that the rates here are normalized slightly differently from \cite{Bo}. Furthermore, we can always set the parameter $\al$ equal to $1$ by shifting the height function by $\rho=\frac12\log_q(\al)$; in this case, $h_k \in \rho+\Z$ for all $k$ and the rates are invariant under the transformation $q\to q^{-1}$.

In the next section, we propose a higher spin generalization of dynamic ASEP. However, before we introduce the generalized process, it will be convenient to consider the following interacting particle process on $M\in\N$ sites which is equivalent to the above-described dynamic ASEP for height functions on $\Z \cap [1,M]$. \\
In \cite{BoCo}, a slope increment of $-1$ was associated with a particle and a slope increment of $1$ with an empty site. Note that if we would go in the opposite direction (from right to left), a slope increment of $1$ is associated with a particle and a slope increment of $-1$ with an empty site, which is in our setting a more convenient view of looking at it. To be precise, let $\xi_k$ be the number of particles on site $k$, which is either 0 or 1, then $\xi_k$ is determined by 
\[
	h_k=h_{k+1}+1\ \text{ if }\ \xi_k=1 \qquad \text{and} \qquad h_k=h_{k+1}-1\ \text{ if }\ \xi_k=0,
\]
or in shorter notation,
\begin{align*}
	h_{k}=h_{k+1}+(2\xi_k-1).
\end{align*}
This implies that the height function $(h_k)_{k=1}^M$ is a function of $\xi=(\xi_k)_{k=1}^M \in \{0,1\}^M$. We fix the height function to be equal to a real number $\rho$ at the `virtual' site $M+1$,
\begin{align}
	h_{M+1}=\rho, \label{eq:boundaryvalue}
\end{align}
which can be considered as a boundary value (on the right). Then for $k=1,\ldots,M+1$, the height function is given by
\[
h_k=h_{k,\rho}(\xi) =  \rho+\sum_{j=k}^M (2\xi_j-1),
\]
where we suppress the dependence on $\rho$ and/or $\xi$, unless we explicitly need it. We adopt the convention that the empty sum equals zero so that the boundary value \eqref{eq:boundaryvalue} holds. Let us mention that
\[
\sum_{j=k}^{M}(2\xi_j-1)=\sum_{j=k}^{M}\xi_j - \sum_{j=k}^{M}(1-\xi_j)
\] 
is the sum of the particles per site minus the sum of the free spaces per site. \\

We mention here that in a given configuration $\xi$, the height function on site $1$ is also constant and equal to $\rho + 2|\xi|-M$. So if one fixes a boundary $\rho'$ on a left virtual site $0$ (i.e.~ $h'_0=\rho'$), we obtain essentially the same interacting particle system if one takes  $h'_k=\rho' + \sum_{j=1}^k(1-2\xi_j)$. However, we choose the convention that in the height function, we sum over the particles per site minus the sum over the free spaces per site (instead of the free spaces per site minus the particles per site), which implies that we need to impose a boundary value on the right.\\

Alternatively, a similar (but different) process can be defined where in the height function we sum over the particles per site minus the free spaces per site on the \textit{left} of and including site $k$\footnote{We emphasize that this is a different height function than $h'_k$ defined in the previous paragraph.}, which we do in section \ref{sec:LeftDynASEP}. For this reason, we distinguish between those cases by adding a superscript $+$ (or $-$) to the height function when summing over particles on the right (or left). Thus we write for $\lambda,\rho\in\R$,
\[
h^+_k=h^+_{k,\rho}(\xi)=\rho+\sum_{j=k}^M (2\xi_j-1)
\] 
and
\[
h^-_k=h^-_{k,\lambda}(\xi)=\lambda+\sum_{j=1}^k (2\xi_j-1).
\]
For now, let us focus on the version with $h^+_k$. The jump rates for that corresponding particle process are now given by 
\[
\begin{split}
	\xi & \mapsto \xi^{k-1,k} \quad  \text{at rate} \quad q^{-1} \frac{ 1+ q^{-2h^+_k} }{1+q^{-2h^+_k-2}}, \\
	\xi &\mapsto \xi^{k,k-1} \quad \text{at rate} \quad q \frac{ 1+ q^{-2h^+_k} }{1+ q^{-2h^+_k+2}},
\end{split}
\]
where $\xi^{k-1,k} = (\ldots,\xi_{k-1}-1,\xi_{k}+1,\ldots )$ and $\xi^{k,k-1} = (\ldots,\xi_{k-1}+1,\xi_{k}-1,\ldots )$. 
\medskip

Let us assume $q \in (0,1)$, then in the limit $\rho \to -\infty$ dynamic ASEP becomes standard ASEP with jump rate $q^{-1}$ for jumps to the right and rate $q$ for jumps to the left. Moreover, in the limit $\rho \to \infty$ we obtain the same standard ASEP, but with $q$ replaced by $q^{-1}$.

\section{Generalized dynamic ASEP}\label{sec:GeneralizedDynASEP}
In this section, we introduce a higher spin version of dynamic ASEP and state several corresponding Markov dualities. The definition of the process is motivated by representation theory of the quantum algebra $\U_q(\mathfrak{sl}_2)$: the process generator is, up to an additive constant, the realization of sums of coproducts of the quantum Casimir element in a particular representation. However, in this section, all results are stated without reference to the representation theory. The proofs of the duality results stated in this section, which make use of this representation-theoretic interpretation of the generator, are postponed to later sections. 

\subsection{Generalized ASEP}
Before we introduce the higher spin version of dynamic ASEP let us first introduce $\ASEP{q,\vec{N}}$, which is a higher spin version of the standard ASEP, i.e.~each site allows a finite number of particles. This process was introduced in \cite{CGRS} in the case where each site allows the same number of particles. Here we define the process in a slightly more general way, namely the maximum number of particles may differ on each site. 

\medskip

For $k=1,\ldots,M$ let $N_k \in \N$ be the maximum number of particles allowed on site $k$, and denote $\vec{N}=(N_1,\ldots,N_M)$. The process $\ASEP{q,\vec{N}}$ is a continuous-time Markov jump process on the state space $X = \{0,\ldots,N_1\}\times \cdots \times \{0,\ldots,N_M\}$ depending on a parameter $q>0$. Given a state $\eta=(\eta_k)_{k=1}^M$, a particle on site $k$ jumps to site $k+1$ at rate 
\begin{equation} \label{eq:c^+_k}
	c_k^{+}(\eta)=q^{-(\eta_{k+1}+N_{k}-\eta_{k}+1)} [\eta_k]_q [N_{k+1}-\eta_{k+1}]_q,
\end{equation}
and a particle on site $k$ jumps to site $k-1$ at rate
\begin{equation} \label{eq:c^-_k}
	c^-_{k}(\eta) = q^{\eta_{k-1}+N_k-\eta_{k}+1} [\eta_{k}]_q[N_{k-1}-\eta_{k-1}]_q .
\end{equation}
The Markov generator of the process is then given by
\[
L_{q,\vec{N}}f(\eta) = \sum_{k=1}^{M-1} c^{+}_k(\eta) [f(\eta^{k,k+1})-f(\eta)] + c^-_{k+1}(\eta) [f(\eta^{k+1,k})-f(\eta)].
\]
\begin{remark}\* 
	\begin{itemize}
		\item By associating particles with free places and vice versa we get a symmetry between $\ASEP{q,\vec{N}}$ and $\ASEP{q^{-1},\vec{N}}$. That is, let $\{\eta(t)\}_{t\geq0}$ be the process that evolves according to $\ASEP{q,\vec{N}}$. If we define $\eta'=\vec{N}-\eta$, then $\{\eta'(t)\}_{t\geq0}$ evolves according to $\ASEP{q^{-1},\vec{N}}$. This symmetry can also be found in the duality functions involving those processes. 
		\item For $N_1=\ldots=N_M=2j$ with $j \in \frac12\N$, this becomes $\ASEP{q,2j}$ as defined in \cite{CGRS}. For \mbox{$N_1=\ldots=N_M=1$}, this is the standard ASEP where particles jump to the left with rate $q$ and to the right with $q^{-1}$.
	\end{itemize}
\end{remark}

\medskip

\subsection{Generalized dynamic ASEP}
Now we are ready to define a higher-spin version of dynamic ASEP. Similar to dynamic ASEP from the previous section, the rates can be written as a product of the rate of (generalized) ASEP and a factor containing the height function.
\begin{Definition} \label{Def:dynamic ASEP}
	$\RASEP{q,\vec{N},\rho}$ is a continuous-time Markov jump process on the state space $X$ depending on parameters $q>0$ and $\rho \in \R$. Given a state $\xi=(\xi_k)_{k=1}^M \in X$ we define the height function $(h^{+}_{k})_{k=1}^M$ by
	\[
	h^+_k= \rho + \sum_{j=k}^M (2\xi_j - N_j),
	\]
	and on the right we set the boundary value $h^+_{M+1} = \rho$. 
	Then a particle on site $k$ jumps to site $k+1$ at rate
	\[
	C^{\mathsf R,+}_k(\xi) = c^+_k(\xi) \frac{\left(1+q^{2\xi_k-2h^+_{k}}\right)\Big(1+q^{2\xi_{k+1}-2h^{+}_{k+1}}\Big)}{\Big(1+q^{-2h^{+}_{k+1}}\Big)\Big(1+q^{-2h^+_{k+1}-2}\Big)},
	\]	
	and a particle on site $k$ jumps to site $k-1$ at rate
	\[
	C^{\mathsf R,-}_{k}(\xi) = c^-_{k}(\xi) \frac{\Big(1+q^{-2\xi_{k-1}-2h^+_k}\Big)\Big(1+q^{-2\xi_{k}-2h^+_{k+1}}\Big)}{\Big(1+q^{-2h^+_k}\Big)\Big(1+q^{-2h^+_{k}+2}\Big)}.
	\]
\end{Definition} 
\begin{remark}
	Since 
	\[
	\sum_{j=k}^M (2\xi_j - N_j) = \sum_{j=k}^M \xi_j - \sum_{j=k}^M (N_j-\xi_j),
	\]
	this factor in the height function is again the sum of the particles per site minus the sum of the free spaces per site.
\end{remark}
\begin{remark}
	One can also rewrite the rates as
	\begin{align}
		\begin{split}
		C^{\mathsf R,+}_k(\xi) &= [\xi_k]_q[N_{k+1}-\xi_{k+1}]_q\frac{\big(q^{h_k^+-\xi_k}+q^{-(h_k^+-\xi_k)}\big)\big(q^{h_{k+1}^+-\xi_{k+1}}+q^{-(h_{k+1}^+-\xi_{k+1})}\big)}{\big(q^{h_{k+1}^+}+q^{-h_{k+1}^+}\big)\big(q^{h_{k+1}^++1}+q^{-(h_{k+1}^++1)}\big)},\\
		C^{\mathsf R,-}_{k}(\xi) &= [\xi_k]_q[N_{k-1}-\xi_{k-1}]_q\frac{\big(q^{h_k^++\xi_{k-1}}+q^{-(h_k^++\xi_{k-1})}\big)\big(q^{h_{k+1}^++\xi_{k}}+q^{-(h_{k+1}^++\xi_{k})}\big)}{\big(q^{h_{k}^+}+q^{-h_{k}^+}\big)\big(q^{h_{k}^+-1}+q^{-(h_{k}^+-1)}\big)}.
		\end{split}\label{eq:rewrittenrates}
	\end{align}
	From this we can see that the value of the parameter $q$ in $\RASEP{q,\vec{N},\rho}$ is not related to the asymmetry of the process as the jump rates $C^{\mathsf R,+}_k$ and $C^{\mathsf R,-}_{k}$ are invariant under $q \leftrightarrow q^{-1}$. In contrast, the rates $c^+_k$ and $c^-_{k}$ of $\ASEP{q,\vec{N}}$ are not $q\leftrightarrow q^{-1}$ invariant. 
\end{remark}
Note that, similar to dynamic ASEP in Section \ref{sec:dynamic ASEP}, this process can be written solely in terms of the height function $h^+_k$ by using
\[
h^+_{k}=h^+_{k+1}+2\xi_k - N_k.
\]
In this way, we can consider it as a process on the state space of height functions on $[1,M]\cap \Z$ taking values in $\rho+\Z$, for which the slope between $k$ and $k+1$ can take values in 
\[\{-N_k,-N_k+2,...,N_k-2,N_k\}.\]
The dynamic parameter $\rho$ can be considered as a boundary value on the right for the height function. The added label `R' stands for `right', indicating that there is a prescribed boundary value for the height function at the right boundary (at the virtual site $M+1$). Consequently, considered as an interacting particle process, the jump rates to and from site $k$ depend on the number of particles on the right of site $k$ through the values $h^+_k$ and $h^+_{k+1}$ of the height function. For later references we introduce the $\RASEP{q,\vec{N},\rho}$ Markov generator, which is given by
\[
L_{q,\vec{N},\rho}^{\mathsf R} f (\xi) = \sum_{k=1}^{M-1} C^{\mathsf R,+}_k(\xi) [f(\xi^{k,k+1})-f(\xi)] + C^{\mathsf R,-}_{k+1}(\xi) [f(\xi^{k+1,k})-f(\xi)].
\]

\begin{remark}[Special cases] \label{remark:special cases of ASEP_R}
	$\RASEP{q,\vec{N},\rho}$ reduces to other processes as follows:
	\begin{itemize}
		\item For $N_1=\ldots=N_L=1$ we recover standard dynamic ASEP from Section \ref{sec:dynamic ASEP}. Indeed, the height function $h^+_k$ changes to $h^+_k +2$ if a particle at site $k-1$ jumps to site $k$. This is only possible if $\xi_{k-1}=1$ and $\xi_{k}=0$, thus in that case
		\begin{align*}
			C^{\mathsf R,+}_{k-1}(\xi) = q^{-1} \frac{1+q^{-2h_{k}^+}}{1+q^{-2h_{k}^+-2}}.
		\end{align*}
		Similarly, we obtain that the height functions $h^+_k$ changes to $h^+_k-2$ with rate
		\begin{align*}
			C^{\mathsf R,-}_{k}(\xi) = q \frac{1+q^{-2h_{k}^+}}{1+q^{-2h_{k}^+ +2}}.
		\end{align*}
		\item Assume $q \in (0,1)$. In the limit $\rho \to -\infty$ $\RASEP{q,\vec{N},\rho}$ becomes $\ASEP{q,\vec{N}}$, which follows from
		\[
		\lim_{\rho \to -\infty} C^{\mathsf R,+}_k(\xi) = c^+_k(\xi) \quad \text{and} \quad \lim_{\rho \to -\infty} C^{\mathsf R,-}_{k}(\xi) = c^-_{k}(\xi).
		\]
		Moreover, in the limit $\rho \to \infty$ we have
		\[
		\lim_{\rho \to \infty} C^{\mathsf R,+}_k(\xi) = c^+_k(\xi;q^{-1}) \quad \text{and} \quad \lim_{\rho \to \infty} C^{\mathsf R,-}_{k}(\xi) = c^-_{k}(\xi;q^{-1}),
		\]
		so that $\RASEP{q,\vec{N},\rho}$ becomes $\ASEP{q^{-1},\vec{N}}$. In this sense $\RASEP{q,\vec{N},\rho}$ interpolates between $\ASEP{q,\vec{N}}$ and $\ASEP{q^{-1},\vec{N}}$. This is similar to the limits of dynamic ASEP from Section \ref{sec:dynamic ASEP} to standard ASEP.\\
	\end{itemize}
\end{remark}
\begin{remark}[Dynamic ASEP on $\Z$]
	Let $\vec{N}=(N_k)_{k\in\Z}$ with each $N_k\in\N$. Then we can also define $\RASEP{q,\vec{N},\rho}$ on $\Z$ instead of $M$ sites as long as there are a finite number of particles. The only non-trivial thing is how to define the height function $h^+_k$ when $k> M$. When having $M$ sites, we fixed the height function at the `virtual' site $M+1$ to be $\rho$. If we allow particles to jump further to the right, the height function at site $M+1$ is not fixed anymore. However, we can still use this site as a reference point. Every time a particle jumps from site $M$ to $M+1$, the height function at site $M+1$ is raised by $2$. Therefore, define
	\[
	h^+_{M+1}=\rho+2\sum_{j\geq M+1} \xi_j.
	\]
	Requiring $h^+_{k}=h^+_{k+1}+2\xi_k - N_k$ for all $k \in \Z$, we obtain the height function 
	\[
	h^+_k=\begin{cases}
		h^+_{M+1} + \sum_{j=k}^M (2\xi_j -N_j)\qquad &\text{if }k\leq {M+1}, \\[5pt]
		h^+_{M+1}- \sum_{j=M+1}^{k-1}(2\xi_j-N_j)\qquad &\text{if } k> M+1.
	\end{cases}
	\]
	In this setting, $M+1$ is just an arbitrary reference site, so we may for example fix $M+1=0$.
	
	The duality functions obtained later on are still valid for $\Z$ since sites with no particles for both dual processes do not contribute to the duality function. 
\end{remark}

	Let us take a closer look at the rates of $\RASEP{q,\vec{N},\rho}$. Since the rates are invariant under the transformation $q\to q^{-1}$, we can assume without loss of generality that $q\in(0,1)$. One can show that the rates $C^{\mathsf R,+}_k$ and $C^{\mathsf R,-}_k$, seen as functions of $h_{k}^+$, are monotonically decreasing and increasing respectively. Indeed, from \eqref{eq:rewrittenrates} one obtains,
	\[
		C^{\mathsf R,-}_k =  [\xi_k]_q[N_{k-1}-\xi_{k-1}]_qf_1(h_k^+)f_2(h_k^+),
	\]
	where
	\begin{align*}
		f_1(h_k^+)=\frac{\big(q^{h_k^++\xi_{k-1}}+q^{-(h_k^++\xi_{k-1})}\big)}{\big(q^{h_{k}^+}+q^{-h_{k}^+}\big)},\\
		f_2(h_k^+) = \frac{\big(q^{h_{k}^++N_k-\xi_{k}}+q^{-(h_{k}^++N_k-\xi_{k})}\big)}{\big(q^{h_{k}^+-1}+q^{-(h_{k}^+-1)}\big)}.
	\end{align*}
	Here we used $h_{k+1}^+(\xi)=h_k^+(\xi)+N_k-2\xi_k$. A straightforward calculation for the derivative with respect to $h_k^+$ shows that $f_1',f_2'\geq 0$. Since also $f_1,f_2\geq 0$, we conclude that $C^{\mathsf R,-}_k$ is monotonically increasing in $h_k^+$. Similarly, one can show that $C^{\mathsf R,+}_k$ is monotonically decreasing in $h_k^+$. As noted in Remark \ref{remark:special cases of ASEP_R}, $\RASEP{q,\vec{N},\rho}$ becomes $\ASEP{q,\vec{N}}$ in the limit $\rho\to-\infty$ and $\ASEP{q^{-1},\vec{N}}$ if we let $\rho\to\infty$. Therefore, the height function $h_k^+$ makes the rates of $\RASEP{q,\vec{N},\rho}$ interpolate between $\ASEP{q,\vec{N}}$, when $h_k^+$ goes to $-\infty$, and $\ASEP{q^{-1},\vec{N}}$, when $h_k^+$ goes to $\infty$. \\
	\\
	To understand the behaviour of the system better, let us consider the case where there is only one particle present in $\RASEP{q,\vec{N},\rho}$ and $N_k=N\in\N$ for all $k$. Then we have a nearest neighbor asymmetric random walker which jumps from site $k\to k+1$ with rate $C_k^{\mathsf{R},+}(e_k)$ and from site $k\to k-1$ with rate $C_k^{\mathsf{R},-}(e_k)$, where $e_k$ is the one-particle configuration with a particle at site $k$. Recall that both $h_1=\rho + 2|\xi|-|\vec{N}|$ and $h_{M+1}^+=\rho$ are fixed when $\rho,\vec{N}$ and the total number of particles in the system are chosen. In the setting with one particle, $h_{1}^+=\rho+2 -|\vec{N}|$. So besides the situation with $M=N=1$, we always have $h_1^+ \leq h_{M+1}^+$. The function $k\mapsto h_k^+$ is a straight line with increment $N$ between two neighboring sites, unless the particle is at site $k$, when the increment is $N-2$. Since we just saw that the rate $C^{\mathsf R,+}_k$ is monotonically decreasing in the variable $h_k^+$ and $C^{\mathsf R,-}_k$ monotonically increasing, the more the particle is on the right of the lattice $\{1,2,...,M\}$, the less it wants to jump to the right and the more it wants to jump to the left. Similarly, when our particle moves further to the left, the less it wants to jump to the left and the more it wants to jump to the right. \\
	\\
	Since $C_k^{\mathsf{R},-}$ is monotonically increasing and $C_k^{\mathsf{R},+}$ monotonically decreasing, it is a natural question to ask for which value of $h_k^+$ the rates are equal. Using \eqref{eq:rewrittenrates}, we see that this will be the case if and only if
	\begin{align}
		q^{h_k^+-1} + q^{-(h_k^+-1)} = q^{h_k^+ -1+ N}+q^{-(h_k^+ -1+ N)}. \label{eq:leftrate=rightrate}
	\end{align}
	Note that the function 
	\[
		x\mapsto \half(q^x + q^{-x})
	\]
	is just $\cosh(\ln(q)x)$, thus \eqref{eq:leftrate=rightrate} is satisfied if and only if
	\[
		h_k^+-1 = h_k^+ -1+ N\qquad \text{ or } \qquad	h_k^+-1 = -( h_k^+ -1+ N).
	\] 
	The first equation has no solution, the second leads to $h_k=1-\half N$, which can only be satisfied for certain values of $\rho$. However, if $1-\half N \geq \rho \leq (M-\half)N-2|\xi|+1$, then
	\[
		h_1^+=\rho + 2|\xi| -MN \leq 1-\half N
	\] 
	and $h_{M+1}^+ \geq 1-\half N$. Since the height function has a maximum increment\footnote{In the setting with $1$ particle, the increment is either $N$ or $N-2$.} of $\pm N$ between neighboring sites, there will be an integer $j$ such that $h_j^+(e_j) \leq 1-\half N$ and $h_{j+1}^+(e_{j+1}) \geq 1-\half N$. In that setting, the particle has a preference to jump to the right if present at site $j$, and to the left if present at site $j+1$. Moreover, note that $h_j^+$ is relatively close to $0$ compared to the value $h_{M+1}^+-h_1^+ = MN -2|\xi|$. Therefore, the sign of the preferred direction the particle wants to jump will flip in the region where the height function is close to zero. That is, the particle is attracted to this region. However, the existence of such a region depends on the value of $\rho$, so let us describe three different scenarios.
	\begin{itemize}
		\item $\rho \ll 0$.\\
		When $\rho$ is small enough such that $h_k^+ \leq 1-\half N \leq 0$ for all $k$, the particle has a preference to jump to the right on each site. Moreover, if it gets further away from site $M$, the rate for jumping to the right increases, and the rate for jumping to the left decreases. Note that this is consistent with the fact that the jump rates get closer to the ones of $\ASEP{q,\vec{N}}$ when the height function decreases.\\
		
		\item $1-\half N < \rho \leq (M-\half)N - 2|\xi| +1$. \\
		As discussed before, there will be an integer $j$ such that the particle has a preference to jump to the right if present at (or left of) site $j$ and a preference to move to the left if present at (or right of) site $j+1$. Also, if the particle gets further to the left, the rate for jumping right increases, and if the particle gets further to the right, the rate for jumping left increases. Moreover, if the particle moves further left or right of site $j$, the jump rates will approximate the rates of $\ASEP{q,\vec{N}}$ or $\ASEP{q^{-1},\vec{N}}$ respectively.\\
		
		\item $\rho \gg 0$.\\
		When $\rho$ is large enough such that $h_1^+ >1+\half N$,  then $h_k^+ \geq 1-\half N$ for all $k$. Therefore, the particle will have a preference to jump to the left on each site. Moreover, if it gets further right of site $1$, the rate for jumping left increases, and the rate for jumping right decreases. This is again consistent with the fact that the jump rates get closer to the jump rates of $\ASEP{q^{-1},\vec{N}}$ when the height function increases.
	\end{itemize}
	In Figure \ref{fig:3situationsrho}, one can see instances of these three situations where there is a particle at site $2$ and site $M-1$. Note that as long as the particles don't jump to neighboring sites, the jump of one particle does not influence the jump rate of the other particle. So if we assume that $M\geq 5$, the particles are not neighbors and the previous one-particle analysis still goes through for the states depicted. 
		\begin{figure}[h]
			\begin{tikzpicture}[scale=0.85]
				%L=10
				\draw (0,0) [dotted] -- (4.5,0);
				\foreach \i in {1,...,10}
				{
					\draw (.5*\i-.5,-.1) [thick]-- (.5*\i-.5,.1);
				}
				\draw (.5,.3) [fill=black] circle(.1);
				\draw (.3,.5) [->] -- (.7,.5);
				\draw (4,.3) [fill=black] circle(.1);
				\draw (3.8,.5) [->] -- (4.2,.5);
				\draw (4.5,-.75) [blue,thick,dashed] -- (4,-0.75) -- (1,-2.25) -- (0.5,-2.25) -- (0,-2.5);
				
				\begin{scope}[shift={(6,0)}]
					\draw (0,0) [dotted] -- (4.5,0);
					\foreach \i in {1,...,10}
					{
						\draw (.5*\i-.5,-.1) [thick]-- (.5*\i-.5,.1);
					}
					\draw (.5,.3) [fill=black] circle(.1);
					\draw (.3,.5) [->] -- (.7,.5);
					\draw (4,.3) [fill=black] circle(.1);
					\draw (3.8,.5) [<-] -- (4.2,.5);
					\draw (4.5,1) [blue,thick,dashed] -- (4,1) -- (1,-0.5) -- (.5,-0.5) -- (0,-0.75);
				\end{scope}	
				
				\begin{scope}[shift={(12,0)}]
					\draw (0,0) [dotted] -- (4.5,0);
					\foreach \i in {1,...,10}
					{
						\draw (.5*\i-.5,-.1) [thick]-- (.5*\i-.5,.1);
					}
					\draw (.5,.3) [fill=black] circle(.1);
					\draw (.3,.5) [<-] -- (.7,.5);
					\draw (4,.3) [fill=black] circle(.1);
					\draw (3.8,.5) [<-] -- (4.2,.5);
					\draw (4.5,2.75) [blue,thick,dashed] -- (4,2.75) -- (1,1.25) -- (.5,1.25) -- (0,1);
				\end{scope}
			\end{tikzpicture}
			
			\caption{Three different $\ASEP{q,\vec{N},\rho}$ instances with 2 particles (black dots) and the corresponding height function $h^+$ (dashed line), where $N_k=2$ for all $k$. From left to right: $\rho\ll 0$, $1-\half N \leq \rho\leq (M-\half)N - 2|\xi| +1$, $\rho\gg 0$. The arrow indicates the preferred jump direction.} \label{fig:3situationsrho}
		\end{figure}

\subsection{Duality between ASEP and dynamic ASEP}
A main result of this paper is (an orthogonal) Markov duality between $\ASEP{q,\vec{N}}$ and $\RASEP{q,\vec{N},\rho}$. The duality function is given in terms of (dual) $q$-Krawtchouk polynomials \cite[\S14.15 and \S14.17]{KLS} that we now introduce. Define for $c\in \R$ and $N \in \N$,
\[
K_n(x;c,N;q) = \rphis{3}{2}{q^{-n}, q^{-x}, -cq^{x-N} }{q^{-N}, 0 }{q,q}.
\]
For $n=0,\ldots,N$ these are polynomials in $q^{-x}-cq^{x-N}$ of degree $n$; as such these polynomials are called the dual $q$-Krawtchouk polynomials. Moreover, for $x=0,\ldots,N$ these are polynomials in $q^{-n}$ of degree $x$; as such they are known as $q$-Krawtchouk polynomials. Throughout the paper we will just refer to these functions as $q$-Krawtchouk polynomials. We define `1-site duality functions' by
\begin{align}
	k(n,x;\rho;N;q) = c_{\mathrm k}(n,\rho,N)  K_n(x;q^{2\rho},N;q^2),\label{eq:1siteduality}
\end{align}
where the coefficient $c_\mathrm{k}$ can be found in appendix \ref{app:overview}. We define the duality function $K_{\mathsf R}: X \times X \to \R$ as a product of the 1-site duality functions,  
\begin{equation}\label{eq:duality function K}
	K_{\mathsf R}(\eta,\xi)= K_{\mathsf R}(\eta,\xi;\rho,\vec{N};q) = q^{-\frac12 u(\eta;\vec{N})} \prod_{k=1}^M k(\eta_k, \xi_k;h^{+}_{k+1}(\xi);N_k;q), 
\end{equation}
where 
\begin{equation} \label{eq: u(eta,N}
	\begin{split}
		u(\eta;\vec{N}) &= \sum_{k=1}^M  \Big(\eta_kN_k-2\eta_k\sum_{j=1}^{k}  N_j \Big). 
	\end{split}
\end{equation}
Note that the product \eqref{eq:duality function K} has a `nested' structure as the $k$-th factor, which is the 1-site duality function corresponding to site $k$, depends on $\xi_{k+1},\ldots,\xi_M$ through $h^{+}_{k+1}(\xi)$, i.e.~on the dual particles on the right of site $k$. The following result shows that $K_{\mathsf R}$ is a duality function between $\ASEP{q,\vec{N}}$ and $\RASEP{q,\vec{N},\rho}$. 
\begin{theorem}\label{Thm:dualityASEPRASEP}
	For states $\eta,\xi \in X$,
	\[
	[ L_{q,\vec{N}} K_{\mathsf R}(\,\cdot\, , \xi)](\eta) = [L_{q,\vec{N},\rho}^{\mathsf R} K_{\mathsf R}(\eta,\,\cdot\,)](\xi).
	\]
\end{theorem}
The full proof, for which we rely on quantum algebra techniques, can be found in Section \ref{subsec:dualityDynASEPandASEP}. Below we give a short sketch of the proof.
\begin{proof}[Sketch of proof]
	Using that both generators can be written as a sum of operators acting on two sites, proving the duality boils down to showing it for the part of the generator for sites $k,k+1$. That is, we write out both sides of the duality equation for the interaction between these two sites and obtain
	\[
		\begin{split} c^{+}_k(\eta) [K_\mathsf{R}(\eta^{k,k+1},\xi)-K_\mathsf{R}(\eta,\xi)] + c^-_{k+1}(\eta) [K_\mathsf{R}(\eta^{k+1,k},\xi)-K_\mathsf{R}(\eta,\xi)] &=\\
			 C^{\mathsf R,+}_k(\xi) [K_\mathsf{R}(\eta,\xi^{k,k+1})-K_\mathsf{R}(\eta,\xi)] + C^{\mathsf R,-}_{k+1}(\xi) [&K_\mathsf{R}(\eta,\xi^{k+1,k})-K_\mathsf{R}(\eta,\xi)].\end{split}
	\]
	This (quite complex) identity for the multivariate $q$-Krawtchouk polynomials $K_\mathsf{R}$, where an action on the $\eta$-variable is transferred to the $\xi$ variable,  can then be proven using representations of the quantum algebra $\U_q(\mathfrak{sl}_2)$.
\end{proof}
\begin{remark}
	When taking $N_1=...=N_M=1$, this is the duality proven in \cite{BoCo}, although with a different duality function. We will address this further in Remark \ref{rem:HahnDuality}.
\end{remark}
Since $\RASEP{q,\vec{N},\rho}$ is invariant under $q \mapsto q^{-1}$, we immediately also obtain a duality with $\ASEP{q^{-1},\vec{N}}$.
\begin{corollary}\label{Cor:dualityASEPRASEP}
	For states $\eta,\xi \in X$,
	\[
	[ L_{q^{-1},\vec{N}} K_{\mathsf R}(\,\cdot\, , \xi;q^{-1})](\eta) = [L_{q,\vec{N},\rho}^{\mathsf R} K_{\mathsf R}(\eta,\,\cdot\,;q^{-1})](\xi).
	\]
\end{corollary}

\subsection{Current of $\RASEP{q,\vec{N},\rho}$}
	Similar to Section 3.4 of \cite{CGRS}, we can show that the expectation of the particle current of generalized dynamic ASEP can be expressed in terms of expectations of an ASEP system with only one dual particle. We define the hyperbolic current of $\RASEP{q,\vec{N},\rho}$, started from $\xi(0)$ and now at time $t$, as
	\[
		J_{k}^{\text{hyp}}(t)= \frac{[h_k^+(\xi(t))]_q}{[h_k^+(\xi(0))]_q},
	\]
	where $\xi(t)$ is the state at time $t$. Since the map $[\cdot]_q:\R\to\R$ is bijective, one can obtain $ h_k^+(\xi(t))$ uniquely from the value of $J_{k}^{\text{hyp}}(t)$. Then we can calculate
	\[
		 h_k^+(\xi(t)) - h_k^+(\xi(0)) = \sum_{j\geq k} 2\xi_j(t) - \sum_{j\geq k} 2\xi_j(0) = 2J_k(t),
	\]
	where $J_k(t)$ is the usual current which counts the number of particles in the time interval $[0,t]$ that jumped from site $k-1$ to site $k$ minus the particles that jumped in the opposite direction. Assume $q\in(0,1)$, then $\RASEP{q,\vec{N},\rho}$ becomes $\ASEP{q,\vec{N}}$ in the limit $\rho\to-\infty$ and
	\[
		\lim\limits_{\rho\to-\infty} J_{k}^{\text{hyp}}(t) = q^{2J_k(t)}.
	\]
	Thus our definition of the hyperbolic current corresponds in the limit $\rho\to-\infty$ to the $q$-exponential current from \cite{CGRS}. Using the duality function $K_\mathsf{R}$, we can prove the following theorem which links the first moment of the hyperbolic current of $\RASEP{q,\vec{N},\rho}$ to expectations of $\ASEP{q,\vec{N}}$ with one dual particle, which is just an asymmetric nearest neighbor random walker on the lattice $\{1,2,...,M\}$.
	\begin{theorem}
		Let $\xi=\xi(0)$ be a configuration of $\RASEP{q,\vec{N},\rho}$. Then the first moment of the hyperbolic current satisfies
	\begin{align*}
		&\mathbb{E}_\xi\bigg[ J_{k}^{\text{hyp}}(t) \bigg] = q^{\sum_{j=1}^{k-1}N_j}\frac{\big[\rho + 2|\xi| -|\vec{N}|\big]_q}{[h_k^+(\xi)]_q}  \\
		&\hspace{2.1cm}+ \sum_{i=1}^{k-1} q^{\sum_{j=i+1}^{k-1}N_j-\sum_{j=1}^{i-1}N_j}\frac{[N_i]_q}{[h_k^+(\xi)]_q} \mathbf{E}_i \bigg[\frac{q^{\sum_{j=1}^{n(t)} N_j}}{[N_{n(t)}]_q}(q^{-N_{n(t)}}[h_{n(t)+1}^+(\xi)]_q - [h_{n(t)}^+(\xi)]_q)\bigg],
	\end{align*}
		where $\mathbb{E}_\xi$ is the expectation of $\RASEP{q,\vec{N},\rho}$ which started at $\xi$ at $t=0$ and $\mathbf{E}_j$ is the expectation of one random walker started at site $j$ at $t=0$, which jumps from site $i$ to $i-1$ with rate $q^{N_i}[N_{i-1}]_q$ and from site $i$ to site $i+1$ with rate $q^{-N_i}[N_{i+1}]_q$.
	\end{theorem}
	\begin{proof} 
	Recall that $e_j$ is the one-particle state with a particle at site $j$. From the duality relation on generator level given by Theorem \ref{Thm:dualityASEPRASEP}, we obtain
	\begin{align}
	\mathbb{E}_\xi \big[K_\mathsf{R}(e_k,\xi(t))\big] = \mathbf{E}_k \big[ K_\mathsf{R}(e_{n(t)},\xi )\big],\label{eq:dualityexpectation}
	\end{align}
	where $n(t)$ is the position of an asymmetric random walker with rates from $\ASEP{q,\vec{N}}$, i.e. it jumps from site $j$ to $j-1$ with rate $q^{N_j}[N_{j-1}]_q$ and from site $j$ to $j+1$ with rate $q^{-N_j}[N_{j+1}]_q$. We will explicitly compute $K_\mathsf{R}$. Taking $\eta=e_k$ in the definition of $K_\mathsf{R}(\eta,\xi)$ from \eqref{eq:1siteduality} and \eqref{eq:duality function K} gives
	\begin{align*}
		K_\mathsf{R}(e_k,\xi)&= q^{-\half u(e_k;\vec{N})} k(1,\xi_k;h_{k+1}^+(\xi);N_k;q)\\
		&= -q^{- \half -h_{k+1}^+(\xi)+ \sum_{j=1}^k N_j}\rphis{3}{2}{q^{-2}, q^{-2\xi_k},-q^{2h_{k+1}^++2\xi_k -2 N_k}}{q^{-2N_k}, 0}{q^2;q^2} \\
		&=  -q^{- \half -h_{k+1}^+(\xi)+ \sum_{j=1}^k N_j}\bigg(1 + \frac{(1-q^{-2})(1-q^{-2\xi_k})(1+q^{2h_{k+1}^++2\xi_k -2 N_k})}{(1-q^{-2N_k})(1-q^2)}q^2 \bigg)\\
		&= -\frac{q^{- \half -h_{k+1}^+(\xi)+ \sum_{j=1}^k N_j}}{(1-q^{-2N_k})}\big(1-q^{-2N_k}-(1-q^{-2\xi_k})(1+q^{2h_{k+1}^++2\xi_k -2 N_k})  \big)\\
		&= -\frac{q^{- \half + \sum_{j=1}^{k-1} N_j}(q-q^{-1})}{(1-q^{-2N_k})}(q^{-N_k}[h_{k+1}^+(\xi)]_q - [h_k^+(\xi)]_q).
	\end{align*} 
	Therefore, \eqref{eq:dualityexpectation} gives
	\begin{align*}
		&\frac{q^{\sum_{j=1}^{k-1} N_j}}{(1-q^{-2N_k})}\mathbb{E}_\xi \bigg[q^{-N_k}[h_{k+1}^+(\xi(t))]_q - [h_k^+(\xi(t))]_q\bigg] \\
		&\hspace{2cm} = \mathbf{E}_k \bigg[\frac{q^{\sum_{j=1}^{n(t)-1} N_j}}{(1-q^{-2N_{n(t)}})}(q^{-N_{n(t)}}[h_{n(t)+1}^+(\xi)]_q - [h_{n(t)}^+(\xi)]_q)\bigg],
	\end{align*}
	which leads to the recursive relation
	\begin{align*}
		&\mathbb{E}_\xi \bigg[[h_{k+1}^+(\xi(t))]_q\bigg] = q^{N_k}\mathbb{E}_\xi \bigg[[h_k^+(\xi(t))]_q\bigg] \\
		&\hspace{3cm} + q^{-\sum_{j=1}^{k-1} N_j}[N_k]_q \mathbf{E}_k \bigg[\frac{q^{\sum_{j=1}^{n(t)} N_j}}{[N_{n(t)}]_q}(q^{-N_{n(t)}}[h_{n(t)+1}^+(\xi)]_q - [h_{n(t)}^+(\xi)]_q)\bigg].
	\end{align*}
	Applying this formula $k-1$ times and using $h_1^+(\xi(t))=\rho + 2|\xi| -|\vec{N}|$, we obtain
	\begin{align*}
		&\mathbb{E}_\xi\bigg[ [h_k^+(\xi(t))]_q \bigg] = q^{\sum_{j=1}^{k-1}N_j}\big[\rho + 2|\xi| -|\vec{N}|\big]_q  \\
		&\hspace{2.7cm}+ \sum_{i=1}^{k-1} q^{\sum_{j=i+1}^{k-1}N_j-\sum_{j=1}^{i-1}N_j}[N_i]_q \mathbf{E}_i \bigg[\frac{q^{\sum_{j=1}^{n(t)} N_j}}{[N_{n(t)}]_q}(q^{-N_{n(t)}}[h_{n(t)+1}^+(\xi)]_q - [h_{n(t)}^+(\xi)]_q)\bigg].
	\end{align*}
	Now divide both sides by $[h_k^+(\xi)]_q$. 
	\end{proof}

\subsection{Reversibility of dynamic ASEP}
The $q$-Krawtchouk polynomials are well-known orthogonal polynomials and the orthogonality relations imply orthogonality relations for the duality functions that we will state here. First, the orthogonality relations for the $q$-Krawtchouk polynomials read in terms of the 1-site duality functions $k(n,x)=k(n,x;\rho;N;q)$ (see \cite[\S14.15 and \S14.17]{KLS}), 
\begin{equation} \label{eq:orthogonality 1-site k}
	\begin{split}
		\sum_{x=0}^N  k(m,x) k(n,x) W(x) &= \frac{ \de_{m,n} }{w(n)},\\
		\sum_{n=0}^N  k(n,x) k(n,y) w(n) &= \frac{ \de_{x,y} }{W(x)},
	\end{split}
\end{equation}
with positive weight functions $w$ and $W$ given by 
\begin{align}
	w(n;N;q) &= q^{n(n-N)} \qbinom{q^2}{N}{n},\label{eq:orthokrawtchoukASEP}\\
	W(x;N,\rho;q)&= \frac{ 1+q^{4x+2\rho-2N}}{1+q^{2\rho-2N}} \frac{ (-q^{2\rho-2N};q^2)_x }{(-q^{2\rho+2};q^2)_x } \frac{q^{-x(2\rho+1+x-2N)}}{(-q^{-2\rho};q^2)_N} \qbinom{q^2}{N}{x}. \label{eq:orthokrawtchoukRASEP}
\end{align}
For $q\to 1$ both these weight functions become (a multiple of) a binomial coefficient. The case $q \to 1$ is studied in more detail in Section \ref{sec:symmetric degenerations}. With these weight functions we define the following weight functions on the state space $X$, 
\begin{align}
	w(\eta;\vec{N};q)&= q^{u(\eta;\vec{N})} \prod_{k=1}^M w(\eta_k;N_k;q), \label{eq:Weightfunction w} \\  
	W_{\mathsf R}(\xi;\vec{N},\rho;q)& = \prod_{k=1}^M W(\xi_k;N_k,h^{+}_{k+1}(\xi);q),\label{eq:Weightfunction WR}
\end{align}
where the factor $u$ is given by \eqref{eq: u(eta,N}. These weight functions provide us with reversible measures with respect to which the duality functions $K_{\mathsf R}$ are orthogonal.
\begin{theorem}\* \label{Thm:reversiblemeasures}
	\begin{enumerate}[label = (\roman*)]
		\item The weight function $w(\,\cdot\,;\vec{N};q)$ is a reversible measure for $\ASEP{q,\vec{N}}$.
		\item The weight function $W_{\mathsf R}(\,\cdot\,;\vec{N},\rho;q)$ is a reversible measure for $\RASEP{q,\vec{N},\rho}$.
		\item The duality functions $K_{\mathsf R}(\eta,\xi)$ defined by \eqref{eq:duality function K} satisfy the orthogonality relations
		\[
		\begin{gathered}
			\sum_{\eta \in X} K_{\mathsf R}(\eta,\xi) K_{\mathsf R}(\eta,\xi') w(\eta;\vec{N};q) = \frac{\de_{\xi,\xi'}}{W_{\mathsf R}(\xi;\vec{N},\rho;q)},\\
			\sum_{\xi \in X} K_{\mathsf R}(\eta,\xi) K_{\mathsf R}(\eta',\xi) W_{\mathsf R}(\xi;\vec{N},\rho;q) = \frac{\de_{\eta,\eta'}}{w(\eta;\vec{N};q)}.
		\end{gathered}
		\]
	\end{enumerate}
\end{theorem}
\begin{proof}
	(i) and (ii). One can prove this directly by checking the detailed balance condition. Alternatively, one can show that the generators $L_{q,\vec{N}}$ and $L_{q,\vec{N},\rho}^\mathsf{R}$ are self-adjoint with respect to $w$ and $W_\mathsf{R}$ respectively. This is done via quantum algebra techniques at the end of Section \ref{subsec:reprUqASEP} and the beginning of Section \ref{subsec:RevDynASEP}.
	
	(iii) The first orthogonality relation follows straightforwardly by applying the orthogonality relations \eqref{eq:orthogonality 1-site k} for the 1-site duality function with respect to the weight $w$. The second orthogonality relation can also be checked directly using the orthogonality \eqref{eq:orthogonality 1-site k} with respect to $W$. This computation is a bit more involved than the first computation because of the nested product structure of the duality functions and the weight function. Alternatively, the second orthogonality follows from the first using standard linear algebra: if $A$ is a square matrix of finite size satisfying $A^tA=I$, then $AA^t=I$ also holds.
\end{proof}
\begin{remark}
	Since $\RASEP{q,\vec{N},\rho}$ is reversible, we get in particular that the process is self-dual. At the moment however, besides the cheap duality functions, i.e. of the  Kronecker-delta type, we have no explicit self-duality functions.
\end{remark}
\begin{remark}
	If for each $k$ one takes $N_k=2j$, $j\in\half \N$, part (i) of the above theorem gives the same family of reversible measures as in Theorem 3.1(a) in \cite{CGRS} (note that in this paper a slightly different definition for the $q$-binomial coefficient is used).
\end{remark}
\begin{remark}
	Since $\RASEP{q,\vec{N},\rho}$ is invariant under sending $q\to q^{-1}$, one would suspect that its reversible measure $W_\mathsf{R}$ is as well. This is true up to a factor depending on $|\xi|$ and $|\vec{N}|$. Indeed, one can compute that for the single site weight function we have
	\begin{align*}
		W(x;N,\rho;q^{-1})=W(x;N,\rho;q) q^{4x(x+\rho-N)+N(N-2\rho-1)}.
	\end{align*}
	Therefore, we can make the single site weight function invariant by taking
	\begin{align*}
		W^\text{inv}(x;N,\rho;q) = q^{2x(x+\rho-N)+\half N(N-2\rho-1)}	W(x;N,\rho;q).
	\end{align*}
	Then, we can define a measure by taking the product of these single site weight functions,
	\begin{align*}
		W_\mathsf{R}^\text{inv}(\xi,\vec{N},\rho;q) = \prod_{k=1}^M W^{\text{inv}}(\xi_k;N_k,h^{+}_{k+1}(\xi);q).
	\end{align*}
	At this point, it is not clear at all that this is still a reversible measure. However, a direct computation shows that 
	\begin{align*}
		W_\mathsf{R}^\text{inv}(\xi,\vec{N},\rho;q) = q^{z(|\xi|,|\vec{N}|,\rho)}\prod_{k=1}^M W(\xi_k;N_k,h^{+}_{k+1}(\xi);q),
	\end{align*}
	where
	\[
	z(a,b,c)=2a\big( a-b\big)+\half b\big(b-1\big) + c(2a-b).
	\]
	Thus by Remark \ref{rem:InvTotPart}, $W^\text{inv}_\mathsf{R}$ is a reversible measure since $W_\mathsf{R}$ is. While $W_{\mathsf R}^\text{inv}$ is invariant under $q \leftrightarrow q^{-1}$, $W_{\mathsf R}$ itself has the useful property that for any $1\leq j \leq M$,
	\begin{align*}
		\sum_{\xi_j,..,\xi_M} \prod_{k=j}^M W(\xi_k;N_k,h_{k+1}(\xi);q) = 1.
	\end{align*}
	This property, which follows directly from the second equality of Theorem \ref{Thm:reversiblemeasures}(iii), taking $\eta=\eta'=0$ and realizing that $M$ can be chosen arbitrarily without changing the single site weight functions, is no longer true for $W_\mathsf{R}^\text{inv}$.	Throughout this paper, we will work with $W_{\mathsf R}$.

\end{remark} 
\begin{remark}
	In \cite{BoCo}, (half-)stationary initial data for the standard dynamic ASEP ($N_k=1$ for all $k$) is generated by a Markov chain $(s_k)_{k\in\Z}$ which has transition probabilities
	\begin{align*}
		\mathbb{P}(s_{k-1} = s_k + 1)= \frac{q^{s_k}}{\alpha + q^{s_k}} \qquad \text{and} \qquad \mathbb{P}(s_{k-1} = s_k - 1)= \frac{\alpha}{\alpha + q^{s_k}}.
	\end{align*}
	In the language of this paper, we have to take $\alpha =q^{-\rho}$, $h^+_k = s_k +\rho$ and $q^2$ instead of $q$, so that we obtain
	\begin{align*}
		\mathbb{P}(h_{k-1}^+ = h_k^+ + 1)= \frac{q^{2h_k^+}}{1 + q^{2h_k^+}} \qquad \text{and} \qquad \mathbb{P}(h_{k-1}^+ = h_k^+ - 1)= \frac{1}{1 + q^{2h_k^+}}.
	\end{align*}
	The reversible measure $W_\mathsf{R}$ is connected to this (half-)stationary initial data in the following way. Take $N_1=N_2=...=N_M=1$, assume $h^+_{k}=c\in\R$ is given and define $A$ to be the set of $\xi_k,...,\xi_M$ such that  $h^+_k = c$ (which is of course only possible for specific values of $\rho$). Then by Bayes' law,
	\begin{align}
		\mathbb{P}(h_{k-1}^+=h_k^++1|h^+_k=c) &= \frac{\mathbb{P}(h_{k-1}^+=h_k^++1 \text{ and }h^+_k=c)}{\mathbb{P}(h^+_k=c)}\nonumber \\
		&= \frac{\displaystyle \sum_{\xi\in \{0,1\}^{k-2}\times \{1\}\times A } W_\mathsf{R}(\xi;\vec{N},\rho;q)}{\displaystyle \sum_{\xi\in\{0,1\}^{k-1}\times A} W_\mathsf{R}(\xi;\vec{N},\rho;q)}\nonumber.
	\end{align}
	If we now divide both numerator and denominator by
	\[
	\displaystyle \sum_{\xi_k,...,\xi_M\in A}\ \prod_{j=k}^{M}W(\xi_j;1,h^+_{j+1};q),
	\] 
	we obtain
	\begin{align}
		= \frac{\displaystyle\sum_{\xi_1,...,\xi_{k-2}\in\{0,1\}^{k-2}}W(1;1,c;q)W(\xi_{k-2};1,c+1;q)\prod_{j=1}^{k-3}W(\xi_j;1,h^+_{j+1};q)}{\displaystyle\sum_{\xi_1,...,\xi_{k-1}\in\{0,1\}^{k-1}}W(\xi_{k-1};1,c;q)W(\xi_{k-2};1,c+2\xi_{k-1}-1;q)\prod_{j=1}^{k-3}W(\xi_j;1,h^+_{j+1};q)}. \label{eq:fracrevmeasures}
	\end{align}
	Since $\displaystyle\sum_\xi W_\mathsf{R}(\xi;\vec{N},\rho;q)=1$ for any $\rho$ and number of sites $M$, we have 	
	\begin{align*}
		\displaystyle\sum_{\xi_1,...,\xi_{k-2}\in\{0,1\}^{k-2}}W(\xi_{k-2};,1,c+1;q)\prod_{j=1}^{k-3} W(\xi_j;1,h^+_{j+1};q)=&1,\\
		\displaystyle\sum_{\xi_1,...,\xi_{k-2}\in\{0,1\}^{k-2}}W(\xi_{k-2};1,c-1;q)\prod_{j=1}^{k-3}W(\xi_j;1,h^+_{j+1};q)=&1,\\
		W(0;1,c;q)+W(1;1,c;q)=&1.
	\end{align*}
	Therefore, \eqref{eq:fracrevmeasures} is equal to
	\begin{align*}
		\frac{W(1;1,c;q)}{W(0;1,c;q)+W(1;1,c;q)} = W(1;1,c;q).
	\end{align*}
	Similarly, we have
	\[
	\mathbb{P}(h_{k-1}^+=h_k^+-1|h^+_k=c) = W(0;1,c;q).
	\]
	Thus
	\begin{align*}
		&h^+_{k-1} = h^+_{k} + 1 \text{ with probability } W(1;1,h_{k}^+;q) = \frac{q^{-2h_k^+}}{1+q^{-2h_k^+}},\\
		&h^+_{k-1} = h^+_{k} - 1 \text{ with probability } W(0;1,h_{k}^+;q) = \frac{1}{1+q^{-2h_k^+}}.
	\end{align*}
	Taking $q^{-1}$ instead of $q$ (which is possible since the process is invariant under this transformation) we get the desired probabilities of the Markov chain which generates the (half-)stationary initial data. In \cite{BoCo} the initial data is found using an orthogonality measure for $q^{-1}$-Hermite polynomials, so interestingly, the orthogonality measure of our multivariate $q$-Krawtchouk polynomials for $N_1=\ldots=N_M=1$ is connected to that particular orthogonality measure of $q^{-1}$-Hermite polynomials. 
\end{remark}

If we let the dynamic parameter $\rho$ tend to $\pm \infty$, one expects the reversible measure $W$ for dynamic ASEP to become the reversible measure $w$ for ASEP. This is, up to functions depending on the total number of particles (see Remark \ref{rem:InvTotPart}), indeed the case. Let us assume $q \in (0,1)$. Then the following limit relations hold,
\begin{equation} \label{eq:rho->infty W->w}
	\begin{split}
		\lim_{\rho \to \infty}  q^{2\rho(|\xi|-|\vec{N}|)} W_{\mathsf R}(\xi;\vec{N},\rho;q) 
		&= q^{\be_1(|\xi|,|\vec{N}|) } w(\xi;\vec{N};q^{-1}),\\
		\lim_{\rho \to -\infty} q^{2\rho|\xi|} W_{\mathsf R}(\xi;\vec{N},\rho;q)
		&= q^{\be_2(|\xi|,|\vec{N}|)} w(\xi;\vec{N};q),
	\end{split}
\end{equation}
with 
\[
\begin{split}
	\be_1(a,b) &= b-a(a+1)- (b-a)^2,\\
	\be_2(a,b) & = a(1-2a+2b).
\end{split}
\]
See Section \ref{sec:calculations of limits} for the explicit calculations.

\subsection{Dynamic ASEP on the reversed lattice}\label{subsec:LASEP}
In the higher spin version of dynamic ASEP from Definition \ref{Def:dynamic ASEP} the value $h^{+}_{k}$ of the height function at site $k$ is determined by the number of particles on the right of site $k$ and the boundary value $\rho$ at the most right side. Here we assume that the sites are labeled from left to right: site $1$ is the most left site and site $M$ is the most right site. In a similar way, we can define a higher spin version of dynamic ASEP where the value of the height function at site $k$ depends on the number of particles on the left of site $k$ and the boundary value at the most left site. In particular, this left dynamic ASEP is the right dynamic ASEP defined on the reversed lattice, i.e.~the sites are labeled from right to left. Relabeling $k \mapsto M-k+1$ then gives a process (on the lattice labeled from left to right) where the rates to and from site $k$ depend on the number of particles on the left of site $k$ and the boundary value of the height function at a virtual site on the left (site $0$). We denote this `left version' of generalized dynamic ASEP by $\LASEP{q,\vec{N},\lambda}$ and its generator by $\genLASEP$. Here, $\lambda \in \R$ is the value of the height function on the virtual site $0$ and, as before, $q>0$ and $\vec{N} \in \N^M$. \\ 

Let us write down the jump rates of the process explicitly. Given a state $\zeta=(\zeta_k)_{k=1}^M \in X$, the  height function $(h^{-}_{k}(\zeta))_{k=1}^M$ is given by
\[
h^{-}_{k} = h^{-}_{k,\lambda}(\zeta) = \lambda + \sum_{j=1}^{k} (2\zeta_j - N_j),
\]
where we again suppress the dependence on $\lambda$ and/or $\zeta$ unless confusion could arise.
We set the left boundary value at the virtual site $0$ by
\begin{align*}
	h^-_{0} = \lambda.
\end{align*}
A particle on site $k$ jumps to site $k+1$ at rate $C^{\mathsf L,+}_{k}(\zeta) = C_{M-k+1}^{\mathsf R,-}(\zeta^{\rev},\vec{N}^\rev)$, and a particle on site $k$ jumps to site $k-1$ at rate $C^{\mathsf L,-}_{k}(\zeta) = C^{\mathsf R,+}_{M-k+1}(\zeta^\rev,\vec{N}^\rev)$. Also, observe that 
\[
h^+_{M-k+1,\lambda}(\zeta^\rev) = h^-_{k,\lambda}(\zeta).
\]
Then a small calculation shows that the jump rates are given by
\[
C^{\mathsf L,+}_k(\zeta) = c^+_k(\zeta) \frac{(1+q^{2(h^{-}_{k-1}+\zeta_k)})(1+q^{2(h^{-}_{k}+\zeta_{k+1})})}{(1+q^{2h^{-}_{k}})(1+q^{2(h^{-}_{k}-1)})},
\]
and 
\[
C^{\mathsf{L},-}_{k}(\zeta) = c^-_{k}(\zeta) \frac{(1+q^{2(h^{-}_{k-1}-\zeta_{k-1})})(1+q^{2(h^{-}_{k}-\zeta_{k})})}{(1+q^{2h^{-}_{k-1}})(1+q^{2(h^{-}_{k-1}+1)})}.
\]
\begin{remark}\*
	\begin{itemize}
		\item We see that, for $q \in (0,1)$, in the limit $\lambda \to \infty$ we recover $\ASEP{q,\vec{N}}$ from $\LASEP{q,\vec{N},\lambda}$ and in the limit $\lambda \to -\infty$ we recover $\ASEP{q^{-1},\vec{N}}$.
		\item  When $N_1=...=N_M=1$, we obtain
		\[
		C^{\mathsf L,+}_k(\zeta) = q^{-1} \frac{1+q^{2h^{-}_{k}}}{1+q^{2h^{-}_{k}-2}},
		\]
		and 
		\[
		C^{\mathsf{L},-}_{k+1}(\zeta) = q \frac{1+q^{2h^{-}_{k}}}{1+q^{2h^{-}_{k}+2}}.
		\] 
		Notice that these are exactly the rates from \cite{BoCo} for jumps of the height function of standard dynamic ASEP where the height function is replaced by minus the height function.
	\end{itemize}
\end{remark}

Since ASEP$_{\mathsf L}$ is the same as ASEP$_{\mathsf R}$ on the reversed lattice, by Corollary  \ref{Cor:dualityASEPRASEP} it follows that $\LASEP{q,\vec{N},\lambda}$ is in duality with (nondynamic) ASEP with parameter $q^{-1}$ on the reversed lattice, with duality function $K_{\mathsf R}(\eta^\rev,\zeta^\rev;\lambda,\vec{N}^\rev;q^{-1})$ for $\eta,\zeta \in X$.  Reversing the lattice for ASEP is the same as replacing parameter $q$ by $q^{-1}$ on the nonreversed lattice. To be precise, for the jump rates we have
\[
\begin{split}
	c_{M-k+1}^+(\eta^\rev;\vec{N}^\rev;q^{-1}) &= c_{k}^-(\eta;\vec{N};q), \\
	c_{M-k+1}^-(\eta^\rev;\vec{N}^\rev;q^{-1}) &= c_{k}^+(\eta;\vec{N};q).
\end{split}
\]
It follows that we have duality between $\LASEP{q,\vec{N},\lambda}$ and $\ASEP{q,\vec{N}}$ with the following nested products of $q$-Krawtchouk polynomials as duality functions
\[
q^{\frac12 u(\eta^\rev;\vec{N}^\rev)} \prod_{k=1}^M k(\eta_k, \zeta_k;h^{-}_{k-1}(\zeta);N_k;q^{-1}).
\]
In order to have orthogonality with respect to reversible measure $w$ from \eqref{eq:Weightfunction w} we multiply this by
\[
q^{-\frac12[u(\eta;\vec{N})+u(\eta^\rev;\vec{N}^\rev)]} = q^{|\eta| |\vec{N}|},
\]
to obtain the following duality result. 
\begin{theorem} \label{Thm:dualityASEPLASEP}
	For $\eta,\zeta \in X$ define 
	\[
	K_{\mathsf L}(\eta,\zeta) = q^{-\frac12 u(\eta;\vec{N})} \prod_{k=1}^M k(\eta_k, \zeta_k; h^{-}_{k-1}(\zeta);N_k;q^{-1} ),
	\]
	then
	\[
	[L_{q,\vec{N}}K_{\mathsf L}(\,\cdot\,,\zeta)](\eta) = [L_{q,\vec{N},\lambda}^{\mathsf L} K_{\mathsf L}(\eta,\,\cdot\,)](\zeta).
	\]
\end{theorem}
In this case, we have again orthogonality relations with respect to the $\ASEP{q,\vec{N}}$ reversible measure $w(\eta;q,\vec{N})$, which by \eqref{eq:orthokrawtchoukASEP q<->q_inv} can written as 
\[
w(\eta;\vec{N};q) = q^{u(\eta;\vec{N})} \prod_{k=1}^M w(\eta_k;N_k;q^{-1}),
\]
and the $\LASEP{q,\vec{N},\lambda}$ reversible measure
\begin{equation} \label{eq: WL=WR-rev}
	\begin{split}
		W_{\mathsf L}(\zeta;\vec{N},\lambda;q) &= W_{\mathsf R}(\zeta^\rev;\vec{N}^\rev,\lambda;q^{-1}) \\ 
		&= \prod_{k=1}^M W(\zeta_k;N_k,h^{-}_{k-1}(\zeta);q^{-1}).
	\end{split}
\end{equation}
Summarizing, we obtain the following corollary which is similar to Theorem \ref{Thm:dualityASEPRASEP}.
\begin{corollary}\* \label{Cor:reversiblemeasures}
	\begin{enumerate}[label = (\roman*)]
		\item The weight function $W_{\mathsf L}(\,\cdot\,;\vec{N},\lambda;q)$ is a reversible measure for $\LASEP{q,\vec{N},\lambda}$.
		\item The duality functions $K_{\mathsf L}(\eta,\zeta)$ defined in Theorem \ref{Thm:dualityASEPLASEP} satisfy the orthogonality relations
		\[
		\begin{gathered}
			\sum_{\eta \in X} K_{\mathsf L}(\eta,\zeta) K_{\mathsf L}(\eta,\zeta') w(\eta;\vec{N};q) = \frac{\de_{\zeta,\zeta'}}{W_{\mathsf L}(\zeta;\vec{N},\lambda;q)},\\
			\sum_{\zeta \in X} K_{\mathsf L}(\eta,\zeta) K_{\mathsf L}(\eta',\zeta) W_{\mathsf L}(\zeta;\vec{N},\lambda;q) = \frac{\de_{\eta,\eta'}}{w(\eta;\vec{N};q)}.
		\end{gathered}
		\]
	\end{enumerate}
\end{corollary}
Similar to before, we can obtain the reversible measures for $\ASEP{q,\vec{N}}$ and $\ASEP{q^{-1},\vec{N}}$ by taking limits from $W_\mathsf{L}$,
\begin{equation} \label{eq:la->pm infty W->w}
	\begin{gathered}
		\lim_{\la \to \infty} q^{-2\la|\zeta|} W_{\mathsf L}(\zeta;\vec{N},\la;q)
		= q^{\ga_1(|\zeta|)}w(\zeta;\vec{N};q), \\
		\lim_{\la \to -\infty}  q^{-2\la(|\zeta|-|\vec{N}|)} W_{\mathsf L}(\zeta;\vec{N},\la;q) 
		= q^{\ga_2(|\zeta|,|\vec{N}|) } w(\zeta;\vec{N};q^{-1}).
	\end{gathered}
\end{equation}
with
\[
\begin{split}
	\ga_1(a) & = a(2a-1)\\
	\ga_2(a,b) & = -b+a(a+1)+ 2(b-a)^2-2ab.
\end{split}
\]
\subsection{Almost self-duality for dynamic ASEP}
Since both $\LASEP{q,\vec{N},\lambda}$ and $\RASEP{q,\vec{N},\rho}$ are dual to $\ASEP{q,\vec{N}}$ with respect to the duality function $K_{\mathsf L}$ and $K_{\mathsf R}$ respectively, they are dual to each other with respect to the duality function $R^v:X \times X \to \R$ given by
\begin{align}\label{eq:sumKrawtchouk}
R^v(\zeta,\xi) =R^v(\zeta,\xi;\lambda,\rho,\vec{N};q)= \sum_{\eta \in X} K_{\mathsf L}(\eta,\zeta) K_{\mathsf R}(\eta,\xi) w(\eta) v^{|\eta|},
\end{align}
where $v \in \R^\times$ is a free parameter. That is, we have the duality relation,
\begin{align}\label{eq:dualityLASEPRASEP}
[L_{q,\vec{N},\lambda}^{\mathsf L} R^v(\, \cdot\, ,\xi)](\zeta) = [L_{q,\vec{N},\rho}^{\mathsf R} R^v(\zeta,\, \cdot\, )](\xi).
\end{align}
See section \ref{subsec:dualityLASEPRASEP} for more details of this scalar-product method, which is used before in e.g.~\cite{CarFraGiaGroRed}.
The function $R^v(\zeta,\xi)$ can be expressed as a nested product of $q$-Racah polynomials. The latter are $q$-hypergeometric polynomials defined by
\begin{align}
R_n(x;\alpha,\beta,\ga,\de;q) = \rphis{4}{3}{q^{-n},\alpha\beta q^{n+1}, q^{-x}, \gamma \delta q^{x+1} }{\alpha q, \beta \delta q, \gamma q}{q,q}. \label{eq:racah4phi3}
\end{align}
For $x,n \in \N$ these are polynomials in $q^{-x} + \gamma \delta q^{x+1}$ of degree $n$ and also polynomials in $q^{-n}+\alpha \beta q^{n+1}$ of degree $x$. The following theorem shows that the duality functions $R^v(\zeta,\xi)$ are a product of these $q$-Racah polynomials.
\begin{theorem}\label{Thm:KrawtchoukRacah}
	We define 1-site duality functions by
	\begin{equation} \label{eq:1-site duality r}
		\begin{split}
			r(y,x;\lambda,\rho,v,N;q) & =c_{\mathrm r}(y,x;\lambda,\rho,v,N;q)R_{y}(x;\alpha,\beta,\gamma,\delta;q^2),
		\end{split}
	\end{equation}
	where
	\[
	(\alpha,\beta,\gamma,\delta) = \left(-v^{-1}q^{\rho+\lambda-N-1}, vq^{\rho-\lambda-N-1}, q^{-2N-2}, -q^{2\lambda} \right).
	\]
	and the coefficient $c_{\mathrm r}$ can be found in Appendix \ref{app:overview}. Then the duality function $R^v$ can then be written as
	\begin{equation} \label{eq:duality Racah polynomials}
		R^v(\zeta,\xi) = \prod_{k=1}^M r(\zeta_k,\xi_k;h^{-}_{k-1}(\zeta),h^{+}_{k+1}(\xi),v,N_k;q).
	\end{equation}
\end{theorem} 
	The full proof can be found in Section \ref{subsec:dualityLASEPRASEP}. The proof boils down to showing the sum \eqref{eq:sumKrawtchouk} of $q$-Krawtchouk polynomials can be expressed in terms of $q$-Racah polynomials, see Lemma \ref{lem:krawthouckracah}. \\
	\\
	Note that \eqref{eq:duality Racah polynomials} is now a `doubly nested' product as the $k$-th factor corresponding to site $k$ depends on the particles on the left of site $k$ through $h^{-}_{k-1}(\zeta)$ as well as on the dual particles on the right of site $k$ through $h^{+}_{k+1}(\xi)$. The following results says that $R^v$ is an orthogonal duality function between $\LASEP{q,\vec{N},\lambda}$ and $\RASEP{q,\vec{N},\rho}$.
\begin{theorem}\* \label{Thm:orthogonaldualityLASEPRASEP}
		The duality function $R^v$ satisfies the following orthogonality relations
		\[
		\begin{split}
			\sum_{\zeta \in X} R^v(\zeta,\xi) R^{v^{-1}}(\zeta,\xi') W_{\mathsf L}(\zeta;\vec{N},\lambda;q) = \frac{ \delta_{\xi,\xi'}}{W_\mathsf{R}(\xi;\vec{N},\rho;q)},\\
			\sum_{\xi\in X} R^v(\zeta,\xi) R^{v^{-1}}(\zeta',\xi) W_{\mathsf R}(\xi;\vec{N},\rho;q) = \frac{ \delta_{\zeta,\zeta'}}{W_\mathsf{L}(\zeta;\vec{N},\lambda;q)}.
		\end{split}
		\]
\end{theorem}
The proof can again be found in Section \ref{subsec:dualityLASEPRASEP}.
\begin{remark}\label{rem:duality racah} \*
	\begin{itemize}
		\item The duality function $R$ satisfies the symmetry relation
		\[
		R^v(\zeta,\xi;\lambda,\rho,\vec{N};q)= R^v(\xi^\rev,\zeta^\rev;\rho,\lambda,\vec{N}^\rev;q^{-1}),
		\]
		which corresponds to the fact that $\LASEP{q,\vec{N},\lambda}$ is obtained from $\RASEP{q,\vec{N},\rho}$ by reversing the order of the sites and interchanging $\rho\leftrightarrow\lambda$.
		\item Since both $\genLASEP$ and $\genRASEP$ are invariant under $q\to q^{-1}$, one expects that the duality function between these two processes is invariant under this transformation as well. A straightforward calculation shows that, up to a factor depending on the total number of particles $|\zeta|$ and $|\xi|$, the duality function $R^v(\zeta,\xi)$ is invariant under sending $(q,v)\to (q^{-1},v^{-1})$.
		\item Some of the factors in the $1$-site duality function $r$ can be taken out or changed and we still obtain a duality function. Let us define $R':X \times X \to \R$ by
		\[
		R'(\zeta,\xi) = \prod_{k=1}^M r'(\zeta_k,\xi_k;h^{-}_{k-1}(\zeta),h^{+}_{k+1}(\xi),v,N_k;q)
		\]
		with
		\[
			\hspace{1cm} r'(x,y;\lambda,\rho,v,N;q)=  v^{y}\frac{(-vq^{\rho+\lambda-N+1};q^2)_{x} (-v^{-1}q^{\rho+\lambda-N+1};q^2)_{y}}{q^{y(y+\rho+\lambda-N)}} R_{y}(x;\alpha,\beta,\gamma,\delta;q^2),	
		\]
		with $\alpha,\beta,\gamma,\delta,v$ as in the definition \eqref{eq:1-site duality r} of $r$. Then $R'$ is also a duality function between $\LASEP{q,\vec{N},\lambda}$ and $\RASEP{q,\vec{N},\rho}$. Indeed, $R(\zeta,\xi) = c(\zeta,\xi) R'(\zeta,\xi)$ where 
		\begin{equation} \label{eq:invariant c}
			\begin{split}
				c(\zeta,\xi) &= \prod_{k=1}^M \frac{ (vq^{2\zeta_k-h^{+}_{k+1}(\xi)+h^{-}_{k-1}(\zeta)-N_k+1};q^2)_{N_k} }{ (vq^{-2\xi_k-h^{+}_{k+1}(\xi)+h^{-}_{k-1}(\zeta)+N_k+1};q^2)_{\xi_k+\zeta_k} } =  \frac{ (vq^{\lambda-\rho+2|\zeta|-|\vec{N}|+1};q^2)_{|\vec{N}|-|\zeta|} }{(vq^{\lambda-\rho-2|\xi|+|\vec{N}|+1};q^2)_{|\xi|}}
			\end{split}
		\end{equation}
		is a function depending only on the total number of particles and dual particles. Now the claim follows from Remark \ref{rem:InvTotPart}. For notational convenience later on, we define
		\begin{align} \label{eq:invariant c def}
			c^v(|\zeta|,|\xi|;\lambda,\rho):= c(\zeta,\xi).
		\end{align}
		Similarly, we have 
		\begin{align}
			C^v(|\zeta|,|\xi|;\lambda,\rho)=\prod_{k=1}^M \frac{(-vq^{h^{+}_{k+1}(\xi)+h^{-}_{k-1}(\zeta)-N_k+1};q^2)_{\xi_k}}{ (-vq^{h^{+}_{k+1}(\xi)+h^{-}_{k-1}(\zeta)-N_k+1};q^2)_{\zeta_k}} = 	\frac{(-vq^{\lambda+\rho-|\vec{N}|+1};q^2)_{|\xi|}}{(-vq^{\lambda+\rho-|\vec{N}|+1};q^2)_{|\zeta|}}.\label{eq:identity C^0}
		\end{align}	
		So that in the $1$-site duality function $r$ (and $r'$) the $x$ in the factor $(-vq^{\rho+\lambda-N+1};q^2)_{x}$ can be replaced by $y$ and we still have a duality function. Similarly, in $(-v^{-1}q^{\rho+\lambda-N+1};q^2)_{y}$ the $y$ can be replaced by $x$. This will be important when taking limits in the duality function $r$ in Section \ref{sec:Degenerations}. See Appendix \ref{sec:appendix Identities} for the simplification of $c^v$ and $C^v$.

	\end{itemize}
\end{remark}

\section{Asymmetric degenerations }\label{sec:Degenerations}
The duality between $\LASEP{q,\vec{N},\lambda}\leftrightarrow \RASEP{q,\vec{N},\rho}$ sits on top of a hierarchy of several other dualities. By taking appropriate limits in the duality equation
\begin{equation} \label{eq:duality equation dynamic ASEP}
	[L_{q,\vec{N},\lambda}^{\mathsf L} R^v(\, \cdot\, ,\xi)](\zeta) = [L_{q,\vec{N},\rho}^{\mathsf R} R^v(\zeta,\, \cdot\, )](\xi)
\end{equation}
from \eqref{eq:dualityLASEPRASEP}, we obtain several other duality relations. This boils down to taking appropriate limits of the process generators, duality functions, and orthogonality relations. That is, we will look at the limits of the form
\begin{itemize}
	\item $f(|\zeta|,|\xi|) R^v(\zeta,\xi)$, which are also duality functions by Remark \ref{rem:InvTotPart}. 
	\item $g(|\zeta|)W_\mathsf{L}(\zeta)$ and $g'(|\xi|)W_\mathsf{R}(\xi)$, which are still reversible measures (again by Remark \ref{rem:InvTotPart}).
\end{itemize}  
In all cases, the functions $f,g$, and $g'$ are conveniently chosen such that the limits are convergent. The explicit calculations of the limits, which are mostly straightforward computations using the $q$-hypergeometric expressions of the duality functions, are postponed to Section \ref{sec:calculations of limits}. In this section, we only look at asymmetric degenerations, by which we mean that we do not take limits in the parameter $q$ yet. This will be done in Section \ref{sec:symmetric degenerations}. Throughout this section, we assume $q \in (0,1)$.

\subsection{Dualities for the asymmetric exclusion process}
Recall from Section \ref{sec:GeneralizedDynASEP} the limits from dynamic ASEP to ASEP,
\[
\lim_{\rho \to \pm \infty} L^{\mathsf R}_{q,\vec{N},\rho} = L_{q^{\mp 1},\vec{N}} \qquad \text{and} \qquad \lim_{\lambda \to \pm \infty} L^{\mathsf L}_{q,\vec{N},\lambda} = L_{q^{\pm 1},\vec{N}}.
\]
This gives us three `different\footnote{Different in the sense that they cannot be obtained from one another by sending $q\to q^{-1}$ or reversing the order of sites.}' limit cases  of the duality between ASEP$_\mathsf{L}\leftrightarrow$ASEP$_\mathsf{R}$, which are listed in Table \ref{tab:limits_duality_dynamic_ASEP}.
\begin{table}[h] 
	\begin{tabular}{|l|ll|l|}
		\hline
		\ &\multicolumn{2}{|l|}{\hspace{0.85cm} Duality between} & Corresponding limit \\ \hline
		(i)&\multicolumn{1}{|l|}{$\ASEP{q,\vec{N}}$}       &  $\RASEP{q,\vec{N},\rho}$      &       $\lambda\to\infty$              \\ \hline
		(ii)&\multicolumn{1}{|l|}{$\ASEP{q,\vec{N}}$ }      &  $\ASEP{q,\vec{N}}$       &  $\lambda\to\infty$, $\rho\to-\infty$                   \\ \hline
		(iii)&\multicolumn{1}{|l|}{$\ASEP{q,\vec{N}}$}       &  $\ASEP{q^{-1},\vec{N}}$       &     $\lambda,\rho\to\infty$                \\ \hline
	\end{tabular}
	\vspace{0.1cm}
	\caption{Three limits of the duality between $\LASEP{q,\vec{N},\lambda}$ and $\RASEP{q,\vec{N},\rho}$.} \label{tab:limits_duality_dynamic_ASEP}\vspace{-0.7cm}
\end{table}\\
For the next proposition, we need the following orthogonal polynomials, which are special cases of $q$-Racah polynomials.
\begin{enumerate}[label=(\roman*)]
	\item The $q$-Hahn polynomials \cite[\S14.6 and \S14.7]{KLS} are defined by 
	\[
	P_n(x;\alpha,\beta,N;q) = R_n(x;\alpha,\beta, q^{-N-1},0;q) = \rphis{3}{2}{q^{-n},\alpha\beta q^{n+1},q^{-x}}{\alpha q, q^{-N}}{q,q}.
	\]
	We define the $1$-site duality function
	\begin{align}\label{eq:1sitedualityHahn}
		p(n,x;\lambda,\rho,v,N;q) =  c_{\mathrm p}(n,x;\lambda,\rho,v,N;q) P_x(n;\alpha,\beta,N;q^2),
	\end{align}
	where 
	\[
		(\alpha,\beta)=(-vq^{\rho+\lambda-N-1}, v^{-1}q^{\rho-\lambda-N-1}),
	\]
	and the coefficient $c_{\mathrm p}$ can be found in Appendix \ref{app:overview}.
	\item The quantum $q$-Krawtchouk polynomials \cite[\S14.14]{KLS} are defined by 
	\[
	K^{\text{qtm}}_n(x;\hat{p},N;q)= \rphis{2}{1}{q^{-n}, q^{-x}}{q^{-N}}{q, \hat{p}q^{n+1}}.
	\]
	We define the $1$-site duality function
	\begin{align}\label{eq:1sitedualityQuantum}
	k^{\mathrm{qtm}}(n,x;\lambda,\rho,v,N;q)=K_{x}^{\mathrm{qtm}}(n; \hat{p},N;q^2),
	\end{align}
	where 
	\[
		\hat{p} = v^{-1} q^{\rho-\lambda-N-1}.
	\] 
	\item The affine $q$-Krawtchouk polynomials \cite[\S14.16]{KLS} are defined by 
	\[
	K^{\text{aff}}_n(x;p',N;q) = \rphis{3}{2}{q^{-n},0,q^{-x}}{p'q,q^{-N}}{q,q}.
	\]
	These are related to the quantum $q$-Krawtchouk polynomials by
	\[
	K_n^{\text{aff}}(x;p',N;q^{-1})= \frac{1}{(q/p';q)_n}K_n^{\text{qtm}}(N-x;p'^{-1},N;q).
	\]
	We define the $1$-site duality function
	\begin{align}\label{eq:1sitedualityAffine}
	k^\mathrm{aff}(n,x;\lambda,\rho,v,N;q)= c_{\mathrm k}^\mathrm{aff}(x;\lambda,\rho,v,N;q)K_{x}^{\mathrm{aff}}(n;p',N;q^2),
	\end{align}
	where $p' = v q^{\rho+\lambda-N-1}$ and the coefficient $c_{\mathrm k}^\mathrm{aff}$ can be found in Appendix \ref{app:overview}.
\end{enumerate}
\ \\
As we take appropriate limits in the duality relation \eqref{eq:duality equation dynamic ASEP}, we obtain the following duality functions which correspond to the dualities given in Table \ref{tab:limits_duality_dynamic_ASEP}.
\pagebreak
\begin{proposition}\* \label{prop:degenerate duality}
	\begin{enumerate}[label=(\roman*)]
		\item Define the function $P^v_{\mathsf R}:X \times X \to \R$ by
		\[
			P^v_{\mathsf R}(\eta,\xi) = \lim_{\lambda \to \infty} q^{2\lambda|\eta|} R^{vq^{-\lambda}}(\eta,\xi).
		\]
		Then $P^v_{\mathsf R}$ is a duality function between $\ASEP{q,\vec{N}}\leftrightarrow\RASEP{q,\vec{N},\rho}$ and
		\[
			P^v_{\mathsf R}(\eta,\xi) =\prod_{k=1}^M p(\eta_k,\xi_k;h^{-}_{k-1,0}(\eta),h^{+}_{k+1}(\xi),v,N_k;q).
		\]
		\item Define the function $K^v_\mathrm{qtm}:X \times X \to \R$ by 
		\[
			K^v_\mathrm{qtm}(\eta,\xi) = 	\lim_{\rho \to -\infty} \frac{ (v^{-2}q^{-2\rho})^{|\eta|} }{c^{v}(|\eta|,|\xi|;0,0)C^v(|\eta|,|\xi|;0,2\rho)} P^{vq^\rho}_{\mathsf R}(\eta,\xi),
		\]
		where $c^v$ can be found in \eqref{eq:invariant c def}.
		Then $K^v_\mathrm{qtm}$ is a self-duality function for $\ASEP{q,\vec{N}}$ and
		\[ 
			K^v_\mathrm{qtm}(\eta,\xi)	= \prod_{k=1}^M k^\mathrm{qtm}(\eta_k,\xi_k;h^{-}_{k-1,0}(\eta),h^{+}_{k+1,0}(\xi),v,N_k;q).
		\]
		\item Define the function $K^v_\mathrm{aff}:X \times X \to \R$ by 
		\[
			K^v_\mathrm{aff}(\eta,\xi) = \lim_{\rho \to \infty} \frac{ q^{2\rho|\eta|}  }{c^{-v}(|\eta|,|\xi|;0,2\rho)} P^{-v q^{-\rho}}_{\mathsf R}(\eta,\xi).
		\]
		Then $K^v_\mathrm{aff}$ is a duality function between $\ASEP{q,\vec{N}}\leftrightarrow \ASEP{q^{-1},\vec{N}}$ and
		\[
			K^v_\mathrm{aff}(\eta,\xi) =  \prod_{k=1}^M k^\mathrm{aff}(\eta_k,\xi_k;h^{-}_{k-1,0}(\eta),h^{+}_{k+1,0}(\xi),v,N_k;q).
		\]
	\end{enumerate}
\end{proposition}
\begin{remark}
	For completeness, the duality relations of Proposition \ref{prop:degenerate duality} are given by
	\begin{enumerate}[label=(\roman*)]
		\item 	$\displaystyle[L_{q,\vec{N}} P^v_{\mathsf R}(\,\cdot\,,\xi)](\eta) = [L_{q,\vec{N},\rho}^{\mathsf{R}} P^v_{\mathsf R}(\eta,\,\cdot\,)](\xi)$,
		\item $\displaystyle [L_{q,\vec{N}}K^v_\mathrm{qtm}(\,\cdot\,,\xi)](\eta) = [L_{q,\vec{N}}K^v_\mathrm{qtm}(\eta,\,\cdot\,)](\xi) $,
		\item $\displaystyle [L_{q,\vec{N}}K^v_\mathrm{aff}(\,\cdot\,,\xi)](\eta) = [L_{q^{-1},\vec{N}}K^v_\mathrm{aff}(\eta,\,\cdot\,)](\xi) $.		
	\end{enumerate}
\end{remark}
\begin{remark} \*\label{rem:HahnDuality}
	Similar to the second item of Remark \ref{rem:duality racah}, the function
	\[
	P'_{\mathsf R}(\eta,\xi) = \prod_{k=1}^M p'(\eta_k,\xi_k;h^{-}_{k-1,0}(\eta),h^{+}_{k+1}(\xi),v,N_k;q)
	\]
	with 
	\[
	\begin{split}
		p'(x,y;\lambda,\rho,v,N;q) &= v^x q^{-x(x+\rho+\lambda-N)}(-vq^{\rho+\lambda-N+1};q^2)_x \\ & \quad \times P_y(x;-vq^{\rho+\lambda-N-1}, v^{-1}q^{\rho-\lambda-N-1},N;q^2)
	\end{split}
	\]
	is again a duality function between $\ASEP{q,\vec{N}}$ and $\RASEP{q,\vec{N},\rho}$.
	For $N_1=\ldots=N_M=1$ the $1$-site duality function $p'(\eta_k,\xi_k)$ equals $1$ for $\eta_k=0$ and for $\eta_k=1$ we obtain
	\[
	\begin{split}
		&p'(1,\xi_k;h^{-}_{k-1,0}(\eta),h^{+}_{k+1}(\xi),v,1;q) =\\ 
		&vq^{-h^{+}_{k+1}(\xi)-h^{-}_{k-1,0}(\eta)}
		\left( 1+vq^{h^{+}_{k+1}(\xi)+h^{-}_{k-1,0}(\eta)} + \frac{q^2}{1-q^2}(1-q^{-2\xi_k})(1+q^{2h^{+}_{k+1}(\xi)+2\xi_k-2})\right).
	\end{split}
	\]
	This function is different from the duality function $Z_{n;q,\alpha}(\vec{x},\vec{s})$ from equation (2.1) in \cite{BoCo}. It would be interesting to know how they relate. Unfortunately, we didn't succeed in finding this connection.
\end{remark}
The orthogonality relations from Theorem \ref{Thm:orthogonaldualityLASEPRASEP} are still valid after taking the limits in the previous proposition. Hence, the duality functions $P^v_\mathsf{R}$, $K_\mathrm{qtm}^v$, and $K_\mathrm{aff}^v$ are still orthogonal with respect to the reversible measures $w$ and/or $W_\mathsf{R}$ (multiplied by some factor depending on the total number of particles).
\begin{proposition}\label{prop:degenerate orthogonality}\* 
	\begin{enumerate}[label=(\roman*)]
		\item The function $P^v_\mathsf{R}$ is an orthogonal duality function between $\ASEP{q,\vec{N}}\leftrightarrow\RASEP{q,\vec{N},\rho}$ and
		\[
		\begin{split}
			\sum_{\eta \in X} P_{\mathsf R}^v(\eta,\xi) P^v_{\mathsf R}(\eta,\xi')\, \om^{\mathrm p}(|\eta|) w(\eta;\vec{N};q) = \frac{ \delta_{\xi,\xi'}}{\om^{\mathrm p}_\mathsf{R}(|\xi|)W_\mathsf{R}(\xi;\vec{N},\rho;q)},\\
			\sum_{\xi\in X} P^v_{\mathsf R}(\eta,\xi) P^v_{\mathsf R}(\eta',\xi) \,\om^{\mathrm p}_\mathsf{R}(|\xi|)W_{\mathsf R}(\xi;\vec{N},\rho;q) = \frac{ \delta_{\eta,\eta'}}{\om^{\mathrm p}(|\eta|)w(\eta;\vec{N};q) }.
		\end{split}
		\]
		\item The function $K^v_\mathrm{qtm}$ is an orthogonal self-duality function for $\ASEP{q,\vec{N}}$ and
		\[
		\begin{split}
			\sum_{\eta\in X} K^v_\mathrm{qtm}(\eta,\xi) K_\mathrm{qtm}^v(\eta,\xi') \, \om^{\mathrm{qtm}}(|\eta|) w(\eta;\vec{N};q) = \frac{\de_{\xi,\xi'}}{\om_\mathsf{R}^\mathrm{qtm}(|\xi|) w(\xi;\vec{N};q)},\\
			\sum_{\xi\in X} K^v_\mathrm{qtm}(\eta,\xi) K^v_\mathrm{qtm}(\eta',\xi) \, \om_\mathsf{R}^\mathrm{qtm}(|\xi|) w(\xi;\vec{N};q) = \frac{\de_{\xi,\xi'}}{\om^\mathrm{qtm}(|\eta|) w(\eta;\vec{N};q)}.
		\end{split}
		\]
		\item The function $K^v_\mathrm{aff}$ is an orthogonal duality function between $\ASEP{q,\vec{N}}\leftrightarrow\ASEP{q^{-1},\vec{N}}$ and
		\[
		\begin{split}
			\sum_{\eta \in X}K^v_\mathrm{aff}(\eta,\xi) K^v_\mathrm{aff}(\eta,\xi')\, \om^\mathrm{aff}(|\eta|) w(\eta;\vec{N};q) = \frac{\de_{\xi,\xi'}}{ \om_\mathsf{R}^\mathrm{aff}(|\xi|) w(\xi;\vec{N};q^{-1})},\\
			\sum_{\xi \in X}K^v_\mathrm{aff}(\eta,\xi) K^v_\mathrm{aff}(\eta',\xi)\, \om_\mathsf{R}^\mathrm{aff}(|\xi|) w(\xi;\vec{N};q^{-1})  = \frac{\de_{\eta,\eta'}}{ \om^\mathrm{aff}(|\eta|) w(\eta;\vec{N};q)}.
		\end{split}
		\]
	\end{enumerate}
	The coefficients $\om$ in front of the reversible measures $w$ and $W_\mathsf{R}$ can be found in Appendix \ref{app:overview}.
\end{proposition}
\begin{remark} 
	The duality function $K^v_\mathrm{qtm}$ from Proposition \ref{prop:degenerate duality}(ii) and its orthogonality relations in \ref{prop:degenerate orthogonality}(ii) were first found in \cite{CFG}. In the present paper we have an alternative proof for \cite[Theorem 3.2]{CFG} by exploiting that it is a degenerate version of the orthogonal duality between $\LASEP{q,\vec{N},\lambda}\leftrightarrow\RASEP{q,\vec{N},\rho}$.
\end{remark}
The product for the duality function $P^v_{\mathsf R}$ in Proposition \ref{prop:degenerate duality}(i) has a `doubly nested' structure, similar to the product structure of the duality function $R^v$ in Theorem \ref{Thm:KrawtchoukRacah}. This is different from the simpler nested structure of the duality function $K_{\mathsf R}$ defined by \eqref{eq:duality function K}, which is a duality function between the same two processes as $P^v_{\mathsf R}$. The simpler function $K_{\mathsf R}$ can be recovered from $P^v_{\mathsf R}$ by taking an appropriate limit involving the free parameter $v$:
\[
\begin{split}
	\lim_{v\to 0} v^{-|\eta|} P^v_{\mathsf R}(\eta,\xi) &  = q^{-\sum_{k=1}^M\eta_k[\eta_k+h^{+}_{k+1}(\xi)+\sum_{j=1}^{k-1}(2\eta_j-N_j)-N_k]} \\
	& \quad \times \prod_{k=1}^M \rphis{3}{2}{q^{-2\eta_k}, q^{-2\xi_k}, -q^{2h^{+}_{k+1}(\xi) + 2\xi_k -2N_k} }{q^{-2N_k},0}{q^2,q^2} \\
	& = (-1)^{|\eta|} q^{-|\eta|(|\eta|-1)} K_{\mathsf R}(\eta,\xi).
\end{split}
\]
We can also obtain the triangular self duality functions for ASEP from \cite{CGRS}. In \cite[Remark 6.2]{CFG} they were obtained by taking the limit $v \to 0$ in the duality function $K^v_\mathrm{qtm}$. Here we show they can also be obtained as a $\rho \to -\infty$ limit, i.e.~the limit that removes the `dynamic part' of dynamic ASEP, of the duality function $K_{\mathsf R}$ given by \eqref{eq:duality function K}. Since there are no free parameters involved in this limit, this shows that $K_{\mathsf R}$ is in a sense the `triangular' duality function between ASEP and ASEP$_{\mathsf R}$. We remark that the orthogonality relations do not survive the limit. Furthermore, letting $\rho \to \infty$ also gives a triangular duality function between $\ASEP{q,\vec{N}}$ and $\ASEP{q^{-1},\vec{N}}$. 
\pagebreak
\begin{proposition} \label{prop:KR->triangular duality}\*
	\begin{enumerate}[label = (\roman*)]
		\item The function $D:X \times X \to \R$ given by
		\[
		\begin{split}
			D(\eta,\xi) &= \lim_{\rho \to -\infty} q^{-\rho|\eta|} K_{\mathsf R}(\eta,\xi) \\
			&=q^{-\half u(\eta;\vec{N})} \prod_{k=1}^L \frac{(q^{-2\xi_k};q^2)_{\eta_k} }{(q^{-2N_k};q^2)_{\eta_k} }q^{\eta_k(\eta_k+2\xi_k-\frac32 N_k +\frac12 + \sum_{j=k+1}^m 2\xi_j - N_j)} 1_{\eta_k\leq\xi_k}
		\end{split}
		\]
		is a self-duality function for $\ASEP{q,\vec{N}}$, i.e.~$[L_{q,\vec{N}}D(\,\cdot\,,\xi)](\eta) = [L_{q,\vec{N}}D(\eta, \,\cdot\,)](\xi)$.
		\item The function $D':X \times X \to \R$ given by
		\[
		\begin{split}
			D'(\eta,\xi) &= \lim_{\rho \to \infty} (-q^{\rho})^{|\eta|} K_{\mathsf R}(\eta,\xi) \\
			& =q^{-\half u(\eta;\vec{N})} \prod_{k=1}^L \frac{(q^{2\xi_k-2N_k};q^2)_{\eta_k} }{(q^{-2N_k};q^2)_{\eta_k} }q^{\eta_k(\eta_k-2\xi_k-\frac32 N_k +\frac12 - \sum_{j=k+1}^m 2\xi_j - N_j)} 1_{\eta_k\leq N_k-\xi_k}
		\end{split}
		\]
		is a duality function between $\ASEP{q,\vec{N}}$ and $\ASEP{q^{-1},\vec{N}}$, i.e.
		\[[L_{q,\vec{N}}D(\,\cdot\,,\xi)](\eta) = [L_{q^{-1},\vec{N}}D(\eta, \,\cdot\,)](\xi).\]
	\end{enumerate}
	
\end{proposition}

\subsection{Duality for the totally asymmetric zero range process}
Finally, we consider also briefly the limit of \eqref{eq:duality equation dynamic ASEP} where the number $N_k$ of allowed particles per site tends to $\infty$. We set $N_1=\ldots=N_M=N$ and let $N\to \infty$. For the dynamic ASEP jump rates we have
\[
(q^{-1}-q)\lim_{N \to \infty}   C_k^{\mathsf R,+}(\xi) = \frac{ 1-q^{2\xi_k} }{1-q^2} \qquad \text{and} \qquad \lim_{N\to \infty}  C_{k+1}^{\mathsf R,-}(\xi) = 0,
\]
in which we recognize the jump rates for the totally asymmetric zero range process $q$-TAZRP, see e.g.~\cite[Section 2]{BoCoSa}. This is the continuous-time Markov jump process on the state space $(\Z_{\geq 0})^M$ where particles can only jump to the right, with generator given by
\[
[L_q^{\mathsf R}f](\xi) =\sum_{k=1}^{M-1} \frac{ 1-q^{2\xi_k} }{1-q^2} [f(\xi^{k,k+1})-f(\xi)].
\]
Note that the dynamic parameter $\rho$ vanishes in this limit.

Similarly, we have
\[
\lim_{N \to \infty} C_k^{\mathsf L,+}(\zeta) = 0 \qquad \text{and} \qquad (q^{-1}-q)\lim_{N\to \infty}  C_{k+1}^{\mathsf L,-}(\zeta) = \frac{ 1-q^{2\zeta_{k+1}} }{1-q^2},
\]
so that we obtain another totally asymmetric zero range process. In this case, particles can only jump to the left, and the generator is given by
\[
[L_q^{\mathsf L}f](\zeta) =\sum_{k=1}^{M-1} \frac{ 1-q^{2\zeta_{k+1}} }{1-q^2} [f(\zeta^{k+1,k})-f(\zeta)].
\]
So we have
\[
(q^{-1}-q)\lim_{N\to \infty} L^{\mathsf R}_{q,\vec{N},\rho} = L^{\mathsf R}_{q} \qquad \text{and} \qquad  (q^{-1}-q)\lim_{N\to \infty} L^{\mathsf L}_{q,\vec{N},\lambda} = L^{\mathsf L}_{q}.
\]
By taking an appropriate limit of the duality function $R^v$, we get the following known duality for $q$-TAZRP (see e.g. \cite{BoCoSa} and \cite{CGRS}).
\begin{proposition} \label{prop:duality q-TAZRP}
	The function $D:(\Z_{\geq0})^M \times (\Z_{\geq0})^M \to \R$ given by
	\[
	D(\zeta,\xi) =  v^{-|\xi|} q^{-|\xi|(|\xi|+\rho+\la)}\lim_{N\to \infty} \frac{ q^{|\xi|MN}}{c^v(|\zeta|,|\xi|;\lambda,\rho)} R^v(\zeta,\xi) = \prod_{k=1}^M q^{2\xi_k \sum_{j=1}^{k}\zeta_j},
	\]
	where $c^v$ is given by \eqref{eq:invariant c def}, satisfies $[L_q^{\mathsf L} D(\, \cdot\, , \xi)](\zeta) = [L_q^{\mathsf R} D(\zeta,\,\cdot\,)](\xi)$.
\end{proposition}

\section{Symmetric degenerations} \label{sec:symmetric degenerations}
Next, we consider the $q \to 1$ limit of the duality equation \eqref{eq:duality equation dynamic ASEP}, Proposition \ref{prop:degenerate duality} and Proposition \ref{prop:degenerate orthogonality}. The duality functions and measures we obtain in this section are similar to the asymmetric case since in the $q\to1$ limit factors of the form $q^A$ disappear and the $q$-shifted factorials become ordinary shifted factorials using
\[
\lim\limits_{q\to 1} \frac{(q^{2a};q^2)_n}{(1-q^2)^n}=\big(a\big)_n.
\]
We have put a hat on the duality functions and measures without a parameter $q$ to distinguish between them and their counterparts that do depend on $q$.\\

Let us first consider to $q\to1$ limit of (nondynamic) $\ASEP{q,\vec{N}}$. This gives the well known generalized symmetric simple exclusion process $\SEP{\vec{N}}$, which is the continuous-time Markov process with state space $X$ where particles jump from site $k$ to $k+1$ at rate
\[
\lim_{q \to 1} c_k^+(\eta) = \eta_k(N_{k+1}-\eta_{k+1}),
\]
and interchanging $k$ and $k+1$ gives the rate for jumps from site $k+1$ to $k$, 
\[	
\lim_{q \to 1} c_k^-(\eta) = \eta_{k+1}(N_{k}-\eta_{k}).
\]
We denote the corresponding Markov generator by $L_{\vec{N}}$. The dynamic ASEP jump rates $\mathsf C_{k}^{\mathsf R,\pm}$ from Definition \ref{Def:dynamic ASEP} are of the form $c_k^{\pm} \cdot d_k^{\pm}$, where $c_k^\pm$ are the ASEP jump rates. Assuming $\rho\in \R$ the factors $d_k^{\pm}$ satisfy $\lim_{q \to 1} d_{k}^\pm =1$, so the limit of $\RASEP{q,\vec{N},\rho}$ is again $\SEP{\vec{N}}$, and the dynamic parameter vanishes in this limit. Clearly, for $\lambda \in \R$ the $q \to 1$ limit of $\LASEP{q,\vec{N},\lambda}$ is also $\SEP{\vec{N}}$. So for $\rho,\lambda \in \R$,
\[
\lim_{q \to 1} L^{\mathsf R}_{q,\vec{N},\rho} = L_{\vec{N}} \qquad \text{and} \qquad \lim_{q \to 1} L^{\mathsf L}_{q,\vec{N},\lambda} = L_{\vec{N}}.
\]
By taking the limit $q \to 1$ of the $\ASEP{q,\vec{N}}$ reversible measure $w$ we obtain the well-known reversible measure for $\SEP{\vec{N}}$ as a product of binomials,
\[
\hat{w}(\eta;\vec{N}) = \lim_{q \to 1} w(\eta;\vec{N};q) = \prod_{k=1}^M \binom{N_k}{\eta_k}.
\]
Taking the limit of the reversible measure for $\RASEP{q,\vec{N},\rho}$ or $\LASEP{q,\vec{N},\la}$ gives essentially the same result,
\[
\lim_{q \to 1} W_{\mathsf R}(\eta;\vec{N},\rho;q) = \lim_{q\to 1} W_{\mathsf L}(\eta;\vec{N},\la;q) = 2^{-|\vec{N}|} \hat{w}(\eta;\vec{N}).
\]
Moreover, for $\rho=\la=\frac12\log_q(v)$ with $v>0$, i.e.~$q^{2\rho}=v$, we have
\[
\lim_{q \to 1} W_{\mathsf R}(\eta;\vec{N},\rho;q) = \lim_{q\to 1} W_{\mathsf L}(\eta;\vec{N},\la;q) = p_v^{|\eta|} (1-p_v)^{|\vec{N}|-|\eta|} \hat w(\eta;\vec{N}), \qquad p_v = \frac{1}{1+v}.
\]
\medskip

\subsection{Generalized dynamic SSEP}
To obtain a non-trivial limit of the factor $d_k^{\pm}$, we replace $q^{2\rho}$ by $-q^{2\rho}$. This can be done by substituting $\rho \mapsto \rho + \pi i / 2\ln(q)$, where $i=\sqrt{-1}$. We then get, using the rewriten rates \ref{eq:rewrittenrates},
\begin{align*}
\begin{split}
	C^{\mathsf R,+}_k(\xi) &= [\xi_k]_q[N_{k+1}-\xi_{k+1}]_q\frac{[h_k^+-\xi_k]_q[h_{k+1}^+-\xi_{k+1}]_q}{[h_{k+1}^+]_q[h_{k+1}^++1]_q},\\
	C^{\mathsf R,-}_{k}(\xi) &= [\xi_k]_q[N_{k-1}-\xi_{k-1}]_q\frac{[h_k^++\xi_{k-1}]_q[h_{k+1}^++\xi_{k}]_q}{[h_{k}^+]_q[h_{k}^+-1]_q}.
\end{split}	
\end{align*}
If we now take the limit $q\to1$, we obtain a dynamic version of the symmetric exclusion process. In a similar way, we obtain a dynamic version of SSEP from ASEP$_{\mathsf L}$.  We impose that the rates stay nonnegative, for example by requiring $|\rho|,|\lambda|> |\vec{N}|$.
\begin{Definition} \label{Def:dynamic SEP} \*
	\begin{enumerate}[label = (\roman*)]
		\item $\RSEP{\vec{N},\rho}$ is a continuous-time Markov jump process on the state space $X$ depending on a parameter $\rho \in \R$. Given a state $\xi=(\xi_k)_{k=1}^M \in X$ we define the height function $(h^{+}_{k+1}(\xi))_{k=1}^M$ as before by
		\[
		h^{+}_{k,\rho}(\xi)=h^{+}_k(\xi) = \rho + \sum_{j=k}^M (2\xi_j - N_j).
		\]
		Then a particle on site $k$ jumps to site $k+1$ at rate
		\[
		\frac{\xi_k(N_{k+1}-\xi_{k+1})(h^{+}_k(\xi)-\xi_k)(h^{+}_{k+1}(\xi)-\xi_{k+1})}{h^{+}_{k+1}(\xi)(h^{+}_{k+1}(\xi)+1)},
		\]
		and a particle on site $k$ jumps to site $k-1$ at rate
		\[
		\frac{\xi_{k}(N_{k-1}-\xi_{k-1})(h^{+}_{k}(\xi)+\xi_{k-1})(h^{+}_{k+1}(\xi)+\xi_{k})}{h^{+}_k(\xi)(h^{+}_k(\xi)-1)}.
		\]
		\item $\LSEP{\vec{N},\lambda}$ is a continuous-time Markov jump process on the state space $X$ depending on a parameter $\lambda \in \R$. Given a state $\zeta=(\zeta_k)_{k=1}^M \in X$ we define the height function $(h^{-}_{k}(\zeta))_{k=1}^M$ as before by
		\[
		h^{-}_{k}(\zeta) = \lambda + \sum_{j=1}^{k} (2\xi_j - N_j).
		\]
		Then a particle on site $k$ jumps to site $k+1$ at rate
		\[
		\frac{\zeta_k(N_{k+1}-\zeta_k)(h^{-}_{k-1}(\zeta)+\zeta_k)(h^{-}_{k}(\zeta)+\zeta_{k+1})}{h^{-}_{k}(\zeta)(h^{-}_{k}(\zeta)-1)},
		\]
		and a particle on site $k$ jumps to site $k-1$ at rate
		\[
		\frac{\zeta_{k}(N_{k-1}-\zeta_{k-1})(h^{-}_{k-1}(\zeta)-\zeta_{k-1})(h^{-}_{k}(\zeta)-\zeta_{k})}{h^{-}_{k-1}(\zeta)(h^{-}_{k-1}(\zeta)+1)}.
		\]
	\end{enumerate}
\end{Definition} 
\begin{remark}\*
	\begin{itemize}
		\item For $N_1=\ldots=N_M=1$ $\RSEP{\vec{N},\rho}$ is the dynamic symmetric simple exclusion process defined in \cite[Definition 9.4]{Bo} resticted to $M$ sites.
		\item In the limit $\rho \to \pm \infty$ we recover $\SEP{\vec{N}}$ from $\RSEP{\vec{N},\rho}$. Similarly, by letting $\lambda \to \pm \infty$ we recover $\SEP{\vec{N}}$ from $\LSEP{\vec{N},\lambda}$.
	\end{itemize}
\end{remark}

\subsection{Duality between SSEP$_\mathsf{L}$ and SSEP$_\mathsf{R}$}
Let us now consider the corresponding $q \to 1$ limit of the $q$-Racah duality functions, which leads to duality functions in terms of Racah polynomials. These are hypergeometric orthogonal polynomials given by
\[
\hat{R}_n(x;\alpha,\beta,\gamma,\delta) = \rFs{4}{3}{-n,n+\alpha+\beta+1,-x,x+\gamma+\delta+1}{\alpha+1,\beta+\delta+1,\gamma+1}{1},
\]
which can be obtained as the $q \to 1$ limit of the $q$-Racah polynomials $R_n(x;\alpha,\beta,\gamma,\delta;q)$ in case $\alpha,\beta,\gamma,\delta \in \R$. We substitute $(v,\rho,\lambda) \mapsto (q^v,\rho+i\pi/2\ln(q),\lambda+\pi i/2\ln(q))$ in the duality function $R^v$ given by \eqref{eq:duality Racah polynomials}. We then let $q \to 1$ to obtain duality functions for SSEP$_\mathsf{L}$ and SSEP$_\mathsf{R}$ as a nested product of Racah polynomials, see the result in Proposition \ref{prop:duality Racah} below. From the reversible measures for $\LASEP{q,\vec{N},\lambda}$ and $\RASEP{q,\vec{N},\rho}$ we obtain the following reversible measures for $\LSEP{\vec{N},\lambda}$ and $\RSEP{\vec{N},\rho}$,
\begin{equation} \label{eq:WR WL q->1}
	\begin{split}
		\hat{W}_{\mathsf L}(\zeta;\vec{N},\lambda) &= \lim_{q \to 1} (1-q^2)^{|\vec{N}|} W_{\mathsf L}(\zeta;\vec{N},\lambda+i\pi/2\ln(q);q) = \prod_{k=1}^{M} \hat{W}(\zeta_k;N_k,h^{-}_{k-1}(\zeta)), \\
		\hat{W}_{\mathsf R}(\xi;\vec{N},\rho) &= \lim_{q \to 1} (1-q^2)^{|\vec{N}|} W_{\mathsf R}(\zeta;\vec{N},\rho+i\pi/2\ln(q);q) = \prod_{k=1}^{M} \hat{W}(\xi_k;N_k,h^{+}_{k+1}(\xi)),
	\end{split}
\end{equation}
where the `1-site' weight function $\hat{W}$ is given by
\[
\begin{split}
	\hat{W}(x;N,\rho)&= \lim_{q \to 1} (1-q^2)^N W(x;N,\rho+i\pi/2\ln(q);q)\\
	& = \frac{2x+\rho-N}{\rho-N} \frac{ (\rho-N)_x}{ (\rho+1)_x (-\rho)_N}\binom{N}{x} .
\end{split}
\]
The Racah duality functions $\hat{R}^v$ in the proposition below are orthogonal with respect to these reversible measures, which follows from letting $q \to 1$ in the orthogonality relations for the $q$-Racah duality functions from Theorem \ref{Thm:orthogonaldualityLASEPRASEP}, see Section \ref{sec:calculations of limits} for the details of the proof.
\begin{proposition} \label{prop:duality Racah}
	Define the function $\hat{R}^v:X \times X \to \R$ by
	\[
	\hat{R}^v(\zeta,\xi) = \lim_{q \to 1} (1-q^2)^{-|\vec{N}|} R^v(\zeta,\xi) = \prod_{k=1}^M \hat{r}(\zeta_k,\xi_k;h^{-}_{k-1}(\zeta),h^{+}_{k+1}(\xi),v,N_k),
	\]
	with the $1$-site duality function
	\[
	\begin{split}
		\hat{r}(y,x;\lambda,\rho,v,N) &= \frac{(\frac12(\rho+\lambda-N+v+1))_x (\frac12(\rho+\lambda-N-v+1))_y (y+\frac12(\lambda-\rho-N+v+1))_N }{ (-x+\frac12(\lambda-\rho+N+v+1))_{x+y} } \\
		&\qquad \times  \hat{R}_x(y;\tfrac12(\rho+\lambda-N-v-1), \tfrac12(\rho-\lambda-N+v-1), -N-1,\lambda).
	\end{split}
	\]
	Then $\hat{R}^v$ is a duality function between $\LSEP{\vec{N},\lambda}$ and $\RSEP{\vec{N},\rho}$, i.e.
	\[ 
		[L_{\vec{N},\lambda}^{\mathsf L} R(\, \cdot\,,\xi)](\zeta) = [L_{\vec{N},\rho}^{\mathsf R} R(\zeta,\, \cdot\,)](\xi).
	\]
	Furthermore, we have the following orthogonality relations,
	\[
	\begin{split}
		\sum_{\zeta \in X} \hat{R}^v(\zeta,\xi) \hat{R}^{-v}(\zeta,\xi') \hat{W}_{\mathsf L}(\zeta;\vec{N},\lambda) = \frac{ \delta_{\xi,\xi'}}{\hat{W}_\mathsf{R}(\xi;\vec{N},\rho)},\\
		\sum_{\xi\in X} \hat{R}^v(\zeta,\xi) \hat{R}^{-v}(\zeta',\xi) \hat{W}_{\mathsf R}(\xi;\vec{N},\rho) = \frac{ \delta_{\zeta,\zeta'}}{\hat{W}_\mathsf{L}(\zeta;\vec{N},\lambda)}.
	\end{split}
	\]
\end{proposition}

\subsection{Degenerations}
Similar to the asymmetric case, we will consider the dualities listed in the table below. Note that in the $q\to1$ limit, the processes $\ASEP{q,\vec{N}}$ and $\ASEP{q^{-1},\vec{N}}$ both become $\SEP{\vec{N}}$.
\begin{table}[h]
	\begin{tabular}{|l|ll|}
		\hline
		\ &\multicolumn{2}{|l|}{\hspace{0.35cm} Duality between}\\ \hline
		(i)&\multicolumn{1}{|l|}{$\SEP{\vec{N}}$}       &  $\RSEP{\vec{N},\rho}$                \\ \hline
		(ii)&\multicolumn{1}{|l|}{$\SEP{\vec{N}}$ }      &  $\SEP{\vec{N}}$                \\ \hline
	\end{tabular}
\end{table}\\
There are two routes for obtaining these dualities at this point. In this paper, we take the $q\to 1$ limit from the dualities from Proposition \ref{prop:degenerate duality} and Theorem \ref{Thm:dualityASEPRASEP}. An equivalent way would be taking limits from the duality between SSEP$_\mathsf{L}$ and SSEP$_\mathsf{R}$ shown in Proposition \ref{prop:duality Racah}.\\
\\
So let us consider the $q \to 1$ limit of the $q$-Hahn and $q$-Krawtchouk\footnote{In this paper, we have several choices for duality functions to obtain a self-duality function for SSEP as a limit. We choose the function $K_{\mathsf{R}}$, but others are equally valid.} duality functions $P_{\mathsf R}$ and $K_\mathsf{R}$. To make these limits convergent, we again need a factor in front that only depends on the parameters and total (dual) particles. Not surprisingly, we will end up with Hahn and Krawtchouk polynomials. 
\begin{enumerate}[label=(\roman*)]
	\item 	The Hahn polynomials \cite[\S9.5 and \S9.6]{KLS} are defined by
			\[
				\hat{P}_n(x;\alpha,\beta,N) = \rFs{3}{2}{-n,n+\alpha+\beta+1,-x}{\alpha+1,-N}{1}.
			\]
			We define the 1-site duality functions
			\[
			\begin{split}
				\hat{p}(n,x;\lambda,\rho,v,N) & = \frac{ (\frac12(\rho+\lambda-N+v+1))_{x} (y+\frac12(\lambda-\rho-N+1+v))_{N} }{ (-x+\frac12(\lambda-\rho+N+1+v))_{x+n} } \\
				& \qquad \times \hat{P}_x(n;\tfrac12(\rho+\lambda-N+v-1),\tfrac12(\rho-\lambda-N-v-1),N).
			\end{split}
			\]
	\item 	The Krawtchouk polynomials \cite[\S9.11]{KLS} are defined by
			\[
				\hat{K}_n(x;p,N) = \rFs{2}{1}{-n,-x}{-N}{\frac1p}.
			\]
			We define the 1-site duality functions
			\[
				\hat{k}(n,x;v,N)= v^{\frac{n}{2}}\hat{K}_{n}(x;\tfrac{1}{1+v},N_k).
			\]
\end{enumerate}
In the $q\to1$ limit of the $q$-Hahn polynomials, we first substitute $(v,\rho) \mapsto (iq^{v},\rho+i\pi/2\ln(q))$ in $P_{\mathsf R}$. The $q$-Krawtchouk polynomials $K_\mathsf{R}$ depend, besides $q$, on the dynamic parameter $\rho$, but not on other extra parameters. We substitute $\rho \mapsto \frac12\log_q(v)$ for $v>0$, i.e.~$q^{2\rho} \mapsto v$, and then let $q \to 1$. In this way, the dynamic parameter $\rho$ ends up in the duality function as just a free parameter. We then obtain the following result similar to Proposition \ref{prop:degenerate duality}, for which the proof can be found in Section \ref{sec:calculations of limits}.
\begin{proposition}\* \label{prop:degenerate duality sym}
	\begin{enumerate}[label=(\roman*)]
		\item Define the function $\hat{P}_{\mathsf R}^v:X \times X \to \R$ by
		\[
		\hat{P}^v_{\mathsf{R}}(\eta,\xi) = \lim_{q \to 1} (1-q^2)^{|\zeta|-|\vec{N}|} P^{iq^v}_{\mathsf{R}}(\eta,\xi),
		\]
		where $\rho$ is replaced by $\rho+i\pi/2\ln(q)$ in $P^{iq^v}_{\mathsf{R}}$.
		Then $\hat{P}_\mathsf{R}$ is a duality function between $\SEP{\vec{N}}\leftrightarrow\RSEP{\vec{N},\rho}$ and 
		\[
		\hat{P}^v_{\mathsf{R}}(\eta,\xi)= \prod_{k=1}^M \hat{p}(\eta_k,\xi_k;h^{+}_{k+1}(\xi),h^{-}_{k-1,0}(\eta),v,N_k).
		\]
		\item Define the function $\hat{K}^v:X \times X \to \R$ by
		\[
		\hat{K}^v(\eta,\xi) = \lim_{q \to 1} (-1)^{|\eta|} K_{\mathsf R}(\eta,\xi),
		\]
		where $\rho=\tfrac12\log_q(v)$ in $K_\mathsf{R}$. Then for $v>0$, $\hat{K}^v$ is a self-duality function for $\SEP{\vec{N}}$ and 
		\[
		\hat{K}^v(\eta,\xi)= \prod_{k=1}^M \hat{k}(\eta_k,\xi_k;v,N_k).
		\]
	\end{enumerate}
\end{proposition}	
For completeness, the duality relations of Proposition \ref{prop:degenerate duality sym} are given by
\begin{enumerate}[label=(\roman*)]
	\item $[L_{\vec{N}} \hat{P}_{\mathsf R}(\, \cdot\,,\xi)](\eta) = [L_{\vec{N},\rho}^{\mathsf R} \hat{P}_{\mathsf R}(\eta,\, \cdot\,)](\xi)$,
	\item $[L_{\vec{N}} \hat{K}^v(\, \cdot\,,\xi)](\eta) = [L_{\vec{N}} \hat{K}^v(\eta,\, \cdot\,)](\xi)$.
\end{enumerate}
\begin{remark}\*
	\begin{itemize}
		\item Let us remark that the same result can be obtained by substituting $v \mapsto v-\la$ in Proposition \ref{prop:duality Racah} and letting $\la \to \infty$.
		\item There are several proofs of the self-duality result of $\SEP{\vec{N}}$, see e.g.~\cite{FrGi2019,ReSau2018, Gr2019}. 
		\item Note that the product $\hat{K}^v$ no longer has a nested structure.
	\end{itemize}
\end{remark}
	By taking the $q\to1$ limit in Proposition \ref{prop:degenerate orthogonality}, we also obtain orthogonality relations for the duality functions $\hat{P}^v_\mathsf{R}$ and $\hat{K}^v$. The proof can again be found in Section \ref{sec:calculations of limits}.
	\begin{proposition}\* \label{prop:degenerate orthogonality sym}
		\begin{enumerate}[label=(\roman*)]
		\item The function $\hat{P}_\mathsf{R}^v$ is an orthogonal duality function between $\SEP{\vec{N}}\leftrightarrow\RSEP{\vec{N},\rho}$ and
		\[
		\begin{split}
			\sum_{\eta \in X} \hat{P}^v_{\mathsf R}(\eta,\xi) \hat{P}^v_{\mathsf R}(\eta,\xi') \, \hat{\om}^{\mathrm p}(|\eta|)\hat{w}(\eta;\vec{N}) = \frac{ \de_{\xi,\xi'}}{\hat{\om}^{\mathrm p}_\mathsf{R}(|\xi|) \hat{W}_{\mathsf R}(\xi;\vec{N},\rho)},\\
			\sum_{\xi \in X} \hat{P}^v_{\mathsf R}(\eta,\xi) \hat{P}^v_{\mathsf R}(\eta',\xi) \, \hat{\om}^{\mathrm p}_\mathsf{R}(|\xi|)\hat{W}_{\mathsf R}(\xi;\vec{N},\rho) = \frac{ \de_{\eta,\eta'}}{\hat{\om}^{\mathrm p}(|\eta|) \hat{w}(\eta;\vec{N})}.
		\end{split}
		\]
		\item The function $\hat{K}^v$ is an orthogonal self-duality function for $\SEP{\vec{N}}$ and 
		\[
		\begin{split}
			\sum_{\eta \in X}& \hat K^v(\eta,\xi) \hat K^v(\eta,\xi') \hat w(\eta;\vec{N}) = \frac{\de_{\xi,\xi'}}{\hat \om^{\mathrm k}(|\xi|) \hat w(\xi;\vec{N})}, \\
			\sum_{\xi \in X}& \hat K^v(\eta,\xi) \hat K^v(\eta',\xi) \hat \om^{\mathrm k}(|\xi|) \hat w(\xi;\vec{N}) = \frac{\de_{\eta,\eta'}}{\hat w(\eta;\vec{N})}.
		\end{split}
		\]
		The explicit expressions of the coefficients $\hat \om$ in front of the reversible measures $\hat w$ and $\hat W_\mathsf{R}$ can be found in Appendix \ref{app:overview}.
	\end{enumerate}
	\end{proposition}

\section{The quantum algebra $\U_q(\mathfrak{sl}_2)$ and $q$-Krawtchouk polynomials}\label{sec:QuantumAlgebra-Krawtchouk}
In this section, we state the necessary properties regarding the quantum algebra $\U_q(\mathfrak{sl}_2)$ and the $q$-Krawtchouk polynomials required for proving the main results of this paper concerning the Markov processes. We will first introduce $\U_q(\mathfrak{sl}_2)$ and give a representation of this algebra which has close connections with ASEP and dynamic ASEP. Then lastly, we will state some recurrence relations for $q$-Krawtchouk polynomials that will be useful later on.
\subsection{The algebra $\U_q(\mathfrak{sl}_2)$}
$\mathcal U_q := \U_q\big(\mathfrak{sl}_2\big)$ is the quantized universal enveloping algebra of the Lie algebra $\mathfrak{sl}_2$. This is the unital, associative, complex algebra generated by $K$, $K^{-1}$, $E$, and $F$, subject to the relations
\begin{align}
	K K^{-1} = 1 = K^{-1}K, \quad KE = qEK, \quad KF= q^{-1}FK, \quad EF-FE =\frac{K^2-K^{-2}}{q-q^{-1}}.\label{eq:UqRelations}
\end{align}
The Casimir element
\begin{equation} \label{eq:Casimir}
	\Om = \frac{q^{-1} K^2 +qK^{-2}-2}{(q^{-1}-q)^2} + EF= \frac{q^{-1}K^{-2}+qK^2-2}{(q^{-1}-q)^2} +FE
\end{equation}
is a central element of $\U_q$, i.e. $\Omega X=X\Omega$ for all $X\in\U_q$.
We use the $*$-structure on $\U_q$ which comes from the Lie algebra $\mathfrak{su}(2)$. This is the anti-linear involution defined on the generators by
\[
K^*=K, \quad E^*=F, \quad F^* = E, \quad (K^{-1})^* = K^{-1}.
\]
Note that the Casimir element is self-adjoint in $\U_q$, i.e.~$\Om^*=\Om$.\\ 

\noindent The comultiplication $\De:\U_q\to\U_q\otimes\U_q$ is a $*$-algebra homomorphism defined on the generators by
\begin{equation} \label{eq:comult}
	\begin{split}
		\De(K) = K \tensor K,&\quad  \De(E)= K \tensor E + E \tensor K^{-1}, \\
		\De(K^{-1}) = K^{-1} \tensor K^{-1},&\quad  \De(F) = K \tensor F + F \tensor K^{-1}.
	\end{split}
\end{equation}
The self-adjoint element $\Delta(\Omega)$ will be the generator of our Markov processes. It follows from \eqref{eq:Casimir} and \eqref{eq:comult} that
\begin{equation} \label{eq:De(OM)}
	\begin{split}
		\De(\Om) = &\frac{1}{(q^{-1}-q)^{2}} \big[ q( K^2 \tensor K^2 )+ q^{-1} (K^{-2} \tensor K^{-2}) -2 (1\tensor 1) \big] \\
		&+ K^2 \tensor FE +KE\tensor FK^{-1} + FK \tensor K^{-1}E + FE \tensor K^{-2}.
	\end{split}
\end{equation}
Another important element in $\U_q$ we need, is the twisted primitive element $Y_\rho$, defined by
\[
Y_\rho = q^{\frac12} EK + q^{-\frac12}FK - [\rho]_q(K^2-1), \quad \rho\in\R.
\]
This satisfies
\begin{align}
	\De(Y_\rho) = K^2 \tensor Y_\rho + Y_\rho \tensor 1, \qquad Y_\rho^* =Y_\rho.\label{eq:coproductYrho}
\end{align}
In Lie algebras, the comultiplication of an element $X$ is defined by $\Delta(X)=1\otimes X + X\otimes 1$. Note that $Y_\rho$ almost satisfies this. The $K^2$ in the above equation will cause the asymmetry of the process.

\subsection{A representation of $\U_q$ related to ASEP} \label{subsec:reprUqASEP} The generator of $\ASEP{q,\vec{N}}$ can be realized by sums of coproducts $\Delta(\Omega)$ in an $M$-fold tensor product representation of $\U_q$. Let us introduce the representations we need. Define $H_k$ to be the $(N_k+1)$-dimensional Hilbert space of (continuous) functions $f:\{0,1,...,N_k\}\to\C$ with inner product induced by the orthogonality measure \eqref{eq:orthokrawtchoukASEP} of the $q$-Krawtchouk polynomials,
\[
\langle f,g \rangle_{H_k} = \sum_{n=0}^{N_k} f(n) \overline{g(n)} w(n;N_k;q)q^{-2nu_k(\vec{N})},
\]
where $w(n;q,N_k)$ can be found in \eqref{eq:orthokrawtchoukASEP} and the factor
\[
	u_k(\vec{N})= -\half N_k + \sum_{j=1}^{k} N_j
\]
is present to prevent needing ground state transformations, as is used in e.g.~\cite{CGRS}, later on. Our duality functions will be elements of the $M$-fold tensor product of $H_k$,
\[
	H=H_1 \otimes H_2 \otimes  \ldots \otimes H_M.
\]
We will interpret $H$ as functions on states $\eta$ in our state space $X= \{0,\ldots,N_1\}\times \cdots \times \{0,\ldots,N_M\}$. Moreover, note that 
\begin{align*}
	u(\eta;\vec{N})= -2\sum_{k=1}^M \eta_k u_k(\vec{N}),
\end{align*}
where $u$ is given by \eqref{eq: u(eta,N}, 
hence the measure from the inner product of $H$, corresponds to $w(\eta;q,\vec{N})$ from \eqref{eq:Weightfunction w}: 
\begin{align}
	\langle f,g\rangle_H = \sum_\eta f(\eta)\overline{g(\eta)}w(\eta;\vec{N};q).\label{eq:HilbertH}
\end{align}

Let $B(H_k)$ be the space of linear operators on $H_k$ and $\pi_k:\U_q\to B(H_k)$ the $*$-representation defined by
\begin{equation} \label{eq:representation}
	\begin{split}
		[\pi_k(K)f](n ) &= q^{n-\frac12 N_k} f(n), \\
		[\pi_k(E)f](n) &= q^{u_k(\vec{N})} [n]_q f(n-1), \\
		[\pi_k(F) f](n) & = q^{-u_k(\vec{N})} [N_k-n]_q f(n+1),\\
		[\pi_k(K^{-1}) f](n) &= q^{\frac12N_k-n}f(n).
	\end{split}
\end{equation}
One can easily verify that this is a $*$-representation, i.e. $\pi_k(X^*)=\pi_k(X)^*$ for all $X\in \U_q$, by checking this for the generators $K,K^{-1},E$ and $F$.\\

Denote by $\pitensor$ the tensor product representation of $\pi_k$ and $\pi_{k+1}$,
\[
\pitensor(X\otimes Y)= \pi_k(X)\otimes \pi_{k+1}(Y),\qquad X,Y\in \U_q.
\] 
A direct calculation shows that the representation $\pitensor$ of $\Delta(\Omega)$ is the generator of $\ASEP{q,\vec{N}}$ for sites $k$ and $k+1$ plus some constant, i.e.
\begin{align*}
	\begin{split}[\pitensor (\Delta(\Omega))f](\eta) =& c^+_k [f(\eta^{k,k+1})-f(\eta)] + c^-_{k+1} [f(\eta^{k+1,k})-f(\eta)] \\
		&+  \big[\tfrac12(N_k+N_{k+1}+1)\big]_q^2f(\eta).\end{split}
\end{align*}
Therefore, if we subtract the constant and sum over $k$, we get the generator of $\ASEP{q,\vec{N}}$:
\begin{equation} \label{eq:ASEPCasimir2}
\genASEP= \sum_{k=1}^{M-1} \pitensor \Big(\Delta(\Omega)- \big[\tfrac12(N_k+N_{k+1}+1)\big]_q^2 \Big).
\end{equation}
It now immediately follows that $w$ is a reversible measure for $\ASEP{q,\vec{N}}$: 
\begin{proof}[Alternative proof of Theorem \ref{Thm:reversiblemeasures}(i)]
	Since the $\pi_k$ are $*$-representations, $\De$ is a $*$-homomorphism and $\Omega^*=\Omega$, we have that $\genASEP$ is self-adjoint with respect to the measure $w$. Therefore, $w$ is a reversible measure for $\ASEP{q,\vec{N}}$.
\end{proof}
\subsection{Properties of $q$-Krawtchouk polynomials}
In this subsection, we introduce three different recurrence relations for the $q$-Krawtchouk polynomials that we need later on. For convenience, write
\begin{align*}
	k(n,x;\rho)= k(n,x;\rho,N;q)\label{eq:dualqkrawtchouknosite}.
\end{align*}
The first recurrence relation we give is a standard three-term recurrence relation, the other two are very similar but more dynamic in the sense that the parameter $\rho$ will be changing as well. \\

All sets of orthogonal polynomials $\{p_n(x)\}$, where $n$ is the degree of the polynomial and $x$ the variable, satisfy a three-term recurrence relation in the degree of the polynomial of the form
\[
xp_n(x)=A(n) p_{n+1}(x) + B(n) p_n(x) + C(n) p_{n-1}(x).
\]  
At the moment we consider the $q$-Krawtchouk polynomials as having degree $x$ in the variable $q^{-2n}$. Then its three-term recurrence relation is given by,
\begin{align}
	q^{-2n}k(n,x;\rho) &=  a_{-1}(x) k(n,x-1;\rho) +a_0(x) k(n,x;\rho) +a_1(x) k(n,x+1;\rho),\label{eq:qdifkrawtchouk}
\end{align} 
The coefficients can be found in Appendix \ref{app:coefficients}, see also \cite[\S14.17]{KLS}.
Moreover, the $q$-Krawtchouk polynomials satisfy two more relations where besides $x$, the parameter $\rho$ will be changing as well. 
\begin{lemma}\label{lem:qdifkrawtchouk}
	The $q$-Krawtchouk polynomials satisfy $q$-difference equations of the form
	\begin{align}
		\begin{split}q^{-2n}k(n,x;\rho) =&  a_{-2,2}(x) k(n,x-2;\rho+2)+a_{-1,2}(x) k(n,x-1;\rho+2) \\
			&+a_{0,2}(x) k(n,x;\rho+2) ,\end{split}\label{eq:dualqkrawtchoukxrho+}\\
		\begin{split} q^{-2n}k(n,x;\rho) =&  a_{0,-2}(x) k(n,x;\rho-2)  +a_{1,-2}(x) k(n,x+1;\rho-2) \\
			&+a_{2,-2}(x) k(n,x+2;\rho-2).\end{split}\label{eq:dualqkrawtchoukxrho-}
	\end{align}
	Explicit expressions for $a_{j,m}(x)$ can be found in Appendix \ref{app:coefficients}. 
\end{lemma}
\noindent Note that the coefficient $a_{j,m}(x)$ is in front of $k(n,x+j,\rho+m)$.
\begin{proof}
	A direct computation shows that
	\begin{align}
		q^{-n}k(n,x;\rho)=\frac{1+q^{2x+2\rho-2N}}{1+q^{4x+2\rho-2N}}k(n,x;\rho+1) + \frac{1-q^{-2x}}{1+q^{2N-4x-2\rho}}k(n,x-1;\rho+1).\label{eq:dualqkrawtchoukxrho}
	\end{align}
	Multiplying both sides by $q^{-n}$ and then applying \eqref{eq:dualqkrawtchoukxrho} to the right-hand side gives \eqref{eq:dualqkrawtchoukxrho+}. Using the identity \cite[(III.6) and (III.7)]{GR}
	\[\rphis{3}{2}{q^{-n}, q^{-x}, -cq^{x-N}}{q^{-N},0}{q,q}=(-c)^n\rphis{3}{2}{q^{-n}, q^{x-N}, -c^{-1}q^{-x}}{q^{-N},0}{q,q}\]
	for $q$-hypergeometric series, we obtain
	\begin{align}
	k(n,x;-\rho,N;q)=(-1)^nk(n,N-x,\rho,N;q).\label{eq:KrawtchoukRhoto-Rho}
	\end{align}
	If we replace $\rho$ by $-\rho$ in \eqref{eq:dualqkrawtchoukxrho+} and apply the above symmetry, one obtains \eqref{eq:dualqkrawtchoukxrho-} after replacing $N-x$ by $x$.
\end{proof}

\section{Algebraic construction of generalized dynamic ASEP} \label{sec:ConstructionDynASEP}
\subsection{Constructing the generator}
Constructing the generator of $\RASEP{q,\vec{N},\rho}$ is done by transferring the action of $\pitensor(\De(\Om))$ from the $\eta$-variable to the $\xi$-variable using the $q$-Krawtchouk polynomials. We will proceed in the following three steps. 
\begin{enumerate}[label=(\arabic*)]
	\item In their usual action, we can let the operators $\pitensor(\Delta(Y_{h^{+}_{k+2}(\xi)}))$ and $\pitensor(\De(K^{-2}))$ act on the $\eta$ variable of the duality function $K_{\mathsf{R}}(\eta,\xi)$ given by \eqref{eq:duality function K} as a nested product of $q$-Krawtchouk polynomials. We will show that we can transfer these actions to be exclusively depending on the $\xi$ variable. This is the content of Lemma \ref{lem:etatoxi}.
	\item Then we show that $\Omega$ can be written in terms of $Y_\rho$ and $K^{-2}$, i.e. the latter two elements are `building blocks' for the Casimir $\Omega$. Consequently, $\Delta(\Omega)$ can be written in terms of $\Delta(Y_\rho)$ and $\De(K^{-2})$. 
	\item In the last step, we explicitly compute the action of $\pitensor(\Delta(\Omega))$ on $K_{\mathsf{R}}(\eta,\xi)$ in the $\xi$ variable by combining the previous two steps. This will give the generator on sites $k$ and $k+1$ of $\RASEP{q,\vec{N},\rho}$, which is summarized in Theorem \ref{Thm:dynamicASEP}.
\end{enumerate} 
For step (1), we will show that we can transfer the $\eta$-dependent actions 
\begin{align*}
	[\pitensor(\De(Y_{h^{+}_{k+2}(\xi)}))K_{\mathsf{R}}(\cdot,\xi)](\eta)\qquad \text{and}\qquad [\pitensor(\De(K^{-2}))K_{\mathsf{R}}(\cdot,\xi)](\eta)
\end{align*} 
to the $\xi$ variable.
\begin{lemma}\label{lem:etatoxi}
	The operator $\pitensor(\De(Y_{h^{+}_{k+2}(\xi)}))$ acts as a multiplication operator on $K_{\mathsf{R}}(\eta,\xi)$,
	\begin{align}
		[\pitensor(\De(Y_{h^{+}_{k+2}(\xi)}))K_{\mathsf{R}}(\cdot,\xi)](\eta)=&\ \Big([h^{+}_{k+2}(\xi)]_q-[h^{+}_{k}(\xi)]_q\Big)K_{\mathsf{R}}(\eta,\xi).\label{eq:lemetatoxi1}
	\end{align}
	The operator $\pitensor(\De(K^{-2}))$ is a 9-term operator for $K_{\mathsf{R}}(\eta,\xi)$ in the $\xi$-variable,
	\begin{align} 
		\begin{split}q^{-N_k-N_{k+1}}\big[\pitensor(\De(K^{-2}))K_{\mathsf{R}}(\cdot,\xi)\big](\eta) =&\ a_{-1}(\xi_{k+1}) \sum_{j=0}^2 a_{j,-2}(\xi_{k})K_{\mathsf{R}}(\eta,\xi+j\epsilon_k-\epsilon_{k+1}) \\
			&+a_0(\xi_{k+1}) \sum_{j=-1}^1 a_{j}(\xi_{k})K_{\mathsf{R}}(\eta,\xi+j\epsilon_{k})\\
			&+a_1(\xi_{k+1}) \sum_{j=-2}^0 a_{j,2}(\xi_{k})K_{\mathsf{R}}(\eta,\xi+j\epsilon_{k}+\epsilon_{k+1}),\end{split} \label{eq:lemetatoxi2}
	\end{align}
	where $\epsilon_j$ is the standard unit vector in $\R^M$ with 1 as $j$-th element.
\end{lemma} 
\begin{proof}
	First, let us define for this proof only,
	\begin{align*}
		k(\eta_k,\xi_k;\rho)= q^{\eta_k u_k(\vec{N})}k(\eta_k,\xi_k;\rho,N_k;q),
	\end{align*}
	then
	\begin{align*}
		K_{\mathsf{R}}(\eta,\xi)=\prod_{k=1}^M k(\eta_k,\xi_k;h^{+}_{k+1}(\xi))
	\end{align*}
	is the duality function between $\ASEP{q,\vec{N}}$ and $\RASEP{q,\vec{N},\rho}$ from \eqref{eq:duality function K}. Moreover, by multiplying both sides of the recurrence relation \eqref{eq:qdifkrawtchouk} and the two recurrence relations from Lemma \ref{lem:qdifkrawtchouk} by the factor $q^{\eta_k u_k(\vec{N})}$, we see that these three relations hold equally well for $k(\eta_k,\xi_k,\rho)$ defined above.
	For the moment, let us fix the number of sites $M$ to be 2. Later on, we will generalize to $M\in\N$ sites by starting from the two rightmost sites and then inductively working towards the left. The (2-site) duality function $K_{\mathsf{R}}(\eta,\xi)$  is now given by
	\begin{align}
		\begin{split}
			K_{\mathsf{R}}(\eta,\xi)=& q^{-\half u(\eta;\vec{N})} k(\eta_1,\xi_1;N_1,h^{+}_{2}(\xi);q)k(\eta_{2},\xi_{2};N_2,\rho;q)\\
			=& k(\eta_1,\xi_1;h^{+}_{2}(\xi))k(\eta_{2},\xi_{2};\rho).
		\end{split}\label{eq:twositeduality}
	\end{align}
	Note that the right $q$-Krawtchouk polynomial only depends on site $2$, but that the left one depends on sites $1$ and $2$ since $h^{+}_{2}(\xi)$ contains $\xi_2$ and $N_2$. Later on, this will be crucial since this will allow us to work inductively from right to left when extending to $M\in\N$ sites.\\
	
	We will now show the following two things.
	\begin{enumerate}[label=(\roman*)]
		\item $K_{\mathsf{R}}(\eta,\xi)$ are eigenfunctions of $\pitensortwo(\De(Y_\rho))$ with eigenvalue $\mu=[\rho]_q-[h^+_{1}(\xi)]_q$. Hence $\pitensortwo(\De(Y_\rho))$ acts as multiplication by $\mu$ on $K_{\mathsf{R}}(\eta,\xi)$.
		\item Using the three-term recurrence relation \eqref{eq:qdifkrawtchouk} and the $q$-difference equations from Lemma \ref{lem:qdifkrawtchouk}, we can show that $\pitensortwo(\De(K^{-2}))$ acts as a $9$-term operator in the $\xi$-variable on $K_{\mathsf{R}}(\eta,\xi)$.
	\end{enumerate}
	For (i), we use that the $q$-Krawtchouk polynomials $k(\cdot,\xi_k;\rho): \{0,...,N_k\}\to \R$ are eigenfunctions of $\pi_k(Y_\rho)$,
	\begin{align}
		[\pi_k(Y_\rho)k(\cdot,\xi_k;\rho)](\eta_k)=([\rho]_q-[2\xi_k-N_k+\rho]_q)k(\eta_k,\xi_k;\rho). \label{eq:YrhoKrawtchoukEigenfunction}
	\end{align}
	Using this and the explicit expression \eqref{eq:coproductYrho} for $\Delta(Y_\rho)$, one can show that the 2-site duality function $K_{\mathsf{R}}(\eta,\xi)$ is an eigenfunction of $\pitensortwo(\De(Y_\rho))$,
	\begin{align}
		\begin{split} [\pitensortwo (\Delta(Y_\rho))K_{\mathsf{R}}(\cdot,\xi)](\eta)
			&=([\rho]_q-[\rho+2(\xi_1+\xi_2)-(N_1+N_2)]_q)K_{\mathsf{R}}(\eta,\xi) \\
			&= ([\rho]_q-[h^+_{1}(\xi)]_q)K_{\mathsf{R}}(\eta,\xi).\end{split}\label{eq:DeltaYrhoeigenfunctions}
	\end{align}
	See Appendix \ref{app:eigenfunctions} for more details. Thus, $\pitensortwo(\De(Y_\rho))$ can also act on $K_{\mathsf{R}}(\eta,\xi)$ by multiplication by the ($\eta$-independent) eigenvalue $[\rho]_q-[h^+{1}(\xi)]_q$. Note that everything goes entirely similar if we take sites $k,k+1$ instead of sites $1,2$, and $h^+_{k}(\xi),h^+_{k+1}(\xi), h^+_{k+2}(\xi)$ instead of $h^+_1(\xi), h^+_2(\xi)$ and $\rho$, proving \eqref{eq:lemetatoxi1}.\\
	\\
	\noindent The verification of (ii) is more subtle. By \eqref{eq:representation}, the operator $\pi_k(K^{-2})$ is multiplication by $q^ {N_k-2\eta_k}$. Therefore, using $\Delta(K^{-2})=K^{-2}\otimes K^{-2}$, 
	\begin{align}
		\begin{split}
			[\pitensortwo(\Delta(K^{-2}))K_{\mathsf{R}}(\cdot,\xi)](\eta)=&q^{N_1+N_{2}}q^{-2\eta_1-2\eta_{2}}K_{\mathsf{R}}(\eta,\xi) \\
			=&q^{N_1+N_2} q^{-2\eta_1} k(\eta_1,\xi_1;h^{+}_{2}(\xi))q^{-2\eta_{2}}k(\eta_2,\xi_2;\rho).
		\end{split}
		\label{eq:DeK^-2onKrawtchouk}
	\end{align}  
	We can use the $q$-difference equation \eqref{eq:qdifkrawtchouk} and \eqref{eq:dualqkrawtchoukxrho+}, \eqref{eq:dualqkrawtchoukxrho-} from Lemma \ref{lem:qdifkrawtchouk} to show that $q^{-2\eta_k}$ can act in three different ways on the $\xi_k$ variable of $k(\eta_k,\xi_k;\rho)$,
	\begin{alignat}{2}
		q^{-2\eta_k}k(\eta_k,\xi_k;\rho)=&  a_{-1} k(\eta_k,\xi_k-1;\rho) +a_0 k(\eta_k,\xi_k;\rho) +a_1 k(\eta_k,\xi_{k}+1;\rho),&& \label{eq:difqkraw0} \\
		= &a_{-2,2} k(\eta_k,\xi_k-2;\rho+2) +a_{-1,2} k(\eta_k,\xi_k-1;\rho+2) +a_{0,2} k(\eta_k,\xi_k;\rho+2),\hspace{-0.19cm}&&\label{eq:difqkrawrho+}\\
		=  &a_{0,-2} k(\eta_k,\xi_k;\rho-2)  +a_{1,-2} k(\eta_k,\xi_k+1;\rho-2) +a_{2,-2} k(\eta_k,\xi_k+2;\rho-2),\hspace{-0.19cm}&&\label{eq:difqkrawrho-}
	\end{alignat}
	where the coefficients $a_j$ and $a_{j,m}$ depend on $\xi_k$, $N_k$ and $\rho$ (and not on $\eta$) and can be found in Appendix \ref{app:coefficients}.\\
	
	\noindent The naive approach here is to use the standard three-term recurrence relation \eqref{eq:difqkraw0} for both 
	\begin{align*}
		q^{-2\eta_1}k(\eta_1,\xi_1;h^{+}_{2}(\xi)) \qquad \text{and} \qquad  q^{-2\eta_{2}}k(\eta_2,\xi_2;\rho)
	\end{align*}
	in \eqref{eq:DeK^-2onKrawtchouk}. However, this would not work since the variable $\xi_{2}$ of the right one is part of the parameter $h^+_{2}(\xi)$ of the left one. To see this, apply \eqref{eq:difqkraw0} for `$q^{-2\eta_{2}}$' to the right $q$-Krawtchouk polynomial,
	\begin{align}
		\begin{split}
			q^{-2\eta_1-2\eta_{2}}K_{\mathsf{R}}(\eta,\xi) =&\, q^{-2\eta_1}k(\eta_1,\xi_1,h^+_{2}(\xi))\big[a_{-1}(\xi_2) k(\eta_{2},\xi_{2}-1;\rho)\\
			&+a_0(\xi_2) k(\eta_{2},\xi_{2};\rho)+a_1(\xi_2) k(\eta_{2},\xi_{2}+1;\rho)\big].
		\end{split} \label{eq:qdifallterms}
	\end{align}
	Let us take a closer look at the term
	\begin{align}
		q^{-2\eta_1}k(\eta_1,\xi_1,h^+_{2}(\xi))a_{-1}(\xi_2) k(\eta_{2},\xi_{2}-1;\rho).\label{eq:qdif1}
	\end{align}
	Since `$h^{+}_{2}(\xi)$' contains a term `$2\xi_{2}$', and the variable of the right $q$-Krawtchouk polynomial is `$\xi_{2}-1$', we have to adjust `$h^+_{2}(\xi)$' by `$-2$' to obtain a product of $q$-Krawtchouk polynomials of the same form as $K_{\mathsf{R}}(\eta,\xi)$. That is, since the variable of the right polynomial is $\xi_2-1$, we have to adjust the left $q$-Krawtchouk polynomial so that we end up with terms 
	\begin{align*}
		K_\mathsf{R}(\eta,\xi_1+j,\xi_2-1)= k(\eta_1,\xi_1+j;h^{+}_{2}(\xi)-2)k(\eta_2,\xi_2-1;\rho), 
	\end{align*} 
	with $j\in\Z$. Therefore, we have to use the third relation \eqref{eq:difqkrawrho-} for the factor $q^{-2\eta_1}k(\eta_1,\xi_1,h^+_{2}(\xi))$ in \eqref{eq:qdif1} to obtain
	\begin{align*}
		\big(a_{0,-2}(\xi_1) k(\eta_1,\xi_1;h^{+}_{2}(\xi)-2)  &+a_{1,-2}(\xi_1) k(\eta_1,\xi_1+1;h^{+}_{2}(\xi)-2) \\
		&+a_{2,-2}(\xi_1) k(\eta_1,\xi_1+2;h^{+}_{2}(\xi)-2)\big)\times a_{-1}(\xi_2) k(\eta_{2},\xi_{2}-1;\rho).
	\end{align*}
	This is equal to
	\[
	a_{-1}(\xi_2)\big[ a_{0,-2}(\xi_1) K_{\mathsf{R}}(\eta,\xi_1,\xi_{2}-1) + a_{1,-2}(\xi_1) K_{\mathsf{R}}(\eta,\xi_1+1,\xi_{2}-1)a_{2,-2}(\xi_1) K_{\mathsf{R}}(\eta,\xi_1+2,\xi_{2}-1)\big].
	\]
	Let us now look at the other terms in \eqref{eq:qdifallterms}. Using the first equation \eqref{eq:difqkraw0} for the term 
	\[
	q^{-2\eta_1}k(\eta_1,\xi_1;h^+_{2}(\xi))a_0(\xi_2) k(\eta_{2},\xi_{2};\rho)
	\]
	and the second equation \eqref{eq:difqkrawrho+} for
	\[
	q^{-2\eta_1}k(\eta_1,\xi_1;h^+_{2}(\xi))a_1(\xi_2) k(\eta_{2},\xi_{2}+1;\rho),
	\]
	we obtain \eqref{eq:lemetatoxi2} for $k=1$.\\
	
	\noindent Let us now generalize this to $M\in\N$ sites and $\pitensor$ for $k=1,2,\ldots,M-1$. It is crucial to observe that in \eqref{eq:twositeduality} we can pick the parameter $\rho$ of the right $q$-Krawtchouk polynomial freely. However, its variable $\xi_{2}$ and dimension $N_{2}$ have to get into the `$\rho$' parameter of the left $q$-Krawtchouk polynomial by adding $2\xi_2-N_2$. Therefore, we can work from right to left inductively, as long as we change the `$\rho$' parameter accordingly every time. That is, we have to add $2\xi_{k+1}-N_{k+1}$ to the `$\rho$' parameter of the 1-site duality function of site $k$ each time we go from site $k+1$ to $k$,
	\[
	h^+_{k+1}(\xi)=h^+_{k+2}(\xi)+2\xi_{k+1}-N_{k+1}.
	\] 
	Doing this iteratively, we see that this agrees with the definition of our height function,
	\[
	h^+_{k}(\xi)=\rho+\sum_{j=k}^M (2\xi_j - N_j).\qedhere
	\]
\end{proof}
Now that we have finished step (1), we move on to step (2): expressing the Casimir $\Om$ in terms of $K^{-2}$ and $Y_\rho$. One can prove that 
\begin{align}
	\Omega = \frac{f(Y_\rho-[\rho]_q,K^{-2})}{(q+q^{-1})(q-q^{-1})^2}+ \frac{(q+q^{-1})K^{-2}}{(q-q^{-1})^2} + [\rho]_q\frac{Y_\rho-[\rho]_q}{q+q^{-1}} -\frac{ 2}{(q-q^{-1})^2} ,\label{eq:CasimirinYrhoK-2}
\end{align}
where $f:\U_q\times \U_q \to\U_q$ is the function given by
\begin{align}
	f(A,B)=(q^2+q^{-2})ABA - A^2B-BA^2.
\end{align}
This identity in $\U_q$ can be shown by either a direct calculation using the commutation relations \eqref{eq:UqRelations}, or by observing that \eqref{eq:CasimirinYrhoK-2} is actually a relation in the degenerate version of the Askey-Wilson algebra $\text{AW}(3)$ generated by $Y_\rho$ and $K^{-2}$, see e.g. \cite{GranZhed} or \cite[Theorem 2.2]{GroeneveltWagenaar}\footnote{Note that in \cite{GroeneveltWagenaar} the Casimir differs from $\Om$ by a scaling factor and an additive constant.}. Note that we can pick our parameter $\rho$ freely, in particular we can take $\rho=h^{+}_{k+2}(\xi)$. Now, taking the coproduct on both sides we obtain
\begin{align}
	\De(\Omega) = \frac{f(\De(Y_\rho)-[\rho]_q,\De(K^{-2}))}{(q+q^{-1})(q-q^{-1})^2}+ \frac{(q+q^{-1})\De(K^{-2})}{(q-q^{-1})^2} + [\rho]_q\frac{\De(Y_\rho)-[\rho]_q}{q+q^{-1}}  -\frac{ 2}{(q-q^{-1})^2}  ,\label{eq:DeCasimirinYrhoK-2}
\end{align}
completing step (2).\\
\ \\
Lastly, for step (3) we will combine steps (1) and (2) and do an explicit computation to obtain the explicit action from $\pitensor(\De(\Om))$ on the $\xi$ variable of the nested product of $q$-Krawtchouk polynomials $K_\mathsf{R}(\eta,\xi)$. This gives the rates of the generator of $\RASEP{q,\vec{N},\rho}$ given in Definition \ref{Def:dynamic ASEP}. 
\begin{theorem}\label{Thm:dynamicASEP}
	The action of the operator $\pitensor(\De(\Omega))$ on the $\eta$-variable of $K_{\mathsf{R}}(\eta,\xi)$ can be transferred to the $\xi$-variable,
	\begin{align*}
		\begin{split}[\pitensor(\De(\Omega))K_{\mathsf{R}}(\cdot,\xi)](\eta) =&  C^{\mathsf R,+}_k [K_{\mathsf{R}}(\eta,\xi^{k,k+1})-K_{\mathsf{R}}(\eta,\xi)] +  C^{\mathsf R,-}_{k+1}[K_{\mathsf{R}}(\eta,\xi^{k+1,k})-K_{\mathsf{R}}(\eta,\xi) ] \\
			&+ \big[\tfrac12(N_k+N_{k+1}+1)\big]_q^2 )K_{\mathsf{R}}(\eta,\xi).\end{split}
	\end{align*}
	Here, $C^{\mathsf R,+}$ and $C^{\mathsf R,-}$ are the $\xi$-dependent rates from $\RASEP{q,\vec{N},\rho}$ given in Definition \ref{Def:dynamic ASEP}.
\end{theorem}
\noindent Note that the factor $\big[\tfrac12(N_k+N_{k+1}+1)\big]_q^2$ is the same as the one appearing in \eqref{eq:ASEPCasimir2} for $\ASEP{q,\vec{N}}$.
\begin{proof}
	The idea is to use \eqref{eq:DeCasimirinYrhoK-2} and Lemma \ref{lem:etatoxi} to transfer the action of $\pitensor(\De(\Om))$ on $K_{\mathsf{R}}(\eta,\xi)$ from the $\eta$-variable to the $\xi$-variable. Applying $\pitensor$ to \eqref{eq:DeCasimirinYrhoK-2} we obtain for any $\rho\in\C$,
	\begin{align}
		\begin{split}[\pitensor(\Delta(\Omega))=& \frac{f\big(\pitensor(\De(Y_{\rho})-[\rho])_q,\pitensor(\De(K^{-2}))\big)+(q+q^{-1})^2\pitensor(\De(K^{-2}))}{(q+q^{-1})(q-q^{-1})^2} \\
			&+ [\rho]_q\frac{\pitensor(\De(Y_{\rho})-[\rho]_q)}{q+q^{-1}}  -\frac{ \pitensor(2)}{(q-q^{-1})^2}. \end{split}\label{eq:DeOmwrittenout}
	\end{align}
	By Lemma \ref{lem:etatoxi}, we have that $\pitensor(\De(K^{-2}))$ acts on $K_\mathsf{R}(\eta,\xi)$ as a $9$-term operator in the $\xi$-variable and
	\begin{align}
		[\pitensor(\De(Y_{h^{+}_{k+2}(\xi)})-[h^{+}_{k+2}(\xi)]_q)K_\mathsf{R}(\cdot,\xi)](\eta) = - [h^{+}_{k}(\xi)]_qK_\mathsf{R}(\eta,\xi).\label{eq:piDeYrho}
	\end{align}
	Since $\pitensor(\De(K^{-2}))$ is a $9$-term operator, we need some extra notation. Denote by $\xi_{k,k+1}(j,m)$ the state $\xi$ in which there are $j$ and $m$ particles added to site $k$ and $k+1$ respectively, i.e.
	\begin{align*}
		\xi_{k,k+1}(j,m)=(\xi_1,\ldots,\xi_{k-1},\xi_k+j,\xi_{k+1}+m,\xi_{k+2},\ldots,\xi_M).
	\end{align*}  
	Note that $\xi_{k,k+1}(-1,1)=\xi^{k,k+1}$ and $\xi_{k,k+1}(1,-1)=\xi^{k+1,k}$. Moreover, if $j+m\neq 0$, the amount of particles in the states $\xi$ and $\xi_{k,k+1}(j,m)$ is different. Let us apply $\pitensor(\De(\Om))$ to $K_\mathsf{R}(\eta,\xi)$. Then we obtain, using \eqref{eq:DeOmwrittenout} with $\rho=h^{+}_{k+2}(\xi)$, Lemma \ref{lem:etatoxi} and \eqref{eq:piDeYrho},
	\begin{align}
		\begin{split}
			\big[\pitensor(\De(\Om))K_{\mathsf{R}}(\cdot,\xi)\big](\eta) =&\ q^{N_k+N_{k+1}}a_{-1}(\xi_{k+1}) \sum_{j=0}^2 a_{j,-2}(\xi_{k})\beta_{k}(j-1) K_{\mathsf{R}}(\eta,\xi_{k,k+1}(j,-1)) \\
			&+q^{N_k+N_{k+1}}a_0(\xi_{k+1}) \sum_{j=-1}^1 a_{j}(\xi_{k})\beta_k(j)K_{\mathsf{R}}(\eta,\xi_{k,k+1}(j,0))\\
			&+q^{N_k+N_{k+1}}a_1(\xi_{k+1}) \sum_{j=-2}^0 a_{j,2}(\xi_{k})\beta_k(j+1)K_{\mathsf{R}}(\eta,\xi_{k,k+1}(j,1))\\
			&- \left(\frac{[h^{+}_{k+2}(\xi)]_q[h^{+}_{k}(\xi)]_q}{q+q^{-1}}+\frac{2}{(q-q^{-1})^2} \right)K_\mathsf{R}(\eta,\xi),
		\end{split}\label{eq:proof9term}
	\end{align}
	where
	\begin{align*}
		\begin{split}\beta_k(m)=&\frac{ (q^2+q^{-2})[h^{+}_{k}(\xi)]_q[h^{+}_{k}(\xi)+2m]_q-[h^{+}_{k}(\xi)]_q^2-[h^{+}_{k}(\xi)+2m]_q^2 +(q+q^{-1})^2}{(q-q^{-1})^2(q+q^{-1})}. \end{split}
	\end{align*}
	The expression \eqref{eq:proof9term} greatly simplifies from a $9$-term operator to a $3$-term operator, since for all $p\in\R$ we have the identity
	\begin{align}
		(q^2+q^{-2})[p]_q[p+2]_q - [p]_q^2-[p+2]_q^2 + (q+q^{-1})^2 =0,
	\end{align} 
	as readily verified by a direct computation. Therefore, $\beta_k(m)=0$ if $m= \pm 1$. Consequently, only the terms in \eqref{eq:proof9term} with $K_\mathsf{R}(\eta,\xi_{k,k+1}(j,m))$ remain where $j+m=0$, i.e. the amount of particles is preserved. Thus,
	\begin{align*}
		\begin{split}
			\big[\pitensor(\De(\Om))K_{\mathsf{R}}(\cdot,\xi)\big](\eta) =&\ q^{N_k+N_{k+1}}a_{-1}(\xi_{k+1})  a_{1,-2}(\xi_{k})\beta_{k}(0) K_{\mathsf{R}}(\eta,\xi^{k+1,k}) \\
			&+q^{N_k+N_{k+1}}a_0(\xi_{k+1}) a_{0}(\xi_{k})\beta_k(0)K_{\mathsf{R}}(\eta,\xi)\\
			&+q^{N_k+N_{k+1}}a_1(\xi_{k+1})  a_{-1,2}(\xi_{k})\beta_k(0)K_{\mathsf{R}}(\eta,\xi^{k,k+1})\\
			&- \left(\frac{[h^{+}_{k+2}(\xi)]_q[h^{+}_{k}(\xi)]_q}{q+q^{-1}}+\frac{2}{(q-q^{-1})^2} \right)K_\mathsf{R}(\eta,\xi).
		\end{split}
	\end{align*}
	Using the explicit expressions for $a_j(\xi_{k})$ and $a_{j,-2j}(\xi_{k+1})$ found in Appendix \ref{app:coefficients}, we obtain the factors $C^{\mathsf R,+}_k$ and $C^{\mathsf R,-}_k$ for $K_\mathsf{R}(\eta,\xi^{k,k+1})$ and $K_\mathsf{R}(\eta,\xi^{k+1,k})$ respectively. Therefore,
	\begin{align}
		\begin{split}[\pitensor(\De(\Omega))K_{\mathsf{R}}(\cdot,\xi)](\eta) =&\, C^{\mathsf R,+}_k [K_{\mathsf{R}}(\eta,\xi^{k,k+1})-K_{\mathsf{R}}(\eta,\xi)]\\
			&+  C^{\mathsf R,-}_k[K_{\mathsf{R}}(\eta,\xi^{k+1,k})-K_{\mathsf{R}}(\eta,\xi) ] \\
			&+ \gamma_k(\xi) K_{\mathsf{R}}(\eta,\xi)\end{split}\label{eq:DeOmForm}
	\end{align} 
	for some factor $\gamma_k(\xi)$. To find this factor, observe that $K_\mathsf{R}(0,\xi)=1$ for any $\xi$ by \eqref{eq:duality function K} and the fact that $k(0,x;q,N,\rho)=1$. If we now take $\eta=0$ in \eqref{eq:DeOmForm}, we obtain
	\begin{align*}
		[\pitensor(\De(\Omega))K_{\mathsf{R}}(\cdot,\xi)](0) = \gamma_k(\xi) K_{\mathsf{R}}(0,\xi).
	\end{align*}
	Since $\pitensor(\De(\Om))$ is related to the generator of the $\ASEP{q,\vec{N}}$ process via \eqref{eq:ASEPCasimir2}, we also have that
	\begin{align*}
		[\pitensor(\De(\Omega))K_{\mathsf{R}}(\cdot,\xi)](0) = \big[\tfrac12(N_k+N_{k+1}+1)\big]_q^2.
	\end{align*}
	Hence $\gamma_k(\xi)=\big[\tfrac12(N_k+N_{k+1}+1)\big]_q^2$. 
\end{proof}
\subsection{Duality between ASEP$_\mathsf{R}$ and ASEP} \label{subsec:dualityDynASEPandASEP}
Because of the way we constructed the generator of $\RASEP{q,\vec{N},\rho}$, we automatically get a Markov duality with the standard $\ASEP{q,\vec{N}}$ and the  $q$-Krawtchouk polynomials $K_{\mathsf{R}}(\eta,\xi)$ as duality functions, which is the content of Theorem \ref{Thm:dualityASEPRASEP}.
\begin{proof}[Proof of Theorem \ref{Thm:dualityASEPRASEP}]
	Combining \eqref{eq:ASEPCasimir2} with Theorem \ref{Thm:dynamicASEP}, we have
	\begin{align*}
		[L_{q,\vec{N}}K_{\mathsf{R}}(\cdot,\xi)](\eta) &=\sum_{k=1}^{M-1} c^+_k [K_{\mathsf{R}}(\eta^{k,k+1},\xi)-K_{\mathsf{R}}(\eta,\xi)] + c^-_k [K_{\mathsf{R}}(\eta^{k+1,k},\xi)-K_{\mathsf{R}}(\eta,\xi)].\\
		&= \sum_{k=1}^{M-1}[ \pitensor\big(\De(\Om)-[\tfrac12(N_k+N_{k+1}+1)]_q^2\big)K_{\mathsf{R}}(\cdot,\xi)](\eta) \\
		&= \sum_{k=1}^{M-1} C^{\mathsf R,+}_k [K_{\mathsf{R}}(\eta,\xi^{k,k+1})-K_{\mathsf{R}}(\eta,\xi)] +  C^{\mathsf R,-}_k[K_{\mathsf{R}}(\eta,\xi^{k+1,k})-K_{\mathsf{R}}(\eta,\xi) ]\\
		&= [L^{\mathsf{R}}_{q,\vec{N},\rho}K_{\mathsf{R}}(\eta,\cdot)](\xi). \qedhere
	\end{align*}
\end{proof}
\subsection{Reversibility of ASEP$_\mathsf{R}$}\label{subsec:RevDynASEP}
We end this section by giving an alternative proof of the reversibility of $\RASEP{q,\vec{N},\rho}$. Similarly as for $\ASEP{q,\vec{N}}$, we do this by proving that its generator is self-adjoint with respect to the reversible measure $W_\mathsf{R}$ defined in \eqref{eq:Weightfunction WR}.
\begin{proof}[Alternative proof of Theorem \ref{Thm:reversiblemeasures}(ii)]
	We will use the orthogonality of the duality function $K_\mathsf{R}(\eta,\xi)$ given in Theorem \ref{Thm:reversiblemeasures}(iii). First, define the Hilbert space $H^\mathsf{R}$ as functions acting on states $\xi$ in our state space $X=\{0,\ldots,N_1\}\times \cdots \times \{0,\ldots,N_M\}$ with inner product induced by the measure $W_\mathsf{R}$,
	\begin{align*}
		\langle f,g\rangle_{H^\mathsf{R}} = \sum_\xi f(\xi)g(\xi) W_\mathsf{R}(\xi;\vec{N},\rho;q).
	\end{align*} 
	Next, define the linear operator $\Lambda_\mathsf{R}:H\to H^\mathsf{R}$ by
	\begin{align*}
		(\Lambda_\mathsf{R} f)(\xi)=\left\langle f, K_\mathsf{R}(\cdot,\xi) \right\rangle_{H},
	\end{align*} 
	where $H$ was the Hilbert space with inner product given by $w(\eta;\vec{N};q)$, see \eqref{eq:HilbertH}.
	Then $\Lambda$ is a unitary operator since it maps the orthogonal basis $\delta_\eta$ of $H$ to an orthogonal basis of $H^\mathsf{R}$ while preserving its norm. Indeed, trivially we have
	\begin{align*}
		||\delta_\eta||_{H}^2 = w(\eta;\vec{N};q). 
	\end{align*}
	On the other hand, since
	\begin{align*}
		\Lambda_\mathsf{R}(\delta_\eta)(\xi) = K_\mathsf{R}(\eta,\xi)w(\eta;\vec{N};q),
	\end{align*}
	we can use Theorem \ref{Thm:reversiblemeasures}(iii) to obtain
	\begin{align*}
		||\Lambda_\mathsf{R}(\delta_\eta)||_{H^\mathsf{R}}^2 =w(\eta;\vec{N};q).
	\end{align*}
	Now let us define a $*$-representation of $\U_q^{\otimes M}$ equivalent to
	\begin{align*}
		\pi=\pi_1 \otimes \pi_2 \otimes \cdots \otimes \pi_M
	\end{align*}
	by intertwining with $\Lambda_\mathsf{R}$. That is, let $\sigma$ be the representation on $H^\mathsf{R}$ defined by
	\begin{align}
		\sigma(X)=\Lambda_\mathsf{R} \circ\pi(X)\circ \Lambda^{-1}_\mathsf{R}, \qquad X \in \U_q^{\otimes M}.
	\end{align}
	For convenience, write for $X\in \U_q\otimes \U_q$,
	\begin{align*}
		\sigma_{k,k+1}(X)=\sigma(\underbrace{1\otimes \cdots \otimes 1}_{k-1} \otimes X \otimes \underbrace{1 \cdots \otimes 1}_{M-(k+1)}).
	\end{align*}
	Note that $\sigma_{k,k+1}(\Delta(Y_{h^{+}_{k+2}(\xi)}))$ and $\sigma_{k,k+1}(\Delta(K^{-2}))$ are exactly the actions of $\pitensor(\De(Y_{h^{+}_{k+2}(\xi)}))$ and $\pitensor(\Delta(K^{-2}))$ on the $\xi$ variable of $K_\mathsf{R}(\eta,\xi)$ given in Lemma \ref{lem:etatoxi}. Indeed, using that $\pitensor(\De(Y_\rho))$ is self-adjoint for all $\rho\in\R$, we have
	\begin{align*}
		(\Lambda_\mathsf{R} (\pitensor(\De(Y_{h^{+}_{k+2}(\xi)}))f))(\xi) &= \left\langle \pitensor(\De(Y_{h^{+}_{k+2}(\xi)}))f, K_\mathsf{R}(\,\cdot\,,\xi) \right\rangle_{H}\\
		&= \left\langle f, [\pitensor(\De(Y_{h^{+}_{k+2}(\xi)}))K_\mathsf{R}](\,\cdot\,,\xi) \right\rangle_{H}\\
		&= \Big([h^{+}_{k+2}(\xi)]_q-[h^{+}_{k}(\xi)]_q\Big) (\Lambda_\mathsf{R} f)(\xi),
	\end{align*}
	hence
	\begin{align*}
		\Lambda_\mathsf{R} \circ\pitensor(\De(Y_{h^{+}_{k+2}(\xi)})) = \Big([h^{+}_{k+2}(\xi)]_q-[h^{+}_{k}(\xi)]_q\Big) \Lambda_\mathsf{R}.
	\end{align*}
	Therefore,
	\begin{align*}
		\Big[\sigma_{k,k+1}(\De(Y_{h^{+}_{k+2}(\xi)}))f\Big](\xi)= \Big([h^{+}_{k+2}(\xi)]_q-[h^{+}_{k}(\xi)]_q\Big)f(\xi).
	\end{align*}
	The case $\sigma_{k,k+1}(\De(K^{-2}))$ is similar. Since the right-hand side of \eqref{eq:CasimirinYrhoK-2} is invariant under taking the $*$-operation, we can repeat the proof of Theorem \ref{Thm:dynamicASEP} to obtain that
	\begin{align*}
		[\sigma_{k,k+1}(\De(\Omega))f](\xi) = &\, C^{\mathsf R,+}_k [f(\xi^{k,k+1})-f(\xi)] + C^{\mathsf R,-}_k[f(\xi^{k+1,k})-f(\xi) ] \\
		&+ \big[\tfrac12(N_k+N_{k+1}+1)\big]_q^2 )f(\xi).
	\end{align*}
	Therefore, we can construct the generator of $\RASEP{q,\vec{N},\rho}$ with $\sigma$,
	\begin{align*}
		\genRASEP = \sum_{k+1}^{M-1}\sigma_{k,k+1}\left( \De(\Omega) -\big[\tfrac12(N_k+N_{k+1}+1)\big]_q^2 \right).
	\end{align*}
	Since $\sigma$ is a $*$-representation and $\De(\Omega^*)=\De(\Omega)$, we have that $\genRASEP$ is self-adjoint with respect to the reversible measure $W_\mathsf{R}$.
\end{proof}
\section{Generalized dynamic ASEP on the reversed lattice}\label{sec:LeftDynASEP}
As explained in Subsection \ref{subsec:LASEP}, we can define dynamic ASEP on the reversed lattice, referred to as $\LASEP{q,\vec{N},\lambda}$. In this section, we prove duality of this process with $\RASEP{q,\vec{N},\rho}$ with a doubly nested product of $q$-Racah polynomials as duality functions. 
Algebraically, one can describe this process in a similar fashion as before by replacing $K$ by $K^{-1}$ and $q$ by $q^{-1}$. Instead of the twisted primitive element 
\[
Y_\rho =  q^{\frac12} EK + q^{-\frac12}FK - [\rho]_q(K^2-1),
\] 
one then gets the twisted primitive element
\[
\tilde{Y}_\lambda =  q^{-\frac12} EK^{-1} + q^{\frac12}FK^{-1} - [\lambda]_q(K^{-2}-1).
\]
Eigenfunctions of $\pi_k(\tilde{Y}_\lambda)$ are given by $q^{-1}$-Krawtchouk polynomials. Moreover, the inner product of these functions with the eigenfunctions of $\pi_k(Y_\rho)$ are $q$-Racah (or Askey-Wilson) polynomials \cite{Groeneveltquantumaskey,Ros}. Furthermore, this pair of twisted primitive elements satisfies the Askey-Wilson algebra relations \cite{GranZhed}, an algebra that encodes many properties of the Askey-Wilson polynomials. Therefore, it comes as no surprise that duality functions between $\LASEP{q,\vec{N},\lambda}$ and $\RASEP{q,\vec{N},\rho}$ are $q$-Racah polynomials.

\subsection{Duality between ASEP$_\mathsf{L}$ and ASEP$_\mathsf{R}$}\label{subsec:dualityLASEPRASEP}
Since both $\RASEP{q,\vec{N},\rho}$ and $\LASEP{q,\vec{N},\lambda}$  are dual to $\ASEP{q,\vec{N}}$, they are also dual to each other. A duality function is given by the inner product in the Hilbert space $H$ corresponding to the reversible measure of $\ASEP{q,\vec{N}}$ of the duality function $K_\mathsf{L}(\eta,\zeta)$ between $\LASEP{q,\vec{N},\lambda}$ and $\ASEP{q,\vec{N}}$ and the duality function $K_\mathsf{R}(\eta,\xi)$ between $\RASEP{q,\vec{N},\lambda}$ and $\ASEP{q,\vec{N}}$. This scalar-product method is used before in e.g. \cite{CarFraGiaGroRed}. We then get 
\begin{align*}
	R(\zeta,\xi)=\langle K_\mathsf{L}(\cdot,\zeta),K_\mathsf{R}(\cdot,\xi)\rangle_{H}=\sum_{\eta}  K_\mathsf{L}(\eta,\zeta)K_\mathsf{R}(\eta,\xi)w(\eta;\vec{N};q),
\end{align*}
where $w$ is the reversible measure of $\ASEP{q,\vec{N}}$ given in \eqref{eq:Weightfunction w}. We can fit in an extra parameter $v\in\R^\times$ by slightly adjusting one of the original duality functions. Define
\begin{align*}
	K_\mathsf{L}^v(\eta,\zeta) = (-v)^{|\eta|} K_\mathsf{L}(\eta,\zeta) = q^{-\frac12 u(\eta;\vec{N})}\prod_{k=1}^M  (-v)^{\eta_k} k(\eta_k, \zeta_k;h^{-}_{k-1}(\zeta),N_k;q^{-1}),
\end{align*}
where the minus sign is present because of the formula we will encounter in Lemma \ref{lem:krawthouckracah}.
Since this extra factor $(-v)^{|\eta|}$ is a function depending on the total number of particles in the state $\eta$, $K_\mathsf{L}^v(\eta,\zeta)$ is also a duality function between $\LASEP{q,\vec{N},\lambda}$ and $\ASEP{q,\vec{N}}$ (see Remark \ref{rem:InvTotPart}). Therefore, define precisely as in \eqref{eq:sumKrawtchouk})
\begin{align*}
	R^v(\zeta,\xi)=\langle K_\mathsf{L}^v(\cdot,\zeta),K_\mathsf{R}(\cdot,\xi)\rangle_{H}.
\end{align*}
We will show that $R^v$ is indeed a duality function and equal to a product of $q$-Racah polynomials as stated in Theorem \ref{Thm:KrawtchoukRacah}.\\

First, let us show that $R^v(\zeta,\xi)$ is indeed a duality function between $\LASEP{q,\vec{N},\lambda}$ and $\RASEP{q,\vec{N},\rho}$. Using the duality between $\RASEP{q,\vec{N},\rho}$ and $\ASEP{q,\vec{N}}$, we get
\begin{align*}
	\big[\genRASEP R^v(\zeta,\cdot)\big](\xi) =& \sum_\eta K_\mathsf{L}^v(\eta,\zeta)\big[ \genRASEP K_\mathsf{R}(\eta,\cdot)\big](\xi) w(\eta;\vec{N};q) \\
	=& \sum_\eta K_\mathsf{L}^v(\eta,\zeta)\big[ \genASEP  K_\mathsf{R}(\cdot,\xi)\big](\eta) w(\eta;\vec{N};q). 
\end{align*}
Since, $\genASEP$ is self-adjoint with respect to the reversible measure $w$, the expression above equals
\begin{align*}
	\sum_\eta \big[ \genASEP K_\mathsf{L}^v(\cdot,\zeta)\big](\eta)  K_\mathsf{R}(\eta,\xi) w(\eta;\vec{N};q)=& \sum_\eta \big[ \genASEP K_\mathsf{L}^v(\eta,\cdot)\big](\zeta)  K_\mathsf{R}(\eta,\xi) w(\eta;\vec{N};q) \\
	=& \big[\genLASEP R^v(\cdot,\xi)\big](\zeta),
\end{align*}
where we used the duality between $\LASEP{q,\vec{N},\lambda}$ and $\ASEP{q,\vec{N}}$ in the last step.\\

Summations over a product of $q$-hypergeometric series, like $R^v(\zeta,\xi)$, cannot always be expressed more explicitly. However, in this particular case the summation over two $\rphisempty{3}{2}$'s becomes a $\rphisempty{4}{3}$ known as $q$-Racah polynomials. The required formula is given in the following Lemma.
\begin{lemma}\label{lem:krawthouckracah}
	Let $r(y,x;\lambda,\rho,v,N)$ be the 1-site duality function as in \eqref{eq:1-site duality r},
	then we have the following summation formula between $q$-Krawtchouk polynomials and $q$-Racah polynomials,
	\begin{align}
		r(y,x;\lambda,\rho,v,N;q)=\sum_{n=0}^{N}(-v)^n k(n,y;\lambda,N;q^{-1})k(n,x;\rho,N;q)w(n;N;q).
		\label{eq:univarqracah}
	\end{align}
\end{lemma} 
\begin{proof}
	We start on the right-hand side of \eqref{eq:univarqracah}. After expressing the Krawtchouk polynomials $k$ as $_3\varphi_2$-functions, the right-hand side is written as
	\begin{align*}
		\sum_{n=0}^{N}\qbinom{q^2}{N}{n}q^{n(n+\lambda-\rho-N)} (-v)^n &\rphis{3}{2}{q^{-2n}, q^{-2x}, -q^{2x+2\rho-2N}}{q^{-2N},0}{q^2,q^2}\\
		&\hspace{1cm} \times \rphis{3}{2}{q^{2n}, q^{2y}, -q^{-2y-2\lambda+2N}}{q^{2N},0}{q^{-2},q^{-2}}.
	\end{align*}
	Showing that this sum is equal to a $q$-Racah polynomial basically comes down to \cite[Lemma 4.6 ]{Groeneveltquantumaskey} and a change of parameters. Indeed, using definition (4.1) in \cite{Groeneveltquantumaskey} we see that the following summation formula between $\rphisempty{3}{2}$ and $\rphisempty{4}{3}$ functions holds,
	\begin{align}
		\begin{split}c_1&\rphis{4}{3}{q^{-2m},abcdq^{2(m-1)},ax,ax^{-1}}{ab,ac,ad}{q^2,q^2}\\
			&=\sum_{n\in \N} c_2  \rphis{3}{2}{q^{-2n}, \sigma q^{k}\hat{x}, \sigma q^{k}\hat{x}^{-1}}{q^{2k},0}{q^2,q^2} 
			\rphis{3}{2}{q^{2n}, \tau q^{-k}\hat{y}, \tau q^{-k}\hat{y}^{-1}}{q^{-2k},0}{q^{-2},q^{-2}},\end{split}\label{eq:sumchiharaAW}
	\end{align}
	where
	\begin{align*}
		c_1&=q^{-m(m-1)} \frac{(acq^{2m},bcq^{2m};q^2)_\infty}{(ab;q^2)_m(c\hat{x},c\hat{x}^{-1};q^2)_\infty} \frac{(ab,ac,ad;q^2)_m}{(-ad)^m},\\
		c_2&=\frac{v^n q^{n(k-1)}}{(\sigma\tau)^{n}} 
		\frac{(q^{-2k};q^{-2})_n}{(q^{-2};q^{-2})_n},
	\end{align*}
	and $(a,b,c,d,\hat{y})=(q^k\sigma,q^k\sigma^{-1},qv\tau^{-1},qv^{-1}\tau^{-1},\tau q^{-k-2m})$. Taking $k=-N$, $\hat{y}=-iq^{-2y+N-\lambda}$, $\hat{x}=iq^{2x+\rho-N}$, $\sigma=iq^{\rho}$ and $\tau=-iq^{-\lambda}$, where again $i=\sqrt{-1}$, we obtain $m=y$ and on the right-hand side we get a finite sum because $(q^{2N};q^{-2})_{n}=0$ if $n > N$,
	\begin{align}
		\begin{split} c_1 &\rphis{4}{3}{q^{-2y},-q^{2y+2\lambda-2N},-q^{2x+2\rho-2N},q^{-2x}}{q^{-2N},-vq^{\rho+\lambda-N+1},-v^{-1}q^{\rho+\lambda-N+1}}{q^2,q^2}=  \\
			&\sum_{n=0}^{N} c_2 \rphis{3}{2}{q^{-2n}, -q^{2x+2\rho-2N}, q^{-2x}}{q^{-2N},0}{q^2,q^2} 
			\rphis{3}{2}{q^{2n}, -q^{-2y-2\lambda+2N}, q^{2y}}{q^{2N},0}{q^{-2},q^{-2}},\end{split}
			\label{eq:sumchiharaAW2}
	\end{align}
	where
	\begin{align*}
		c_1 =& (-1)^{y}d^{-{y}}q^{-y(y-1)} \frac{(  acq^{2y},bcq^{2y};q^2)_\infty}{(ab;q^2)_{y}(c\hat{x},c\hat{x}^{-1};q^2)_\infty}\frac{(ab,ac,ad;q^2)_{y}}{a^{y}}\\
		=& v^{y}q^{-y(y+\rho+\lambda-N)} \frac{(-vq^{\rho+\lambda-N+1}, vq^{2y-\rho+\lambda-N+1};q^2)_\infty (-v^{-1}q^{\rho+\lambda-N+1})_{y}}{(-vq^{2x+\rho+\lambda-N+1},vq^{-2x-\rho+\lambda+N+1};q^2)_{\infty}}\\
		=& v^{y}\frac{(-vq^{\rho+\lambda-N+1};q^2)_{x} (-v^{-1}q^{\rho+\lambda-N+1})_{y}( vq^{2y-\rho+\lambda-N + 1};q^2)_{N} }{q^{y(y+\rho+\lambda-N)}(vq^{-2x-\rho+\lambda+N+1};q^2)_{x+y}},		
	\end{align*}
	and
	\[
	c_2 =\frac{v^n q^{n(-N-1)}}{q^{n(\rho-\lambda)}} 
	\frac{(q^{2N};q^{-2})_{n}}{(q^{-2};q^{-2})_{n}} = \qbinom{q^2}{N}{n}q^{n(n+\lambda-\rho-N)} (-v)^n.
	\]
	We see that the right-hand side of \eqref{eq:sumchiharaAW2} equals the right-hand side of \eqref{eq:univarqracah}. It remains to show that the left-hand side of \eqref{eq:sumchiharaAW2} is the required $q$-Racah polynomial. The coefficient $c_1$ is exactly the factor $c_{\mathbf r}$ as given in Appendix \ref{app:overview}. The $\rphisempty{4}{3}$ on the left-hand side of \eqref{eq:sumchiharaAW2} is a $q$-Racah polynomial $R_y(x;\al,\be,\ga,\de;q^2)$ with parameters
	\begin{align*}
		\alpha = -v^{-1}q^{\rho+\lambda-(N+1)},\ \ \ \beta=vq^{\rho-\lambda-(N+1)},\ \ \ \gamma=q^{-2(N+1)},\ \ \ \delta=-q^{2\lambda},
	\end{align*}
	which shows that the left-hand side of \eqref{eq:sumchiharaAW2} is exactly $r(y,x;\lambda,\rho,v,N)$. 
\end{proof}
Applying this formula $M$-times, we get that the duality function $R^v$ is equal to a (doubly) nested product of $q$-Racah polynomials.
\begin{proof}[Proofs of Theorem \ref{Thm:KrawtchoukRacah}]
	Writing out the definition of $R^v(\zeta,\xi)$, we obtain
	\begin{align*}
		\widehat{R}^v(\zeta,\xi) =& \sum_{\eta_1}\sum_{\eta_2}\cdots\sum_{\eta_M} K_\mathsf{L}^v(\eta,\zeta)K_\mathsf{R}(\eta,\xi)w(\eta;\vec{N};q)\\
		=& \sum_{\eta_1}\sum_{\eta_2}\cdots\sum_{\eta_M}   \prod_{k=1}^M (-v)^{\eta_k}k(\eta_k, \zeta_k;h^{-}_{k-1}(\zeta),N_k;q^{-1}) k(\eta_k, \xi_k;h^{+}_{k+1}(\xi),N_k;q)w(\eta_k;N_k;q).
	\end{align*}
	Note that the terms with $u(\eta;\vec{N})$ exactly canceled each other. Since $\eta_k$ only appears in the $k$-th term of the product, we can interchange the order of summation and product in the previous expression. Thus we obtain
	\begin{align*}
		R^v(\zeta,\xi) =& \prod_{k=1}^M\left(\sum_{\eta_k}(-v)^{\eta_k}k(\eta_k, \zeta_k; h^{-}_{k-1}(\zeta),N_k;q) k(\eta_k, \xi_k;h^{+}_{k+1}(\xi),N_k;q) w(\eta_k;N_k;q)\right)\\
		=& \prod_{k=1}^M r(\zeta_k,\xi_k;h^{-}_{k-1}(\zeta),h^{+}_{k+1}(\xi),v,N_k;q),
	\end{align*}
	where we used Lemma \ref{lem:krawthouckracah} in the last step.
	\end{proof}
	The multivariate $q$-Racah polynomials $R^v$ are orthogonal with respect to the reversible measures $W_\mathsf{L}$ and $W_\mathsf{R}$ of ASEP$_\mathsf{L}$ and ASEP$_\mathsf{R}$.
	\begin{proof}[Proof of Theorem \ref{Thm:orthogonaldualityLASEPRASEP}]
	For (ii), recall from the alternative proof of Theorem \ref{Thm:reversiblemeasures}(ii) in Section \ref{subsec:RevDynASEP} that the unitary operator $\Lambda_\mathsf{R}:H\to H^\mathsf{R}$ was defined by
	\[
		(\Lambda_\mathsf{R}f)(\xi)=\langle f, K_\mathsf{R}(\cdot,\xi) \rangle_{H}.
	\]
	Note that since
	\[
		[\Lambda_\mathsf{R} K_\mathsf{L}^v(\cdot,\zeta)](\xi)=R^v(\zeta,\xi),
	\]
	the second equation from Theorem \ref{Thm:orthogonaldualityLASEPRASEP} can be written as
	\[
		\langle \Lambda_\mathsf{R} K_\mathsf{L}^v(\cdot,\zeta),\Lambda_\mathsf{R} K_\mathsf{L}^{v^{-1}}(\cdot,\zeta')\rangle_{H^\mathsf{R}}.
	\]
	Since $\Lambda_\mathsf{R}$ is unitary, we obtain that the above expression is equal to
	\[
	\langle  K_\mathsf{L}^v(\cdot,\zeta), K_\mathsf{L}^{v^{-1}}(\cdot,\zeta')\rangle_{H}=\frac{\delta_{\zeta,\zeta'}}{W_\mathsf{L}(\zeta;\vec{N},\lambda;q)},
	\]
	where we used Corollary \ref{Cor:reversiblemeasures}(iii) in the last step.
	One can use a similar approach to prove the first equality from \ref{Thm:orthogonaldualityLASEPRASEP}. Alternatively, it can also be obtained from the second equality we just proved by exploiting that ASEP$_\mathsf{L}$ is just ASEP$_\mathsf{R}$ on the reversed lattice. That is, by sending $q\to q^{-1}$, interchanging $\rho\leftrightarrow\lambda$ and reversing the order of sites and using the first point of Remark \ref{rem:duality racah}. 
\end{proof}

\section{Limit calculations} \label{sec:calculations of limits}
In this section, we carry out the explicit calculations to derive the duality results stated in Section \ref{sec:Degenerations} from the (almost) self-duality of dynamic ASEP stated in \eqref{eq:dualityLASEPRASEP} with respect to the $q$-Racah duality functions $R^v$, see Theorem \ref{Thm:KrawtchoukRacah}. The corresponding 1-site duality functions $r$ defined in \eqref{eq:1-site duality r} are given explicitly in terms of $q$-hypergeometric functions by
\begin{equation} \label{eq:r=4phi3}
 \begin{split}r(\zeta_k,&\xi_k;h_{k-1}^-(\zeta),h_{k+1}^+(\xi),v,N_k;q)=c_r(\zeta_k,\xi_k;h_{k-1}^-(\zeta),h_{k+1}^+(\xi),v,N_k;q) \\
 &\times \rphis{4}{3}{q^{-2\xi_k}, -q^{2h^{+}_{k+1}(\xi)-2N_k+2\xi_k}, q^{-2\zeta_k}, -q^{2h^{-}_{k-1}(\zeta)-2N_k+2\zeta_k} }{-v^{-1}  q^{h^{+}_{k+1}(\xi)+h^{-}_{k-1}(\zeta)-N_k+1}, -v q^{h^{+}_{k+1}(\xi)+h^{-}_{k-1}(\zeta)-N_k+1},q^{-2N_k}}{q^2,q^2},\end{split}
\end{equation}
where 
\[
c_r(y,x;\la,\rho,v,N;q)=v^{y}\frac{(-vq^{\rho+\lambda-N+1};q^2)_{x} (-v^{-1}q^{\rho+\lambda-N+1};q^2)_{y} (vq^{2y-\rho+\lambda-N+1};q^2)_{N}}{q^{y(y+\rho+\lambda-N)}(vq^{-2x-\rho+\lambda+N+1};q^2)_{x+y}}.
\]
With this explicit form, it is quite straightforward to compute the appropriate limits.

\subsection{Limits of reversible measures} \label{ssec:limits reversible measure}
Before we consider limits of duality functions we first consider limits of reversible measures $W_{\mathsf R}$ and $W_{\mathsf L}$, defined by \eqref{eq:Weightfunction WR} and \eqref{eq: WL=WR-rev}, respectively. We show that these reversible measures tend to the reversible measure $w$ given by \eqref{eq:Weightfunction w} when the dynamic parameters $\rho$ and $\la$ tend to $\pm \infty$.

Let $0<q<1$, then the 1-site weight functions $w$ and $W$, \eqref{eq:orthokrawtchoukASEP} and \eqref{eq:orthokrawtchoukRASEP}, satisfy
\[
\begin{split}
	\lim_{\rho \to \infty} q^{2\rho(\xi_k-N_k)} W(\xi_k;N_k,h^{+}_{k+1}(\xi);q) &= q^{2N_kh^{+}_{k+1,0}(\xi)}q^{-N_k(N_k-1)}q^{-\xi_k(2h^{+}_{k+1,0}(\xi)+1+\xi_k-2N_k)} \qbinom{q^2}{N_k}{\xi_k} \\
	& = q^{-2h^{+}_{k+1,0}(\xi)(\xi_k-N_k)}q^{-N_k(N_k-1)} q^{\xi_k(3N_k-2\xi_k-1)} w(\xi_k;N_k;q),
\end{split}
\]
where $h^{+}_{k,0}=\sum_{j=k}^M (2\xi_j-N_j)$ and we used $(-aq^{-2\rho};q^2)_N = a^N q^{N(N-1)}q^{-2\rho N} + \mathcal O(q^{-2\rho(N-1)})$ for $\rho \to \infty$,
and
\[
\begin{split}
	\lim_{\rho \to -\infty} q^{2\rho\xi_k} W(\xi_k;N_k,h^{+}_{k+1}(\xi);q) &= q^{-\xi_k(2h^{+}_{k+1,0}(\xi)-1+\xi_k)} \qbinom{q^2}{N_k}{\xi_k} \\& = q^{-\xi_k(2\xi_k-1-N_k+2h^{+}_{k+1,0}(\xi))} w(x_k;N_k;q).
\end{split}
\]
Recalling that $w(x;N_k;q)=w(x;N_k;q^{-1})$ and using \eqref{eq:identity |A||B|} and $\sum_{k=1}^M N_k [\xi_k+\sum_{j=k+1}^M 2\xi_j] = -u(\xi;\vec{N})$, it follows that
\[
\begin{split}
	\lim_{\rho \to \infty} & q^{2\rho(|\xi|-|\vec{N}|)} W_{\mathsf R}(\xi;\vec{N},\rho;q) \\
	& = q^{|\vec{N}|-|\xi| - \sum_{k=1}^M \big( N_k [N_k- 2\sum_{j=k+1}^M (2\xi_j-N_j) ] + 2\xi_k [\xi_k + \sum_{j=k+1}^M(2\xi_j-N_j)] - 3\xi_k N_k\big)} \prod_{k=1}^M w(\xi_k;N_k;q) \\ 
	& = q^{|\vec{N}|-|\xi|(|\xi|+1)-  (|\vec{N}|-|\xi|)^2 } w(\xi;\vec{N};q^{-1}),
\end{split}
\]
and
\[
\begin{split}
	\lim_{\rho \to -\infty} q^{2\rho|\xi|} W_{\mathsf R}(\xi;\vec{N},\rho;q) &=  q^{|\xi|-\sum_{k=1}^M \xi_k [2\xi_k+\sum_{j=k+1}^M 4\xi_j-\sum_{j=k+1}^M 2N_j - N_k]} \prod_{k=1}^M w(\xi_k;N_k;q) \\
	& = q^{|\xi|(1-2|\xi|+2|\vec{N}|)} w(\xi;\vec{N};q),
\end{split}
\]
which proves \eqref{eq:rho->infty W->w}. 

Similarly,
\[
\begin{split}
	\lim_{\la \to \infty } q^{-2\zeta_k \la} W(\zeta_k;N_k, h_{k-1}^-(\zeta);q^{-1}) &= q^{\zeta_k[-1+2\zeta_k-N_k+2h_{k-1,0}^-(\zeta)]} w(\zeta_k;N_k;q) \\
	\lim_{\la \to -\infty } q^{2\la(N_k-\zeta_k)} W(\zeta_k;N_k, h_{k-1}^-(\zeta);q^{-1}) & = q^{2h_{k-1,0}^-(\zeta)(\zeta_k-N_k)} q^{N_k(N_k-1)} q^{-\zeta_k(3N_k-2\zeta_k-1)} w(\zeta_k;N_k;q),
\end{split}
\]
from which it follows that
\[
\begin{split}
	\lim_{\la \to \infty} q^{-2\la|\zeta|} W_{\mathsf L}(\zeta;\vec{N},\la;q) &= q^{|\zeta|(2|\zeta|-1)}w(\zeta;\vec{N};q),\\
	\lim_{\la \to -\infty} q^{2\la(|\vec{N}|-|\zeta|)} W_{\mathsf L}(\zeta;\vec{N},\la;q) & = q^{-|\vec{N}|+|\zeta|(|\zeta|+1)+ 2(|\vec{N}|-|\zeta|)^2-2|\zeta| |\vec{N}| } w(\zeta;\vec{N};q^{-1}).
\end{split}
\]
This proves \eqref{eq:la->pm infty W->w}.

\subsection{Proof of Propositions \ref{prop:degenerate duality}(i) and \ref{prop:degenerate orthogonality}(i)}\ \\
We take the limit $\lambda \to \infty$ of a multiple of the duality function $R$ defined as a product of $1$-site duality functions $r$ by \eqref{eq:r=4phi3}. We replace $v$ by $vq^{-\lambda}$ and then let $\lambda \to \infty$. This gives
\[
\begin{split}
	\lim_{\lambda \to \infty} &q^{2\lambda \zeta_k}r(\zeta_k,\xi_k;vq^{-\lambda})\\ & =v^{\zeta_k}\frac{ (-vq^{h^{+}_{k+1}(\xi)+h^{-}_{k-1,0}(\zeta)-N_k+1};q^2)_{\xi_k} (vq^{1+2\zeta_k -h^{+}_{k+1}(\xi) +h^{-}_{k-1,0}(\zeta) - N_k};q^2)_{N_k}}{q^{\zeta_k(\zeta_k+h^{+}_{k+1}(\xi)+h^{-}_{k-1,0}(\zeta)-N_k)} (vq^{1-2\xi_k -h^{+}_{k+1}(\xi) +h^{-}_{k-1,0}(\zeta) + N_k};q^2)_{\zeta_k+\xi_k} } \\
	& \quad \times 
	\rphis{3}{2}{q^{-2\xi_k}, -q^{2h^{+}_{k+1}(\xi)-2N_k+2\xi_k}, q^{-2\zeta_k}}{-vq^{h^{+}_{k+1}(\xi)+h^{-}_{k-1,0}(\zeta)-N_k+1}, q^{-2N_k}}{q^2,q^2}\\ &  = p(\zeta_k,\xi_k;h^{-}_{k-1,0}(\zeta),N_k,h^{+}_{k+1}(\xi),v,N_k;q),
\end{split}
\]
where we use the 1-site duality function $p$ defined in \eqref{eq:1sitedualityHahn}.
It follows that
\[
\lim_{\lambda \to \infty} q^{2\lambda|\zeta|} R^{vq^{-\la}}(\zeta,\xi) = P^v_{\mathsf R}(\zeta,\xi),
\]
which proves Proposition \ref{prop:degenerate duality}(i). In the same way, we find
\[
\begin{split}
	\lim_{\lambda \to \infty} r(\zeta_k,\xi_k;v^{-1}q^{\lambda}) & = v^{-\zeta_k}\frac{  (-vq^{h^{+}_{k+1}(\xi)+h^{-}_{k-1,0}(\zeta)-N_k+1};q^2)_{\zeta_k}}{(q^{\zeta_k+h^{+}_{k+1}(\xi)+h^{-}_{k-1,0}(\zeta)-N_k})^{\zeta_k}} 	\\
	& \quad \times	\rphis{3}{2}{q^{-2\xi_k}, -q^{2h^{+}_{k+1}(\xi)-2N_k+2\xi_k}, q^{-2\zeta_k}}{-vq^{h^{+}_{k+1}(\xi)+h^{-}_{k-1,0}(\zeta)-N_k+1}, q^{-2N_k}}{q^2,q^2}.
\end{split}
\]
It follows that
\[
\lim_{\lambda \to \infty}R^{v^{-1}q^\lambda}(\zeta,\xi)=  v^{-2|\zeta|}\frac{C^v(|\zeta|,|\xi|;0,\rho)}{c^{v}(|\zeta|,|\xi|;0,\rho)} P^v_{\mathsf R}(\zeta,\xi), 
\]
where $c^v$ is the function given in \eqref{eq:invariant c def},
\begin{equation*} 
	\begin{split}
		c^v(|\zeta|,|\xi|;\lambda,\rho) &= \prod_{k=1}^M \frac{ (vq^{2\zeta_k-h^{+}_{k+1}(\xi)+h^{-}_{k-1}(\zeta)-N_k+1};q^2)_{N_k} }{ (vq^{-2\xi_k-h^{+}_{k+1}(\xi)+h^{-}_{k-1}(\zeta)+N_k+1};q^2)_{\xi_k+\zeta_k} } =  \frac{ (vq^{\lambda-\rho+2|\zeta|-|\vec{N}|+1};q^2)_{|\vec{N}|-|\zeta|} }{(vq^{\lambda-\rho-2|\xi|+|\vec{N}|+1};q^2)_{|\xi|}},
	\end{split}
\end{equation*} and $C^v$ is given in \eqref{eq:identity C^0},
\begin{equation*} 
C^v(|\zeta|,|\xi|;\lambda,\rho) = \prod_{k=1}^M \frac{(-vq^{h^{+}_{k+1}(\xi)+h^{-}_{k-1}(\zeta)-N_k+1};q^2)_{\xi_k}}{ (-vq^{h^{+}_{k+1}(\xi)+h^{-}_{k-1}(\zeta)-N_k+1};q^2)_{\zeta_k}} = 	\frac{(-vq^{\lambda+\rho-|\vec{N}|+1};q^2)_{|\xi|}}{(-vq^{\lambda+\rho-|\vec{N}|+1};q^2)_{|\zeta|}}.
\end{equation*}

Finally, replacing $v$ by $vq^{-\la}$ in the orthogonality relations from Theorem \ref{Thm:orthogonaldualityLASEPRASEP} and letting $\la \to \infty$ using the limits of $W_{\mathsf L}$ form Section \ref{ssec:limits reversible measure}, gives the orthogonality relations for $P_{\mathsf R}^v$ as stated in Proposition \ref{prop:degenerate orthogonality}(i).

\subsection{Proof of Propositions \ref{prop:degenerate duality}(ii) and \ref{prop:degenerate orthogonality}(ii)}\ \\
For the first statement, let $\rho \to -\infty$ in $v= $ in $P_{\mathsf R}^{vq^{\rho}}$ to obtain
\[
\begin{split}
	\lim_{\rho \to -\infty} & \rphis{3}{2}{q^{-2\eta_k}, -q^{2h^{+}_{k+1}(\xi)-2N_k+2\xi_k}, q^{-2\xi_k}}{-v q^{\rho+h^{+}_{k+1}(\xi)+h^{-}_{k-1,0}(\eta)-N_k+1} , q^{-2N_k}}{q^2,q^2} =  \\ &=\rphis{2}{1}{q^{-2\eta_k},  q^{-2\xi_k}}{q^{-2N_k}}{q^2, v^{-1} q^{1+2\xi_k+h^{+}_{k+1,0}(\xi)-h^{-}_{k-1,0}(\eta)-N_k} } \\
	& = K_{\xi_k}^{\text{qtm}}(\eta_k; p_{k,v}(\eta,\xi),N_k;q^2),
\end{split}
\]
where $p_{k,v}(\eta,\xi) = v^{-1} q^{h^{+}_{k+1,0}(\xi)-h^{-}_{k-1,0}(\eta)-N_k-1}$. Then using
\[
\lim_{\rho \to -\infty}v^{-2\eta_k} q^{-2\rho \eta_k}(v q^{N_k-h^{+}_{k+1,0}(\xi)-h^{-}_{k-1,0}(\eta)-\eta_k})^{\eta_k} (-v q^{2\rho+h^{+}_{k+1,0}(\xi)+h^{-}_{k-1,0}(\eta)-N_k+1};q^2)_{\eta_k}  = 1,
\]
it follows that
\[
\lim_{\rho \to -\infty} \frac{ (v^{-2}q^{-2\rho})^{|\eta|} }{c^v(|\eta|,|\xi|;0,0)C^v(|\eta|,|\xi|;0,2\rho)} P_{\mathsf R}^{v q^\rho}(\eta,\xi) = K^v_\mathrm{qtm}(\eta,\xi).
\]
Using the above limit relations and letting $\rho\to-\infty$ in Proposition \ref{prop:degenerate orthogonality}(i) then gives the orthogonality of \ref{prop:degenerate orthogonality}(ii).

\subsection{Proof of Propositions \ref{prop:degenerate duality}(iii) and \ref{prop:degenerate orthogonality}(iii)}\ \\
Letting $\rho\to\infty$ in $P_{\mathsf R}^{v q^{-\rho}}$ we get
\[
\begin{split}
	\lim_{\rho \to \infty} &\rphis{3}{2}{q^{-2\eta_k}, -q^{2h^{+}_{k+1}(\xi)-2N_k+2\xi_k}, q^{-2\xi_k}}{-v q^{h^{+}_{k+1,0}(\xi)+h^{-}_{k-1,0}(\eta)-N_k+1}, q^{-2N_k}}{q^2,q^2}  \\
	&\hspace{1cm}= \rphis{3}{2}{q^{-2\eta_k}, 0, q^{-2\xi_k}}{-v q^{h^{+}_{k+1,0}(\xi)+h^{-}_{k-1,0}(\eta)-N_k+1}, q^{-2N_k}}{q^2,q^2} \\&\hspace{1cm} =  K_{\xi_k}^{\text{aff}}(\eta_k;p'_{k,v}(\eta,\xi),N_k;q^2),
\end{split}
\]
where $p'_{k,v}(\eta,\xi) = -v q^{h^{+}_{k+1,0}(\xi)+h^{-}_{k-1,0}(\eta)-N_k-1}$, from which it follows that
\[
\lim_{\rho \to \infty} \frac{ q^{2\rho|\eta|}  }{c^{-v}(|\eta|,|\xi|;0,2\rho)} P^{-v q^{-\rho}}_{\mathsf R}(\eta,\xi) = K^v_\mathrm{aff}(\eta,\xi). 
\]
Again, using the above limit relations and letting $\rho\to\infty$ in Proposition \ref{prop:degenerate orthogonality}(i) gives the orthogonality of \ref{prop:degenerate orthogonality}(iii).
\medskip

\subsection{Proof of Proposition \ref{prop:KR->triangular duality}}
For the first statement, we use 
\begin{align}
	q^{-\rho|\eta|}K_\mathsf{R}(\eta,\xi) = q^{-\half u(\eta,\vec{N})}\prod_{k=1}^M (-1)^{\eta_k} q^{\half\eta_k(N_k-1)}q^{-\eta_k(\rho+h_{k+1}^+)}K_{\eta_k}(\xi_k;q^{2h_{k+1}^+},N_k;q^2).\label{eq:KrawtchoukTriangular}
\end{align}
Writing out the explicit expression of the $q$-Krawtchouk polynomials in terms of $q$-hypergeometric functions we obtain
\begin{align*}
	\lim\limits_{\rho\to-\infty} q^{-2\eta_k\rho} K_{\eta_k}(\xi_k;q^{2h_{k+1}^+},N_k;q^2)  &= \lim\limits_{\rho\to-\infty}q^{-2\eta_k\rho} \rphis{3}{2}{q^{-2\eta_k}, q^{-2\xi_k},  -q^{2h^{+}_{k+1}+2\xi_k-2N_k} }{q^{-2N_k},0}{q^2,q^2}\\
	&= (-1)^{\eta_k} \frac{(q^{-2\xi_k};q^2)_{\eta_k} }{(q^{-2N_k};q^2)_{\eta_k} } q^{\eta_k(\eta_k+2h_{k+1}^++2\xi_k-2N_k+ 1)}
\end{align*}
if $\eta_k\leq \xi_k$ and $0$ otherwise. Taking the limit of $\rho\to-\infty$ in \eqref{eq:KrawtchoukTriangular} then gives the desired outcome.\\
\\
For the second statement we use the identity \eqref{eq:KrawtchoukRhoto-Rho} with $\rho$ replaced by $-\rho$,
\[
k(n,x;\rho,N;q) = (-1)^n k(n,N-x;-\rho,N;q)
\]
to obtain
\begin{align}
(-q^{\rho})^{|\eta|}K_{\mathsf R}(\eta,\xi) =q^{-\half u(\eta,\vec{N})}\prod_{k=1}^M (-1)^{\eta_k} q^{\half\eta_k(N_k-1)}q^{\eta_k(\rho+h_{k+1}^+)}K_{\eta_k}(N_k-\xi_k;q^{-2h_{k+1}^+},N_k;q^2) \label{eq:KrawtchoukTriangular2}
\end{align}
Writing out the explicit expression of the $q$-Krawtchouk polynomials again, we obtain
\begin{align*}
	\lim\limits_{\rho\to\infty} q^{2\eta_k\rho} K_{\eta_k}(N_k-\xi_k;q^{-2h_{k+1}^+},N_k;q^2)  &= \lim\limits_{\rho\to\infty}q^{2\eta_k\rho} \rphis{3}{2}{q^{-2\eta_k}, q^{2\xi_k-2N_k},  -q^{-2h^{+}_{k+1}-2\xi_k} }{q^{-2N_k},0}{q^2,q^2}\\
	&= (-1)^{\eta_k} \frac{(q^{2\xi_k-N_k};q^2)_{\eta_k} }{(q^{-2N_k};q^2)_{\eta_k} } q^{\eta_k(\eta_k-2h_{k+1}^+-2\xi_k-2N_k+ 1)}
\end{align*}
if $\eta_k\leq N_k-\xi_k$ and $0$ otherwise. Taking the limit of $\rho\to\infty$ in \eqref{eq:KrawtchoukTriangular2} then gives the desired outcome.
\subsection{Proof of Proposition \ref{prop:duality q-TAZRP}} 
This follows from writing the $q$-Racah duality function $R$ given by \eqref{eq:duality Racah polynomials} explicitly in terms of $_4\varphi_3$-functions using \eqref{eq:racah4phi3} and \eqref{eq:1-site duality r}, together with the following limit, 
\[
\begin{split}
	\lim_{N\to \infty} & \rphis{4}{3}{q^{-2\zeta_k}, -q^{2h^{+}_{k+1}(\xi)-2N+2\xi_k}, q^{-2\xi_k}, -q^{2h^{-}_{k-1}(\zeta)-2N+2\zeta_k} }{-v^{-1}  q^{h^{+}_{k+1}(\xi)+h^{-}_{k-1}(\zeta)-N+1}, -v q^{h^{+}_{k+1}(\xi)+h^{-}_{k-1}(\zeta)-N+1},q^{-2N}}{q^2,q^2} \\
	& = \rphis{2}{0}{q^{-2\zeta_k},q^{-2\xi_k}}{-}{q^2,q^{2\zeta_k+2\xi_k}} = q^{2\zeta_k \xi_k},
\end{split}
\]
where we used a limit case of the $q$-Chu-Vandermonde summation identity \cite[(II.6)]{GR} in the last step.

\subsection{Proof of Proposition \ref{prop:duality Racah}}
Let us first recall the following identities,
\[
\lim_{q \to 1} \frac{ (q^{2a};q^2)_n }{(1-q^2)^n} = (a)_n, \qquad a \in \R, \ n \in \Z_{\geq 0},
\]
and, for $a_1,\ldots,a_{r}, b_1,\ldots,b_r \in \R$ and $n \in \Z_{\geq 0}$,
\[
\lim_{q \to 1} \rphis{r+1}{r}{q^{-n},q^{a_1}, \ldots, q^{a_{r}}}{q^{b_1},\ldots,q^{b_r}}{q,z} = \rFs{r+1}{r}{-n,a_1,\ldots,a_{r}}{b_1,\ldots,b_r}{z}.
\]
Then we have the following limit of the 1-site duality function $r$ defined by \eqref{eq:1-site duality r},
\[
\begin{split}
\lim_{q \to 1} &(1-q^2)^{-N} r(y,x;\lambda+\pi i/2\ln(q),\rho+i\pi/2\ln(q),q^v,N;q)  \\
&=\frac{(\frac12(\rho+\lambda-N+v+1))_x (\frac12(\rho+\lambda-N-v+1))_y (y+\frac12(\lambda-\rho-N+v+1))_N }{ (-x+\frac12(\lambda-\rho+N+v+1))_{x+y} } \\
&\quad \times \rFs{4}{3}{-y,\rho-N+y,-x,\la-N+x}{\frac12(\rho+\la-N+1-v), \frac12(\rho+\la-N+1+v),-N}{1}\\
& = \hat{r}(y,x;\la,\rho,v,N),
\end{split}
\]
where $\hat{r}$ is the 1-site duality function given in Proposition \ref{prop:duality Racah}. It follows that
\[
\lim_{q \to 1} (1-q^2)^{-|\vec{N}|} R^{q^v}(\zeta,\xi) = \hat{R}^v(\zeta,\xi).
\]
where $\lambda\to\lambda+\pi i/2\ln(q)$ and $\rho\to \rho+i\pi/2\ln(q)$ in $R^v$.
The orthogonality relations follow by taking the limit in the orthogonality relations in Theorem \ref{Thm:orthogonaldualityLASEPRASEP} using the above limit and the limits \eqref{eq:WR WL q->1} of the reversible measures.

\subsection{Proof of Propositions \ref{prop:degenerate duality sym}(i) and \ref{prop:degenerate orthogonality sym}(i)}
Similar to the previous section we can take limits in the 1-site duality functions $p$ defined in \eqref{eq:1sitedualityHahn},
\[
\begin{split}
\lim_{q \to 1}& (1-q^2)^{y-N} p(y,x;\la,\rho+i\pi/2\ln(q),iq^v,N;q)= \\ 
& \frac{(\frac12(\rho+\la-N+v+1))_x (y+\frac12(\la-\rho-N+v+1))_N } {(-x+\frac12(\la-\rho+N+v+1))_{x+y}} 
 \rFs{3}{2}{-x,\rho-N+x,-y}{\frac12(\rho+\la-N+v+1), -N}{1},
\end{split}
\]
from which it follows that 
\[
\hat{P}_{\mathsf R}(\eta,\xi) = \lim_{q \to 1} (1-q^2)^{|\zeta|-|\vec{N}|} P^{iq^v}_{\mathsf R}(\eta,\xi),
\]
where $\rho\to \rho+i\pi/2\ln(q)$ in $P_{\mathsf{R}}^{iq^v}$.
The orthogonality relations of Proposition \ref{prop:degenerate orthogonality sym}(i) are obtained by letting $q\to 1$ in the orthogonality relations from Proposition \ref{prop:degenerate orthogonality}(i), using the stated limit of the duality function $P^v_{\mathsf R}$, the limits of the reversible measures $w$ and $W_{\mathsf R}$, and the following limits of the invariant functions $\om^p$ and $\om_\mathsf{R}^p$ (see Appendix \ref{app:overview}),
\[
\begin{split}
&\lim_{q \to 1} (1-q^2)^{|\vec{N}|-2x} \om^{\mathrm{p}}(x;\vec{N},\rho+i\pi/2\ln(q),iq^v;q) =\frac{(-1)^{x}(\frac12(\rho-|\vec{N}|+v+1))_x}{ (x+\frac12(v-\rho-|\vec{N}|+1))_{|\vec{N}|-x}},\\
&\lim_{q \to 1} \om_\mathsf{R}^{\mathrm p}(x;\vec{N},\rho+i\pi/2\ln(q),iq^v;q) = \frac{ (-x+\tfrac12(v-\rho+|\vec{N}|+1))_x}{ (\tfrac12(v+\rho-|\vec{N}|+1))_x}.
\end{split}
\]

\subsection{Proof of Propositions \ref{prop:degenerate duality sym}(ii) and \ref{prop:degenerate orthogonality sym}(ii)}
This follows from writing $K_{\mathsf R}$ in terms of the $1$-site duality functions $k$, see Appendix \ref{app:overview}. Writing the function $k$ in terms of a $_3\varphi_2$-function gives
\[
\begin{split}
	\lim_{q \to 1} (-1)^n v^{\frac12 \eta_k} q^{\frac12\eta_k(N_k-1)} &\rphis{3}{2}{q^{-2\eta_k},q^{-2\xi_k}, -v q^{2h^+_{k+1,0}(\xi)+2\xi_k-2N_k}}{q^{-2N_k},0}{q^2,q^2} \\ &=v^{\frac12 \eta_k} \rFs{2}{1}{-\eta_k,-\xi_k}{-N_k}{1+v} \\& = \hat{K}_{\eta_k}(\xi_k;\tfrac1{1+v},N_k).
\end{split}
\]
The orthogonality for $\hat K$ from Proposition \ref{prop:degenerate orthogonality sym}(ii) then follows from taking the limit in the orthogonality relations for $K_{\mathsf R}$ from Theorem \ref{Thm:reversiblemeasures}, using the limits of the reversible measures $w$ and $W_{\mathsf R}$.

\newpage
	
\appendix
\section{Overview duality functions and coefficients}\label{app:overview}
The duality function on top of all other duality functions in this paper is the multivariate $q$-Racah polynomial $R(\eta,\xi;v)$, where $\eta$ is the configuration of $\LASEP{q,\vec{N},\lambda$}, $\xi$ of $\RASEP{q,\vec{N},\rho}$ and $v\in\R^\times$ is a free parameter. When taking appropriate limits of $\lambda,\rho\to\pm\infty$, when can obtain other dualities, found in the table below. \\
\begin{table}[h]
	\begin{tabular}{|l|ll|l|l|l|}
		\hline
		Duality &\multicolumn{1}{|l|}{$\eta$-process} & $\xi$-process & Corresponding limit& DF (1)& DF (2) \\  \hline
		(i)&\multicolumn{1}{|l|}{$\ASEP{q,\vec{N}}$}       & $\RASEP{q,\vec{N},\rho}$      &       $\lambda\to\infty$    &   $P^v_\mathsf{R}(\eta,\xi)$ & $K_\mathsf{R}(\eta,\xi)$       \\ \hline
		(ii)&\multicolumn{1}{|l|}{$\ASEP{q,\vec{N}}$ }      &  $\ASEP{q,\vec{N}}$       &  $\lambda\to\infty$, $\rho\to-\infty$ &      $K^v_\mathrm{qtm}(\eta,\xi)$ & Triangular          \\ \hline
		(iii)&\multicolumn{1}{|l|}{$\ASEP{q,\vec{N}}$}       &  $\ASEP{q^{-1},\vec{N}}$       &     $\lambda,\rho\to\infty$    &     $K^v_\mathrm{aff}(\eta,\xi)$  & Triangular      \\ \hline
	\end{tabular}\\
	\hspace{0.3cm}
	\caption{Overview dualities, where DF is short for duality function.}
\end{table} \vspace{-0.7cm}\\
 The functions are given by
\begin{itemize}
	\item $\displaystyle R^v(\zeta,\xi)= \prod_{k=1}^M r( \zeta_k,\xi_k;h^-_{k-1}(\zeta),h^+_{k+1}(\xi),v,N_k;q) $,
	\item $\displaystyle P^v_\mathsf{R}(\eta,\xi)=\prod_{k=1}^M p(\eta_k,\xi_k;h^-_{k-1,0}(\eta),h^+_{k+1}(\xi),v,N_k;q)$,
	\item $\displaystyle K_\mathsf{R}(\eta,\xi)=q^{-\half\sum_{k=1}^M  \Big(\eta_kN_k-2\eta_k\sum_{j=1}^{k}  N_j \Big)}\prod_{k=1}^M  k(\eta_k,\xi_k;h^+_{k+1}(\xi),N_k;q)$,
	\item $\displaystyle K^v_\mathrm{qtm}(\eta,\xi)=\prod_{k=1}^M k^\mathrm{qtm}(\eta_k,\xi_k;h^-_{k-1,0}(\eta),h^+_{k+1,0}(\xi),v,N_k;q)$,
	\item $\displaystyle K^v_\mathrm{aff}(\eta,\xi)=\prod_{k=1}^M k^\mathrm{aff}(\eta_k,\xi_k;h^-_{k-1,0}(\eta),h^+_{k+1,0}(\xi),v,N_k;q)$,
\end{itemize}
where the 1-site duality functions are given by,
\begin{itemize}
	\item $\displaystyle r(y,x)=c_{\mathrm r}(y,x;\lambda,\rho,v,N;q) R_x(y;-v^{-1}q^{\rho+\lambda-N-1}, vq^{\rho-\lambda-N-1}, q^{-2N-2}, -q^{2\lambda};q^2)$,\vspace{0.1cm} 
	\item $\displaystyle p(n,x) =  c_{\mathrm p}(n,x;\lambda,\rho,v,N;q) P_x(n;-vq^{\rho+\lambda-N-1}, v^{-1}q^{\rho-\lambda-N-1},N;q^2)$,\vspace{0.1cm} 
	\item $\displaystyle k(n,x) = c_{\mathrm k}(n;\rho,N;q)  K_n(x;q^{2\rho},N;q^2)$,\vspace{0.1cm} 
	\item $\displaystyle k^{\mathrm{qtm}}(n,x)=K_{x}^{\mathrm{qtm}}(n; v^{-1} q^{\rho-\lambda-N-1},N;q^2)$,\vspace{0.1cm} 
	\item $\displaystyle k^\mathrm{aff}(n,x)= c_{\mathrm k}^\mathrm{aff}(x;\lambda,\rho,v,N;q)K_{x}^{\mathrm{aff}}(n;v q^{\rho+\lambda-N-1},N;q^2)$,\vspace{0.1cm} 
\end{itemize}
the coefficients by
\begin{itemize}
	\item $\displaystyle c_{\mathrm r}(y,x;\lambda,\rho,v,N;q) = v^{y}\frac{(-vq^{\rho+\lambda-N+1};q^2)_{x} (-v^{-1}q^{\rho+\lambda-N+1};q^2)_{y} (vq^{2y-\rho+\lambda-N+1};q^2)_{N}}{q^{y(y+\rho+\lambda-N)}(vq^{-2x-\rho+\lambda+N+1};q^2)_{x+y}} $,\vspace{0.2cm}\\
	\item $\displaystyle c_{\mathrm p}(n,x;\lambda,\rho,v,N;q) = v^n  \frac{(-vq^{\rho+\lambda-N+1};q^2)_x (vq^{2n-\rho+\lambda-N+1};q^2)_N}{q^{n(n+\rho+\lambda-N)}(vq^{-2x-\rho+\lambda+N+1};q^2)_{x+n}}$,\vspace{0.0cm}\\
	\item $\displaystyle c_{\mathrm k}(n;\rho,N;q)=(-1)^n  q^{-n\rho} q^{\frac12n(N-1)}$,\\
	\item $\displaystyle c_{\mathrm k}^\mathrm{aff}(n,x;\lambda,\rho,v,N;q)=(-v)^n\frac{(v q^{\rho+\lambda-N+1};q^2)_{x}}{q^{n(n+\rho+\lambda-N)}}  $,\\
\end{itemize}
and the polynomials by 
\begin{itemize}
	\item $\displaystyle R_n(x;\alpha,\beta,\ga,\de;q) = \rphis{4}{3}{q^{-n},\alpha\beta q^{n+1}, q^{-x}, \gamma \delta q^{x+1} }{\alpha q, \beta \delta q, \gamma q}{q,q}$\hspace{0.5cm} ($q$-Racah),
	\item $\displaystyle P_n(x;\alpha,\beta,N;q) = \rphis{3}{2}{q^{-n},\alpha\beta q^{n+1},q^{-x}}{\alpha q, q^{-N}}{q,q}$\hspace{1.95cm} ($q$-Hahn),
	\item $\displaystyle K_n(x;c,N;q) = \rphis{3}{2}{q^{-n}, q^{-x}, -cq^{x-N} }{q^{-N}, 0 }{q,q}.$\hspace{2.1cm} ((Dual) $q$-Krawtchouk),
	\item $\displaystyle K^{\text{qtm}}_n(x;p,N;q)= \rphis{2}{1}{q^{-n}, q^{-x}}{q^{-N}}{q, pq^{n+1}}$\hspace{2.5cm} (Quantum $q$-Krawtchouk),
	\item $\displaystyle K^{\text{aff}}_n(x;p,N;q) = \rphis{3}{2}{q^{-n},0,q^{-x}}{pq,q^{-N}}{q,q}$\hspace{3.1cm} (Affine $q$-Krawtchouk).
\end{itemize}
The factors in front of the reversible measures in Proposition \ref{prop:degenerate orthogonality} are given by
\begin{itemize}
	\item 	$\displaystyle\om^p(x) =  \frac{v^{-2x}q^{x(2x-1)}(-vq^{\rho-|\vec{N}|+1};q^2)_x}{(vq^{-\rho+2x-|\vec{N}|+1};q^2)_{|\vec{N}|-x}}$,\vspace{0.2cm}\\
	\item	$\displaystyle\om_\mathsf{R}^p(x)= \frac{(vq^{-\rho-2x+|\vec{N}|+1};q^2)_x}{(-vq^{\rho-|\vec{N}|+1};q^2)_{x}}$,\\
	\item	$\displaystyle		\om^\mathrm{qtm}(x) = v^{-x} q^{x(x+|\vec{N}|-1)}(v q^{2x-N+1};q^2)_{|\vec{N}|-x}$,\\
	\item	$\displaystyle	\om^\mathrm{qtm}_\mathsf{R}(x) = \frac{v^x  q^{x(|\vec{N}|-x+1)}  }{ (v q^{1+|\vec{N}|-2x};q^2)_x}$,
	\item	$\displaystyle\om^\mathrm{aff}(x) = v^{|\vec{N}|-3x} q^{x(|\vec{N}|+x-1)}  (v q^{1-|\vec{N}|};q^2)_x$,\\
	\item	$\displaystyle	\om_{\mathsf{R}}^\mathrm{aff}(x)  = \frac{q^{|\vec{N}|(1-|\vec{N}|)+x(|\vec{N}|-x-1)}}{(v q^{1-|\vec{N}|};q^2)_x}$.
\end{itemize}
The factors in front of the reversible measures in Proposition \ref{prop:degenerate orthogonality sym} are given by
\begin{itemize}
	\item $\displaystyle \hat{\om}^{\mathrm p}(x) = \frac{(-1)^{x}(\frac12(\rho-|\vec{N}|+v+1))_x}{ (x+\frac12(v-\rho-|\vec{N}|+1))_{|\vec{N}|-x}}$,
	\item $\displaystyle \hat{\om}^{\mathrm p}_\mathsf{R}(x)= \frac{ (-x+\tfrac12(v-\rho+|\vec{N}|+1))_x}{ (\tfrac12(v+\rho-|\vec{N}|+1))_x}$,
	\item $\displaystyle \hat \om^{\mathrm k}(x) = \frac{ v^{-x} }{(1+\frac1v)^{|\vec{N}|}}$.
\end{itemize}

\section{Identities} \label{sec:appendix Identities}
In this section, we state some useful identities we use throughout the paper.
\begin{itemize}
\item For $A,B \in \R^M$,
\begin{equation} \label{eq:identity |A||B|}
\begin{split}
|A||B| & =\sum_{k=1}^M \sum_{j=k+1}^M (A_k B_j + A_jB_k)  + \sum_{k=1}^M A_k B_k \\
& = \sum_{k=1}^M \sum_{j=1}^{k-1} (A_k B_j + A_jB_k) + \sum_{k=1}^M A_kB_k.
\end{split}
\end{equation}
\item We have the identity
\begin{equation} \label{eq:orthokrawtchoukASEP q<->q_inv}
	w(n;N;q)=w(n;N;q^{-1})
\end{equation}
for the single site weight function $w$ defined in \eqref{eq:orthokrawtchoukASEP}. The identity directly follows from the readily verified formula
\[(a^{-1};q^{-2})_n = (-1)^n a^{-n} q^{-n(n-1)} (a;q^2)_n,\qquad a\in\R^\times. \]
\item We verify the following identity, stated in \eqref{eq:invariant c}: for $q>0$ and $q\neq 1$,
\[
\prod_{k=1}^M   \frac{ (vq^{2\zeta_k-h^{+}_{k+1}(\xi)+h^{-}_{k-1}(\zeta)-N_k+1};q^2)_{N_k} }{ (vq^{-2\xi_k-h^{+}_{k+1}(\xi)+h^{-}_{k-1}(\zeta)+N_k+1};q^2)_{\xi_k+\zeta_k} }  = \frac{ (vq^{\lambda-\rho+2|\zeta|-|\vec{N}|+1};q^2)_{|\vec{N}|-|\zeta|} }{(vq^{\lambda-\rho-2|\xi|+|\vec{N}|+1};q^2)_{|\xi|}}.
\]
First, assume $q<1$. We use the identity
\begin{equation} \label{eq:q-shifted factorial identity}
(a;q)_n = \frac{ (a;q)_\infty}{(aq^n;q)_\infty},\qquad n=0,1,\ldots,
\end{equation} 
then we find that 
\[
\prod_{k=1}^M   \frac{ (vq^{2\zeta_k-h^{+}_{k+1}(\xi)+h^{-}_{k-1}(\zeta)-N_k+1};q^2)_{N_k} }{ (vq^{-2\xi_k-h^{+}_{k+1}(\xi)+h^{-}_{k-1}(\zeta)+N_k+1};q^2)_{\xi_k+\zeta_k} } 
\]
is equal to	
\[
\prod_{k=1}^M   \frac{ (vq^{2\zeta_k-h^{+}_{k+1}(\xi)+h^{-}_{k-1}(\zeta)-N_k+1};q^2)_{\infty} }{ (vq^{-2\xi_k-h^{+}_{k+1}(\xi)+h^{-}_{k-1}(\zeta)+N_k+1};q^2)_{\infty} }.
\]
Using
\[
h^-_{k-1}(\zeta)+2\zeta_k-N_k=h^-_k(\zeta)\quad \text{and} \quad h^+_{k+1}(\xi)+2\xi_k-N_k=h^+_k(\xi),
\]
this becomes
\[
\begin{split}
\prod_{k=1}^M   \frac{ (vq^{h^-_k(\zeta)-h^{+}_{k+1}(\xi)+1};q^2)_{\infty} }{ (vq^{h^{-}_{k-1}(\zeta)-h^+_k(\xi)+1};q^2)_{\infty} } &= \frac{(vq^{h_M^-(\zeta)-h_{M+1}(\xi)+1};q^2)_\infty}{(vq^{h_0^-(\zeta)-h_1^+(\xi)+1};q^2)_\infty} \\
& = \frac{(vq^{\la-\rho+2|\zeta|-|\vec{N}|+1};q^2)_\infty}{(vq^{\la-\rho-2|\xi|+|\vec{N}|+1};q^2)_\infty}.
\end{split}
\]
Using \eqref{eq:q-shifted factorial identity} again then gives the result for $0<q<1$. Since both sides of the obtained identity are meromorphic functions in $q$, it follows that the result still holds for $q>1$.\\ 

\item Next we verify identity \eqref{eq:identity C^0}: for $q>0$ and $q\neq1$,
\[
\prod_{k=1}^M \frac{(-vq^{h^{+}_{k+1}(\xi)+h^{-}_{k-1}(\zeta)-N_k+1};q^2)_{\xi_k}} {(-vq^{h^{+}_{k+1}(\xi)+h^{-}_{k-1}(\zeta)-N_k+1};q^2)_{\zeta_k}} = 
	\frac{(-vq^{\rho+\lambda-|\vec{N}|+1};q^2)_{|\xi|}}{(-vq^{\rho+\lambda-|\vec{N}|+1};q^2)_{|\zeta|}}.
\]
Indeed, in a similar fashion as before, we have
\[
\begin{split}
\prod_{k=1}^M \frac{(-vq^{h^{+}_{k+1}(\xi)+h^{-}_{k-1}(\zeta)-N_k+1};q^2)_{\xi_k}} {(-vq^{h^{+}_{k+1}(\xi)+h^{-}_{k-1}(\zeta)-N_k+1};q^2)_{\zeta_k}} 
&= \prod_{k=1}^M \frac {(-vq^{h^{+}_{k+1}(\xi)+h^{-}_{k-1}(\zeta)+2\zeta_k-N_k+1};q^2)_{\infty}}{(-vq^{h^{+}_{k+1}(\xi)+h^{-}_{k-1}(\zeta)+2\xi_k-N_k+1};q^2)_{\infty}} \\
& = \prod_{k=1}^M \frac {(-vq^{h^{+}_{k+1}(\xi)+h^{-}_{k}(\zeta)+1};q^2)_{\infty} }{(-vq^{h^{+}_{k}(\xi)+h^{-}_{k-1}(\zeta)+1};q^2)_{\infty}}\\
& =  \frac {(-vq^{\rho+\lambda+2|\zeta|-|\vec{N}|+1};q^2)_{\infty}} {(-vq^{\rho+\lambda+2|\xi|-|\vec{N}|+1};q^2)_{\infty}},
\end{split}
\]
from which the required identity follows.
\end{itemize}

\section{Coefficients recurrence relations $q$-Krawtchouk polynomials}\label{app:coefficients}
		\noindent Let 
		\[
		k(n,x;\rho):=k(n,x;\rho,N;q) = (-1)^n  q^{n\rho} q^{\frac12n(N-1)} K_n(x;q^{2\rho},N;q^2),
		\]
		be the 1-site duality functions from \eqref{eq:1siteduality}, where
		\[
			K_n(x;c,N;q) = \rphis{3}{2}{q^{-n}, q^{-x}, -cq^{x-N} }{q^{-N}, 0 }{q,q}.
		\]
		Then $k(n,x;\rho)$ satisfies the following three $q$-difference equations,
		\begin{align}
			q^{-2n}k(n,x;\rho) =&  a_{-1}(x) k(n,x-1;\rho) +a_0(x) k(n,x;\rho) +a_1(x) k(n,x+1;\rho),\label{eq:qdifkrawtchoukapp}\\
			\begin{split}q^{-2n}k(n,x;\rho) =&  a_{0,2}(x) k(n,x;\rho+2) +a_{-1,2}(x) k(n,x-1;\rho+2) \\
				&+a_{-2,2}(x) k(n,x-2;\rho+2),\end{split}\label{eq:dualqkrawtchoukxrho+app}\\
			\begin{split} q^{-2n}k(n,x;\rho) =&  a_{0,-2}(x) k(n,x;\rho-2)  +a_{1,-2}(x) k(n,x+1;\rho-2) \\
				&+a_{2,-2}(x) k(n,x+2;\rho-2).\end{split}\label{eq:dualqkrawtchoukxrho-app}
		\end{align} 
		The coefficients from \eqref{eq:qdifkrawtchoukapp} are given by
		\begin{align*}
			a_{-1}(x) &= -\frac{q^{4x+2\rho-4N-2}(1-q^{2x})(1+q^{2x+2\rho})}{(1+q^{4x+2\rho-2N-2})(1+q^{4x+2\rho-2N})}, \\
			a_0(x) &= -(a_{-1}(x)+a_1(x)-1),\\
						a_{1}(x) &= \frac{(1-q^{2x-2N})(1+q^{2x+2\rho-2N})}{(1+q^{4x+2\rho-2N})(1+q^{4x+2\rho-2N+2})},
		\end{align*}
		from \eqref{eq:dualqkrawtchoukxrho+app} by
		\begin{align*}
			a_{0,2}(x)&=\frac{(1+q^{2\rho+2x-2N})(1+q^{2\rho+2x-2N+2})}{(1+q^{2\rho+4x-2N})(1+q^{2\rho+4x-2N+2})}, \\ a_{-1,2}(x)&=\frac{(1+q^{-2})(1-q^{-2x})(1+q^{2\rho+2x-2N})}{(1+q^{2N-2\rho-4x-2})(1+q^{4x+2\rho-2N-2})}, \\
			a_{-2,2}(x)&=\frac{(1-q^{-2x})(1-q^{-2x+2})}{(1+q^{2N-2\rho-4x})(1+q^{2N-2\rho-4x+2})},
		\end{align*}
		and from \eqref{eq:dualqkrawtchoukxrho-app} by
		\begin{align*}
			a_{0,-2}(x)&=\frac{(1+q^{-2\rho-2x})(1+q^{-2\rho-2x+2})}{(1+q^{2N-2\rho-4x})(1+q^{2N-2\rho-4x+2})},\\
			a_{1,-2}(x)&= \frac{(1+q^{-2})(1-q^{2x-2N})(1+q^{-2\rho-2x})}{(1+q^{2\rho+4x-2N-2})(1+q^{2N-2\rho-4x-2})},\\
			a_{2,-2}(x)&= \frac{(1-q^{2x-2N})(1-q^{2x-2N+2})}{(1+q^{2\rho+4x-2N})(1+q^{2\rho+4x-2N+2})}.
		\end{align*}
		Note that the coefficients from $a_{j,-2}(x)$ can be obtained from $a_{-j,2}(x)$ by replacing $x$ by $N-x$ and $\rho$ by $-\rho$.

\section{$q$-Krawtchouk polynomials as eigenfunctions of twisted primitive elements in $\U_q$} \label{app:eigenfunctions}
	Let $\pi_k$ be the $N_k+1$ dimensional representation defined on functions $f:\{0,1,\ldots,N_k\}\to\C$ by
	 	\begin{equation} \label{eq:representationapp}
		\begin{split}
			[\pi_k(K)f](n ) &= q^{n-\frac12 N_k} f(n), \\
			[\pi_k(E)f](n) &= q^{u_k(\vec{N})} [n]_q f(n-1), \\
			[\pi_k(F) f](n) & = q^{-u_k(\vec{N})} [N_k-n]_q f(n+1),\\
			[\pi_k(K^{-1}) f](n) &= q^{\frac12N_k-n}f(n).
		\end{split}
	\end{equation}
	and let 
	\begin{align*}
		Y_\rho = q^\half EK + q^{-\half}FK - [\rho]_q(K^2-1).
	\end{align*}
	Similar to the proof of Lemma \ref{lem:etatoxi}, we define for this section only,
	\begin{align}
		k(\eta_k,\xi_k;\rho)= q^{\eta_k u_k(\vec{N})}k(\eta_k,\xi_k;q,N_k,\rho)\label{eq:dualqkrawtchoukwithuapp},
	\end{align}
	Then we have that the $q$-Krawtchouk polynomials $k(\cdot,\xi_k;\rho)\to\C$ are eigenfunctions of $\pi_k(Y_\rho)$ (see e.g. \cite{Koo}), 
	\begin{align}
		[\pi_k(Y_\rho)k(\cdot,\xi_k;\rho)](n) = ([\rho]_q-[\rho+2\xi_k-N_k])k(n,\xi_k;\rho).\label{eq:dualqkrawtchoukeigenfunctionapp}
	\end{align}
	Moreover, we have the following result for the coproduct of $Y_\rho$.
	\begin{proposition}
		Taking $M=2$ gives 
		\begin{align*}
			K_\mathsf{R}(\eta,\xi)= k(\eta_1,\xi_1;\rho+2\xi_2-N_2)k(\eta_2,\xi_2;\rho).
		\end{align*}
		Then this is an eigenfunction of $\pitensortwo(\De(Y_\rho))$,
		\begin{align*}
			[\pitensortwo(\De(Y_\rho))K_\mathsf{R}(\cdot,\xi)](\eta) = ([\rho]_q-[\rho+2(\xi_1+\xi_2)-(N_1+N_2)]_q)K_\mathsf{R}(\eta,\xi).
		\end{align*}
	\end{proposition}
	\begin{proof}
		Indeed, since
		\[
		 	\De(Y_\rho)=K^2 \tensor Y_\rho + Y_\rho \tensor 1,
		\]
		we have
		\begin{align*}
			[\pitensortwo(\De(Y_\rho))K_\mathsf{R}(\cdot,\xi)](\eta)=&\Big[\Big( \pi_1(K^2) \tensor \pi_2(Y_\rho) + \pi_1(Y_\rho) \tensor \pi_2(1)  \Big) K_\mathsf{R}(\cdot,\xi)\Big](\eta)  \\
			=& \Big[\Big( \pi_1(([\rho]_q-[\rho+2\xi_2-N_2]_q)K^2+Y_\rho) \tensor \pi_2(1) \Big) K_\mathsf{R}(\cdot,\xi)\Big](\eta), \\
		\end{align*}
		where we used \eqref{eq:dualqkrawtchoukeigenfunctionapp} for $\pi_2(Y_\rho)$. If we now use the explicit expressions for $Y_\rho$, we see that
		\begin{align*}
			([\rho]_q-[\rho+2\xi_2-N_2]_q)K^2 + Y_\rho  =& q^{\frac12}EK + q^{-\frac12} FK- [\rho+2\xi_2-N_2]_qK^2 + [\rho]_q\\
			=& Y_{\rho+2\xi_2-N_2} +[\rho]_q-[\rho+2\xi_2-N_2]_q.
		\end{align*}
		Therefore, applying \eqref{eq:dualqkrawtchoukeigenfunctionapp} with $\rho$ replaced by $\rho+2\xi_2-N_2$, we obtain that
		\begin{align*}
			\big[\pi_1(([\rho]_q-[\rho+2\xi_2-N_2]_q)K^2+Y_\rho)k(\cdot,\xi_1;\rho+2\xi_2-N_2)\big](\eta_1)
		\end{align*}
		is equal to
		\begin{align*}
			([\rho]_q-[\rho+2(\xi_1+\xi_2)-(N_1+N_2)]_q)k(\eta_1,\xi_1;\rho+2\xi_2-N_2),
		\end{align*}
		Therefore,
		\[
			[\pitensortwo(\De(Y_\rho))K_\mathsf{R}(\cdot,\xi)](\eta) = ([\rho]_q-[\rho+2(\xi_1+\xi_2)-(N_1+N_2)]_q)K_\mathsf{R}(\eta,\xi). \qedhere
		\]
	\end{proof}

\section*{Declarations}
\subsection*{Competing interests}
The authors have no financial or proprietary interests in any material discussed in this article.

\end{document}